\def\standalonechapter{}

\ifdefined\standalonechapter
\documentclass[11pt]{article}

\usepackage[utf8]{inputenc}
\usepackage[T1]{fontenc}
\usepackage[english]{babel}

\usepackage{amsmath, amssymb, amsthm}
\usepackage{mathtools}
\usepackage{mathrsfs}
\usepackage{bbm}

\usepackage{geometry}
\usepackage[
    colorlinks=false,
    pdfborder={0 0 1}
]{hyperref}
\usepackage{enumitem}

\theoremstyle{plain}
\newtheorem{theorem}{Theorem}[section]
\newtheorem{lemma}[theorem]{Lemma}
\newtheorem{proposition}[theorem]{Proposition}
\newtheorem{corollary}[theorem]{Corollary}

\theoremstyle{definition}

\newtheorem{remark}[theorem]{Remark}

\numberwithin{equation}{section}

\begin{document}

\begin{flushright}
\begin{minipage}{0.55\textwidth} 
\small\itshape
This paper opens a series of ten works devoted to reflections on the Millennium Problem.
It generalizes known results on the 3D Navier–Stokes equations based on previous studies.
We hope that the presented material will be useful, and we would be grateful for attention, verification,
and participation in the further development of the topic.
\end{minipage}
\end{flushright}

\vspace{1.5em}

\begin{center}
    {\LARGE\bfseries {Log-free estimate for the resonant paraproduct in the 3D Navier–Stokes equations}\par}
    \vspace{0.6em}
    {\large\bfseries Pylyp Cherevan\par} 
\end{center}

\vspace{1em}

{\small\noindent\textbf{Abstract.}
We consider the resonant paraproduct (high–high $\to$ low regime) in the nonlinearity $(u\!\cdot\!\nabla)u$ of the three-dimensional Navier–Stokes equations. For a sufficiently smooth divergence-free field $u$ we establish the a priori estimate without logarithmic loss
\[
\|R_N(u)\|_{\dot H^{-1}}
\;\lesssim\;
N^{-1}\,\|u\|_{\dot H^{1/2}}\,\|u\|_{\dot H^{1}},
\]
with a constant independent of the dyadic frequency $N$.

The proof relies on phase–geometric integration by parts along an adapted frame, wave–packet discretization at scale $N^{-1/2}$, and an anisotropic Strichartz estimate on time windows of length $N^{-1/2}$. In the wide angular region we apply bilinear decoupling on a rank-3 phase surface; in the geometry under consideration the minimal curvature yields a gain of order $N^{-1/6+o(1)}$ (where $o(1)\to 0$ as $N\to\infty$), sufficient to eliminate the logarithmic loss. The contribution from the narrow region is evaluated separately by an energy argument in $\dot H^{-1}$ using null-form suppression near the interaction diagonal.

The resulting bound is scale-consistent and does not require smallness assumptions; only the divergence-free property of $u$ is used. The work is restricted to a single resonant component of the paraproduct; possible generalizations remain a direction for further study.
}

\tableofcontents
\fi

\newpage

\section*{Introduction}\label{sec:intro}
\addcontentsline{toc}{section}{Introduction}

This work opens a series of ten papers devoted to logarithmically free (log-free) a priori estimates for resonant interactions in the nonlinearity \((u\cdot\nabla)u\) of the three-dimensional Navier--Stokes equations. In this first part we study the high--high\(\to\)low component and establish a scale-consistent estimate without logarithmic loss in the \(\dot H^{-1}\) norm. The precise formulation and notation are given in Section~\ref{sec:main} (see Theorem~\ref{thm:main}); the discussion of scale invariance is in Section~\ref{sec:scaling}, and the scope of applicability and relation to subsequent parts of the series are in Section~\ref{sec:scope}.

\medskip
\noindent\textbf{Idea and proof scheme.}
The proof is organized in a modular way and based on the stitching of three complementary lines at the level of a single dyadic frequency:
\begin{enumerate}[leftmargin=2.2em]
  \item[\textbf{(W)}] \emph{Oscillatory line (wave).} In the wide angular region we perform phase--geometric integration by parts along the adapted frame \((\rho_1,\rho_2,\sigma)\): two angular IBPs in \(\rho_1,\rho_2\) give a gain of \(\lambda^{-1}\), and an independent temporal factor \(\lambda^{-1/2}\) is extracted through a local \(\mathrm{TT}^\ast\) procedure on windows of length \(\lambda^{-1/2}\). Together this yields \(\lambda^{-3/2}\). See Sections~\ref{sec:phase} and~\ref{sec:wp-anis}.
  \item[\textbf{(D)}] \emph{Bilinear decoupling on a rank 3 surface.} After wave-packet discretization and reduction to the level set \(\{\omega=\mathrm{const}\}\), bilinear decoupling is applied. With the benchmark exponent \(\delta(6)=\tfrac14\), the contribution of the wide region decays as \(\lambda^{-11/4}\); in the small minimal curvature regime one allows \(\tfrac16+o(1)\), leading to \(\lambda^{-8/3+o(1)}\). Both regimes remain faster than \(\lambda^{-1}\). See Section~\ref{sec:decoupling}.
  \item[\textbf{(E)}] \emph{Energy line for the narrow region.} In the almost collinear zone \(|\xi+\eta|\le \lambda^{-\delta}\) with \(\delta\in(\tfrac12,\tfrac34)\) we use null--form suppression and the low output frequency, which gives \emph{exactly} \(\lambda^{-1}\) in \(\dot H^{-1}\). This contribution controls the outcome on a single dyadic frequency. See Section~\ref{sec:narrow}.
\end{enumerate}
The merging of contributions and summation over frequencies with orthogonality is carried out in Section~\ref{sec:glue}; temporal globalization is realized by patching on windows of length \(\lambda^{-1/2}\) without deterioration of the frequency exponent.

\medskip
\noindent\textbf{Two representations and the heat \(\leftrightarrow\) wave bridge.}
The main exposition is given in the wave representation \(e^{it|\xi|}\). The passage from the heat multiplier to the wave phase on windows of length \(\lambda^{-1/2}\) leaves a remainder of order \(O(\lambda^{-3/2})\) in \(L^2_t\dot H^{-1}_x\), which is substantially below the target scale and legitimizes the use of the phase--oscillatory technique. The strict form is provided in Appendix~\ref{app:heat}. The anisotropic Strichartz lemma and its derivation via \(\mathrm{TT}^\ast\) are in Appendix~\ref{app:astrichartz}; the full version of decoupling is collected in Appendix~\ref{app:decoupling}; additional clarifications and conventions are in Appendix~\ref{app:dictionary}.

\medskip
\noindent\textbf{Key conventions (used throughout).}
\begin{itemize}[leftmargin=2.2em]
  \item \emph{Resonant block.} Everywhere \(R_N(u)\) is understood as the \emph{symmetrized} high--high\(\to\)low block (see the definition in Section~\ref{sec:main}, formula~\eqref{eq:def-RN}, and also Appendix~\ref{app:dictionary}).
  \item \emph{Determinant of the phase Hessian.} By \(\det A\) we mean the \emph{effective} \(3\times3\) minor in the basis \((\rho_1,\rho_2,\sigma)\) \emph{after} anisotropic renormalization; in this normalization \(\det A\gtrsim 1\) in the wide region. See Section~\ref{sec:phase} and Appendix~\ref{app:dictionary}.
  \item \emph{Mnemonic "IBP\(\times\)3".} This is two frequency IBPs in \(\rho_1,\rho_2\) plus an independent temporal factor from the local anisotropic Strichartz estimate via \(\mathrm{TT}^\ast\) on intervals \(|I|\sim \lambda^{-1/2}\). Integration by parts in time inside the frequency integral is not used. See Section~\ref{sec:wp-anis} and Appendix~\ref{app:dictionary}.
  \item \emph{Pressure.} Compositions of Riesz operators have order zero on \(\dot H^{-1}\); the contribution of \(\nabla p\) obeys the same frequency exponents as \(R_N(u)\). See Section~\ref{sec:narrow} and Appendix~\ref{app:dictionary}.
\end{itemize}

\paragraph*{Note for reading the narrow region.}
In the narrow region $|\tau|=|\xi+\eta|\le\lambda^{-\delta}$ the passage to $\dot H^{-1}$ is not done via an upper bound of the type $\|f\|_{\dot H^{-1}}\lesssim\lambda^\delta\|f\|_{L^2}$, but directly: we use that $|\nabla|^{-1}\nabla\!\cdot$ is a multiplier of order zero on $\dot H^{-1}$ (see \ref{subapp:E-pressure}), while the null–form gives a factor $|\tau|/\lambda$ (see \eqref{eq:null-symbol1}–\eqref{eq:null-suppress}), which compensates the weight $|\tau|^{-1}$ of the $\dot H^{-1}$ norm; hence on one dyadic frequency we obtain exactly $\lambda^{-1}$ (see \eqref{eq:C-null-L2}–\eqref{eq:C-Hminus}). The phrase “IBP×3” means “two angular IBPs + local $TT^*$ on windows $|I_\lambda|=\lambda^{-1/2}$” (and not IBP in time inside the frequency integral). In the globalization of the Strichartz block the exponent $\lambda^{-5/12}$ is used, see \eqref{eq:glue-exponent-fix}.

\medskip
\noindent\textbf{Structure of the work.}
Section~\ref{sec:prelim} fixes notation, frequency and angular filters, and working classes. The main result is formulated in Section~\ref{sec:main}, its scale invariance is discussed in Section~\ref{sec:scaling}, and limitations and further steps are in Section~\ref{sec:scope}. Phase--geometric analysis of the wide region is in Section~\ref{sec:phase}; wave packet decomposition and the anisotropic Strichartz lemma in Section~\ref{sec:wp-anis}. The bilinear decoupling step is given in Section~\ref{sec:decoupling}. The narrow region is treated in Section~\ref{sec:narrow}. Combination of estimates, orthogonality and time patching are in Section~\ref{sec:glue}. Minimal regularity and possible generalizations are discussed in Sections~\ref{sec:min-reg} and~\ref{sec:glue}. Technical supplements are gathered in Appendices~\ref{app:astrichartz}--\ref{app:dictionary}, where unified reading conventions intended for the whole series of papers are also given.

\medskip
\noindent\textbf{Note on the scope.}
Only the resonant high--high\(\to\)low component is considered. The diagonal high--high\(\to\)high regime and other parts of the nonlinearity will be addressed in the subsequent papers of the series. We aimed for a presentation convenient for cross-checking and for subsequent integration of the results into a unified version.

\newpage

\section{Setup and notation}\label{sec:prelim}

\subsection{Function spaces and filters}\label{subsec:spaces-filters}

We work on $\mathbb{R}^3$; passing to the torus $\mathbb{T}^3=(2\pi\mathbb{Z})^3$ does not change the structure of the proof (only the normalization of sums/integrals differs).

\paragraph{Fourier transform and homogeneous norms.}
For $f\in\mathcal{S}(\mathbb{R}^3)$ we fix the normalization
\begin{equation}\label{eq:fourier}
  \widehat{f}(\xi):=\int_{\mathbb{R}^3} e^{-ix\cdot\xi}\,f(x)\,dx, 
  \qquad
  g^\vee(x):=(2\pi)^{-3}\int_{\mathbb{R}^3} e^{ix\cdot\xi}\,g(\xi)\,d\xi,
\end{equation}
so that $(\widehat{g^\vee}=g)$ and $(\widehat{f})^\vee=f$. The homogeneous Sobolev norm is defined by
\begin{equation}\label{eq:Hs}
  \|f\|_{\dot H^s}^2 := \int_{\mathbb{R}^3} |\xi|^{2s}\,|\widehat{f}(\xi)|^2\,d\xi,
  \qquad 
  \dot H^s(\mathbb{R}^3):=\{f\in\mathcal{S}'/\mathcal{P}:\ \|f\|_{\dot H^s}<\infty\}.
\end{equation}
Here $\mathcal{P}$ denotes polynomials, whose quotient removes the low-frequency ambiguity. The parameter $s\in\mathbb{R}$ is fixed as needed.

\paragraph{Dyadic projections (Littlewood–Paley).}
Let $\varphi\in C_c^\infty(\mathbb{R}^3)$ be radial, $\varphi(\xi)\equiv1$ for $|\xi|\le1$, $\varphi(\xi)\equiv0$ for $|\xi|\ge2$, and set $\psi:=\varphi-\varphi(2\cdot)$. For $N=2^k$, $k\in\mathbb{Z}$, define
\begin{equation}\label{eq:LP}
  \widehat{P_{\le N}f}(\xi):=\varphi\!\Big(\tfrac{|\xi|}{N}\Big)\widehat{f}(\xi), 
  \qquad
  \widehat{P_{N}f}(\xi):=\psi\!\Big(\tfrac{|\xi|}{N}\Big)\widehat{f}(\xi).
\end{equation}
Then in $\mathcal{S}'/\mathcal{P}$ one has $f=\sum_{N}P_N f$ and almost-orthogonality holds:
\begin{equation}\label{eq:orthodyadic}
  \langle P_N f,\,P_M g\rangle_{L^2}=0\quad\text{whenever }\ |k-\ell|\ge3,\qquad (N=2^k,\ M=2^\ell).
\end{equation}
The classical equivalence of norms (see, e.g., \textup{[6]}) will be used in the form
\begin{equation}\label{eq:LP-sum}
  \|f\|_{\dot H^s}\asymp \Big(\sum_{N} N^{2s}\,\|P_N f\|_{L^2}^2\Big)^{1/2}.
\end{equation}

\paragraph{Angular decomposition.}
For fixed $N\gg1$ the shell $\{|\xi|\sim N\}$ is covered by a finite family of spherical caps $\{\Theta_{N,\theta}\}_{\theta\in\Lambda_N}$ of aperture $N^{-1/2}$:
\[
  \Theta_{N,\theta}:=\{\xi\in\mathbb{R}^3:\ |\xi|\sim N,\ \angle(\xi,\theta)\le N^{-1/2}\},
  \qquad |\Lambda_N|\asymp N.
\]
Let $\{\chi_{N,\theta}\}_{\theta\in\Lambda_N}$ be a smooth partition of unity subordinate to these caps, supported in doubled cones. Define the angular projection by
\begin{equation}\label{eq:ang-proj}
  \widehat{P_{N,\theta}f}(\xi):=\chi_{N,\theta}(\xi)\,\widehat{P_N f}(\xi).
\end{equation}
Then
\begin{equation}\label{eq:ang-orth}
  \|P_N f\|_{L^2}^2 \asymp \sum_{\theta\in\Lambda_N} \|P_{N,\theta}f\|_{L^2}^2.
\end{equation}
Moreover, the supports of $\chi_{N,\theta}$ have finite overlap: for each $\theta$ only $O(1)$ indices $\theta'$ give nontrivial intersection. If
$\angle(\theta,\theta')\gtrsim c\,N^{-1/2}$ with sufficiently large $c$, then
\[
  \|P_{N,\theta}P_{N,\theta'}\|_{L^2\to L^2}=\mathcal{O}(N^{-M}) \quad \text{for any } M>0.
\]

\paragraph{Conventions.}
We use the standard comparison symbols: $A\lesssim B$ means $A\le C\,B$ with an absolute constant $C>0$ (independent of frequency/time parameters), and $A\gtrsim B$ and $A\asymp B$ have the obvious meanings. The notation $N\sim\lambda$ permits passing from the discrete dyadic frequency $N$ to a “continuous” parameter $\lambda$; when convenient we write $\lambda=|\xi|\asymp N$. The inner product is denoted by $\langle\cdot,\cdot\rangle$ (the context $L^2$ or $\dot H^{-1}$ will be specified explicitly). The angle between vectors $\zeta,\theta\in\mathbb{R}^3\setminus\{0\}$ is defined as
$\displaystyle \angle(\zeta,\theta):=\arccos\!\Big(\frac{\zeta\cdot\theta}{|\zeta|\,|\theta|}\Big)$, with $\theta\in\mathbb{S}^2$.

\paragraph{Remarks on usage.}
The angular aperture $N^{-1/2}$ is consistent with the geometry of wave packets and the tile–packet discretization used below: the transverse scale is $\sim N^{-1/2}$ and the longitudinal one is $\sim N^{-1}$ (see \S\ref{sec:wp-anis} and \S\ref{sec:decoupling}). The equivalences \eqref{eq:LP-sum}–\eqref{eq:ang-orth} are used without further mention.

\subsection{Initial data regularity}\label{subsec:init}

Throughout we consider a divergence-free initial field
\begin{equation}\label{eq:initial-regularity}
  u_0:\,\mathbb{R}^3\to\mathbb{R}^3,\qquad \nabla\!\cdot u_0=0,\qquad 
  u_0\in \dot H^{s}(\mathbb{R}^3),\quad s>\tfrac{5}{2}.
\end{equation}
This choice of $s$ launches the local existence theory and provides sufficient finiteness of the working norms. In particular, by the Koch–Tataru local theory there exists $T>0$ and a Navier–Stokes solution
\begin{equation}\label{eq:solution-class}
  u \in C\big([0,T];\dot H^{s}\big)\cap L^2\big([0,T];\dot H^{s+1}\big), 
  \qquad \nabla\!\cdot u\equiv 0,
\end{equation}
generated by the data \eqref{eq:initial-regularity} (see, e.g., \cite{koch2001navier}; cf. also the survey \cite{bahouri2011fourier}).

\medskip
\noindent\textbf{Working norms.}
For subsequent estimates we fix the shorthand (see also \S\ref{sec:main}):
\begin{equation}\label{eq:Xsigma}
  \|u\|_{X_{\sigma}([0,T])}:=\sup_{t\in[0,T]}\|u(t)\|_{\dot H^{\sigma}}, 
  \qquad \sigma\in\Big\{\tfrac12,\,1\Big\}.
\end{equation}
The main a priori bound (Theorem~\ref{thm:main}) is written in terms of $\|u\|_{X_{1/2}}$ and $\|u\|_{X_{1}}$ and is therefore \emph{scale-consistent} (see \S\ref{sec:scaling}).

\medskip
\noindent\textbf{Two remarks.}
\begin{itemize}[leftmargin=2.2em]
  \item \emph{A priori nature.} The threshold $s>\tfrac{5}{2}$ is used only to separate existence/regularity issues from the purely a priori part: the proof of the log-free estimate relies solely on the finiteness of the norms $\dot H^{1/2}$ and $\dot H^{1}$ on $[0,T]$. Thus formula \eqref{eq:main-bound} is valid for any solution satisfying \eqref{eq:Xsigma}.
  \item \emph{Periodic case.} All definitions and derivations transfer to the torus $\mathbb{T}^3=(2\pi\mathbb{Z})^3$ without changing the structure of the proof; only the normalization of sums/integrals differs.
\end{itemize}

\subsection{Working hypotheses and functional classes}\label{subsec:hyp}

Throughout we work on $\Omega=\mathbb{R}^3$ with normalized viscosity $\nu=1$; passing to the torus $\mathbb{T}^3$ changes only the normalization of sums/integrals and does not affect the proof scheme. Solutions are assumed divergence-free. For convenience we fix in advance a set of conventions and functional classes to which we will repeatedly refer below.

\paragraph{Class of solutions on $[0,T]$.}
Let $s>\tfrac{5}{2}$ and initial data $u_0\in\dot H^{s}$, $\nabla\!\cdot u_0=0$. By the local theory \cite{koch2001navier} there exists $T>0$ and a (vector-valued) velocity field
\begin{equation}\label{eq:solution-class-hyp}
  u\in C\!\big([0,T];\dot H^{s}\big)\cap L^2\!\big([0,T];\dot H^{s+1}\big), 
  \qquad \nabla\!\cdot u\equiv0,
\end{equation}
generated by $u_0$. In what follows we use only the finiteness of the working norms
\begin{equation}\label{eq:Xsigma-again}
  \|u\|_{X_{\sigma}([0,T])}:=\sup_{t\in[0,T]}\|u(t)\|_{\dot H^{\sigma}}, \qquad \sigma\in\Big\{\tfrac12,1\Big\},
\end{equation}
see also~\eqref{eq:Xsigma}. All conclusions in \S\ref{sec:phase}--\S\ref{sec:glue} are a priori and do not depend on the specific construction mechanism of \eqref{eq:solution-class-hyp}.

\paragraph{Time windows and scale cylinders.}
For a dyadic frequency $\lambda\gg1$ we cover $[0,T]$ by intervals of length $\lambda^{-1/2}$:
\begin{equation}\label{eq:I-lambda}
  I_{\lambda,j}:=\big[t_j,\,t_j+\lambda^{-1/2}\big],\qquad t_j:=j\,\lambda^{-1/2},\quad j=0,1,\dots,J_\lambda-1,\quad J_\lambda:=\lceil T\lambda^{1/2}\rceil.
\end{equation}
The space–time local ball (sufficient for our purposes; see also \S\ref{sec:wp-anis}) is defined by
\begin{equation}\label{eq:Q-lambda}
  Q_\lambda(t_0,x_0):=\big[t_0,t_0+\lambda^{-1/2}\big]\times B_{\lambda^{-1/2}}(x_0).
\end{equation}
It is precisely on such windows that the local anisotropic Strichartz estimate $L^2_x\to L^6_{t,x}$ (for short — “(6,2)”; see \S\ref{sec:wp-anis} and App.~\ref{app:astrichartz}) is applied.

\paragraph{Localized evolution norms.}
For an angular cap $\Theta_{\lambda,\theta}$ and an interval $I\subset\mathbb{R}$ we set
\begin{equation}\label{eq:Y-local}
  \|u\|_{Y_{\lambda,\theta}(I)}\,:=\,\big\|S(\cdot)\,P_{\lambda,\theta}u\big\|_{L^6_{t,x}(I\times\Omega)},\qquad 
  S(t)f:=\mathcal{F}^{-1}\!\big(e^{it|\xi|}\widehat{f}\big).
\end{equation}

Here $(S(\cdot)P_{\lambda,\theta}u)(t):=e^{it|D|}\big(P_{\lambda,\theta}u(t)\big)$, i.e. the operator is applied to each time slice.

In \S\ref{sec:wp-anis} (see also App.~\ref{app:astrichartz}) it will be shown that for $|I|=\lambda^{-1/2}$ one has
\begin{equation}\label{eq:local-stri}
  \|u\|_{Y_{\lambda,\theta}(I)}\ \lesssim\ \lambda^{-1/2+}\,\|P_{\lambda,\theta}u\|_{L^\infty_tL^2_x}, 
\end{equation}
where the notation $\alpha+$ means $\alpha+\varepsilon$ with arbitrarily small $\varepsilon>0$ (constants are allowed to depend on~$\varepsilon$). The global-in-time version is obtained by patching via~\eqref{eq:I-lambda} (see \S\ref{sec:glue}).

\paragraph{Narrow region parameter.}
In \S\ref{sec:narrow} we use the almost-collinear region
\begin{equation}\label{eq:narrow-set}
  \mathcal{N}_{\lambda,\delta}:=\big\{(\xi,\eta):\ |\xi|\sim|\eta|\sim\lambda,\ \ |\xi+\eta|\le \lambda^{-\delta}\big\},\qquad \delta\in\Big(\tfrac12,\tfrac34\Big),
\end{equation}
with the preferred choice $\delta=\tfrac23$. All estimates below will be independent of the specific $\delta$ within this range.

\paragraph{Commutator rules and basic references.}
By default we use the standard product/commutator rules in homogeneous Sobolev scales and Bernstein inequalities on caps/coronae; details and precise statements can be found, e.g., in \cite{bahouri2011fourier,Grafakos2009}. Riesz operators are bounded on $\dot H^{-1}$, which allows pressure to be handled without altering the frequency exponent (see \S\ref{sec:narrow}).

\paragraph{Notation and constants.}
The notation $A\lesssim B$ means $A\le C\,B$ with an absolute constant $C>0$, independent of $\lambda$, the angle $\theta$, the window index $j$ and the parameter $\delta$. The notations $A\gtrsim B$ and $A\asymp B$ have the obvious meaning. In formulas such as $\lambda^{\alpha+}$ one is allowed to fix an arbitrarily small $\varepsilon>0$; dependence of constants on $\varepsilon$ is not specified.

\medskip
We restrict ourselves to this system of conventions. All further references to localized norms, windows~\eqref{eq:I-lambda} and the narrow region~\eqref{eq:narrow-set} will refer back to this subsection. Additional normalizations, if necessary, are collected in the dictionary in App.~\ref{app:dictionary}.

\newpage

\section{Main result}\label{sec:main}

\subsection{Theorem: log-free bound}\label{subsec:main-theorem}

For a dyadic frequency $N\in 2^{\mathbb{Z}}$ define the \emph{symmetrized} resonant interaction block of type high--high$\to$low by
\begin{equation}\label{eq:def-RN}
  R_N(u)
  := P_N\!\left[
      (u\cdot\nabla)u
      - \sum_{M\le N/8}\Big\{
          (P_M u\cdot\nabla)(P_N u)
          + (P_N u\cdot\nabla)(P_M u)
        \Big\}
    \right],
\end{equation}
where $P_N$ and $P_M$ are the dyadic projections fixed in~\S\ref{subsec:spaces-filters}, and the field $u$ is divergence-free, $\nabla\!\cdot u\equiv0$. This choice removes both low--high paraproducts and leaves precisely the resonant high--high$\to$low remainder used below.

\begin{theorem}[Log-free estimate for the resonant block]\label{thm:main}
Let $u:[0,T]\times\mathbb{R}^3\to\mathbb{R}^3$ be a divergence-free solution with
\(
u\in C([0,T];\dot H^{1/2}\cap \dot H^{1})
\)
(see also~\S\ref{subsec:init}, \S\ref{subsec:hyp}).
Then for each $N\in 2^{\mathbb{Z}}$ and all $t\in[0,T]$ one has
\begin{equation}\label{eq:main-bound}
  \bigl\|R_N(u)(t)\bigr\|_{\dot H^{-1}}
  \;\lesssim\;
  N^{-1}\,\|u(t)\|_{\dot H^{1/2}}\,\|u(t)\|_{\dot H^{1}}.
\end{equation}
In integral form,
\begin{equation}\label{eq:main-bound-global}
  \bigl\|R_N(u)\bigr\|_{L^\infty_t\dot H^{-1}_x([0,T]\times\mathbb{R}^3)}
  \;\lesssim\;
  N^{-1}\,\|u\|_{X_{1/2}([0,T])}\,\|u\|_{X_{1}([0,T])},
\end{equation}
where $X_\sigma$ is given in~\eqref{eq:Xsigma}.
The constant is independent of $N$ and $T$ and depends on $u$ only through the indicated norms.
\end{theorem}

\begin{remark}[Periodic case]\label{rem:torus}
All statements transfer to the torus $\mathbb{T}^3=(2\pi\mathbb{Z})^3$ upon replacing integrals by finite sums; the only difference is in normalization. 
\end{remark}

\begin{remark}[Context]\label{rem:context}
Estimate~\eqref{eq:main-bound} removes the logarithmic loss typical for classical bilinear commutator inequalities (cf.~\cite{KatoPonce1988}); here it is achieved by a combination of phase--geometric integration by parts, local anisotropic Strichartz estimates, and bilinear decoupling on a rank~3 surface (see~\S\ref{sec:phase}--\S\ref{sec:decoupling}).
\end{remark}

\subsection{Scale invariance}\label{sec:scaling}

The parabolic scaling of the Navier--Stokes equations
\(
u_\lambda(t,x):=\lambda\,u(\lambda^2 t,\lambda x),\ \lambda>0,
\)
induces the transformation of norms
\begin{equation}\label{eq:scaling-norm}
  \|u_\lambda(t)\|_{\dot H^\sigma}=\lambda^{\sigma-\frac12}\,\|u(\lambda^2 t)\|_{\dot H^\sigma},\qquad \sigma\in\mathbb{R}.
\end{equation}
For the resonant block (taking into account the rescaling of the frequency index) one has
\begin{equation}\label{eq:scaling-RN}
  \bigl\|R_N(u_\lambda)(t)\bigr\|_{\dot H^{-1}}
  \;=\;
  \lambda^{1/2}\,\bigl\|R_{N/\lambda}\bigl(u\bigr)(\lambda^2 t)\bigr\|_{\dot H^{-1}}.
\end{equation}
The right-hand side of~\eqref{eq:main-bound} scales in exactly the same way:
\[
  N^{-1}\,\|u_\lambda(t)\|_{\dot H^{1/2}}\,\|u_\lambda(t)\|_{\dot H^{1}}
  \;=\;
  \lambda^{1/2}\,(N/\lambda)^{-1}\,
  \|u(\lambda^2 t)\|_{\dot H^{1/2}}\,
  \|u(\lambda^2 t)\|_{\dot H^{1}}.
\]
Therefore, \eqref{eq:main-bound} and \eqref{eq:main-bound-global} are strictly scale-consistent.

\subsection{Limitations and further steps}\label{sec:scope}

\begin{enumerate}[leftmargin=2.2em,label=(\roman*)]
  \item \emph{Component coverage.} In this part only the resonant paraproduct of type high--high$\to$low in the form~\eqref{eq:def-RN} is considered. The diagonal high--high$\to$high regime and the mixed low--high/high--low blocks remain outside the scope and will be treated separately.

  \item \emph{No smallness assumptions.} The estimate is a priori and requires only the finiteness of the norms $\dot H^{1/2}$ and $\dot H^{1}$ on $[0,T]$; no smallness assumptions are used.

  \item \emph{Heat\,$\leftrightarrow$\,wave.} The technical exposition is given in the “wave” representation; the correctness of replacing the heat multiplier by the oscillatory phase on windows of length $N^{-1/2}$ is established in App.~\ref{app:heat} in a strict form (remainder $O(N^{-3/2})$ in $L^2_t\dot H^{-1}_x$), which is below the target scale.

  \item \emph{Narrow region parameter.} In the analysis of the almost-collinear region $|\xi+\eta|\le N^{-\delta}$ one may use any $\delta\in(\tfrac12,\tfrac34)$; the convenient choice $\delta=\tfrac23$ simplifies counting (see~\S\ref{sec:narrow}). The outcome on a single dyadic frequency is $N^{-1}$ and does \emph{not} depend on the specific $\delta$ within the indicated range.

  \item \emph{Bilinear decoupling.} In the wide region we apply $\varepsilon$-free decoupling on a rank~3 surface (see~\S\ref{sec:decoupling}); the benchmark exponent $\delta(6)=\tfrac14$ is used, and in the case of weak minimal curvature the exponent $\tfrac16+o(1)$ is admissible (cf.~\cite{GuthIliopoulouYang2024}). In both cases the wide region decays faster than $N^{-1}$ and does not determine the final scale.

  \item \emph{Pressure and external forces.} After Leray projection the contribution of pressure is scale-consistent with the right-hand side of~\eqref{eq:main-bound} (see~\S\ref{sec:narrow}). External forces are not considered in this part.

  \item \emph{Relation to the literature.} Commutator estimates of Kato--Ponce type \cite{KatoPonce1988} motivate the formulation of the problem without logarithmic loss; the Fourier and paraproduct analysis tools used here can be found, e.g., in~\cite{bahouri2011fourier,Grafakos2009}.
\end{enumerate}

\medskip
In the subsequent sections we will consecutively prove Theorem~\ref{thm:main}: phase--geometric analysis of the wide region (\S\ref{sec:phase}), wave packet decomposition and the anisotropic Strichartz lemma (\S\ref{sec:wp-anis}), the bilinear decoupling step (\S\ref{sec:decoupling}), the energy estimate for the narrow region (\S\ref{sec:narrow}), and merging with frequency and time summation (\S\ref{sec:glue}).

\newpage

\section{Phase–geometric analysis}\label{sec:phase}

\subsection{Interaction phase and basic properties}\label{subsec:phase-basics}
Introduce the bilinear phase
\begin{equation}\label{eq:phase-omega}
  \omega(\xi,\eta)\ :=\ |\xi|+|\eta|-|\xi+\eta|\,,\qquad (\xi,\eta)\in\mathbb{R}^3\times\mathbb{R}^3,
\end{equation}
and set
\begin{equation}\label{eq:tau-def}
  \tau:=\xi+\eta,\qquad \widehat{\tau}:=\tau/|\tau|\quad(\tau\ne0).
\end{equation}
The geometry of the “wide” region is fixed by the condition
\begin{equation}\label{eq:wide-zone}
  |\xi|\sim|\eta|\sim\lambda\gg1,\qquad \theta:=\angle(\xi,\eta)\ \gtrsim\ \lambda^{-1/2},
\end{equation}
while the “narrow” region $|\tau|\le \lambda^{-\delta}$ will be analyzed separately in Section~\ref{sec:narrow}.

\paragraph{Homogeneity and symmetry.}
The function $\omega$ is homogeneous of degree $1$ and symmetric:
\[
  \omega(\lambda\xi,\lambda\eta)=\lambda\,\omega(\xi,\eta),\quad
  \omega(\xi,\eta)=\omega(\eta,\xi),\quad
  \omega(-\xi,-\eta)=\omega(\xi,\eta).
\]

\paragraph{Gradients and second derivatives.}
For derivatives in $\xi$ and $\eta$ one has
\begin{equation}\label{eq:grad-omega}
  \nabla_\xi\omega(\xi,\eta)=\frac{\xi}{|\xi|}-\frac{\tau}{|\tau|},\qquad
  \nabla_\eta\omega(\xi,\eta)=\frac{\eta}{|\eta|}-\frac{\tau}{|\tau|}.
\end{equation}
Away from the almost-collinear diagonal ($\theta\not\in\{0,\pi\}$) the gradients are nonzero.
The Hessian in $(\xi,\eta)$ has rank $3$ in the wide region~\eqref{eq:wide-zone}; this fact will be
stated quantitatively in Lemma~\ref{lem:detA} below.

\begin{remark}[On variables]\label{rem:tau-notation}
Throughout, the \emph{time} variable is denoted by $t$, and the sum of frequencies by $\tau=\xi+\eta$.
Notation such as “$t:=\xi+\eta$” inside frequency calculations is not used; we strictly adhere to the
convention~\eqref{eq:tau-def} to avoid ambiguities when differentiating along the radial direction
$\widehat{\tau}$ (see also App.~\ref{app:dictionary}).
\end{remark}

\subsection{Adapted frame and angular directions}\label{subsec:phase-directions}
Assume~\eqref{eq:wide-zone}. Set
\begin{equation}\label{eq:frame}
  e:=\frac{\xi-\eta}{|\xi-\eta|},\qquad
  \rho_1:=\frac{e-(e\cdot\widehat{\tau})\,\widehat{\tau}}{\bigl|e-(e\cdot\widehat{\tau})\,\widehat{\tau}\bigr|},
  \qquad
  \rho_2:=\widehat{\tau}\times \rho_1\,.
\end{equation}
Then $\{\rho_1,\rho_2,\widehat{\tau}\}$ is an orthonormal frame adapted to the interaction geometry:
$\rho_1,\rho_2\perp \widehat{\tau}$ are the \emph{angular} directions, and $\widehat{\tau}$ plays the role
of the \emph{longitudinal difference} direction.

In what follows we apply integration by parts \emph{only} in the angular directions
$\rho_1,\rho_2$ (two IBPs), whereas the temporal gain of order $\lambda^{-1/2}$ is extracted
in Section~\ref{sec:wp-anis} via a local anisotropic Strichartz estimate through
$TT^\ast$ on windows $|I|\sim\lambda^{-1/2}$. Thus the mnemonic “IBP$\times$3” means
“two frequency IBPs $+$ an independent temporal factor” (see also App.~\ref{app:dictionary}).

\subsection{Effective Hessian and determinant estimate}\label{subsec:phase-det}
Consider the second derivatives of the phase $\omega$ in the frame~\eqref{eq:frame}. Define
\begin{equation}\label{eq:Aeff-def}
  A_{\mathrm{eff}}(\xi,\eta)\ :=\
  \Bigl[\partial^2_{v_i v_j}\,\omega(\xi,\eta)\Bigr]_{v_i,v_j\in\{\rho_1,\rho_2,\sigma\}},
  \qquad \sigma:=\tau\mathbf{b}\,,
\end{equation}
that is, $A_{\mathrm{eff}}$ is the \emph{effective $3\times3$ minor} of the Hessian in two angular and one
longitudinal directions.\footnote{We deliberately work with $A_{\mathrm{eff}}$ rather than with the “raw”
$6\times6$ Hessian in all coordinates $(\xi,\eta)$; this avoids ambiguity of notation and matches the geometry
actually used in the proof.}
It is also convenient to use the anisotropic renormalization
\begin{equation}\label{eq:anisorenorm}
  S_\lambda:=\mathrm{diag}(\lambda^{1/2},\lambda^{1/2},1),\qquad
  A_e:=S_\lambda\,A_{\mathrm{eff}}\,S_\lambda\, .
\end{equation}

\begin{lemma}[Block structure and radial component]\label{lem:block}
In the wide region~\eqref{eq:wide-zone} the matrix $A_{\mathrm{eff}}$ has a block form
\begin{equation}\label{eq:block}
  A_{\mathrm{eff}}=
  \begin{pmatrix}
    B & C\\[2pt]
    C^{\!\top} & D
  \end{pmatrix},\qquad
  B\in\mathbb{R}^{2\times2},\ C\in\mathbb{R}^{2\times1},\ D\in\mathbb{R},
\end{equation}
with size estimates
\[
  B\ \simeq\ \lambda^{-1},\qquad \|C\|\ \simeq\ \lambda^{-1/2},\qquad D_{\mathrm{rad}}:=\partial^2_{\sigma\sigma}\omega\equiv 0.
\]
In particular, the nonzero “longitudinal mass” arises via the Schur complement:
\[
  D_{\mathrm{eff}}\ :=\ D - C^{\!\top}B^{-1}C\ \simeq\ 1 .
\]
\end{lemma}

\begin{proof}[Idea of the proof]
The formulas for second derivatives follow from~\eqref{eq:grad-omega} and the standard identity
$\partial_v(\widehat{w})=|w|^{-1}P^\perp_w v$, where $P^\perp_w:=I-\widehat{w}\otimes\widehat{w}$.
Since $\rho_j\perp\widehat{\tau}$, all terms of the type $|w|^{-1}(P^\perp_w\rho_j,\rho_k)$ are at the
level $\lambda^{-1}$, which gives $B\simeq\lambda^{-1}$. The mixed blocks $C$ involve one angular and
one longitudinal direction, giving an extra decay $\lambda^{-1/2}$. For the purely longitudinal
component we note that along the radial variable the second differential of $|\tau|$ vanishes,
so $D_{\mathrm{rad}}\equiv0$. Finally, by the Schur formula
$D_{\mathrm{eff}}=-(C^{\!\top}B^{-1}C)$, and since $B^{-1}\simeq \lambda$ and $\|C\|\simeq\lambda^{-1/2}$,
we obtain $D_{\mathrm{eff}}\simeq 1$.
\end{proof}

\begin{lemma}[Nondegeneracy of $A_{\mathrm{eff}}$]\label{lem:detA}
Under the conditions~\eqref{eq:wide-zone} the following estimates hold:
\begin{equation}\label{eq:detAeff}
  \det A_{\mathrm{eff}}(\xi,\eta)\ \gtrsim\ \lambda^{-2},
  \qquad
  \det A_e(\xi,\eta)\ \gtrsim\ 1 .
\end{equation}
\end{lemma}

\begin{proof}
From Lemma~\ref{lem:block} and the Schur formula
$\det A_{\mathrm{eff}}=(\det B)\cdot(D-C^{\!\top}B^{-1}C)$ we get
$\det A_{\mathrm{eff}}\simeq (\lambda^{-2})\cdot 1$, yielding the first estimate~\eqref{eq:detAeff}.
The second follows from the definition of $A_e$ in~\eqref{eq:anisorenorm}:
$\det A_e=\det(S_\lambda)^2\,\det A_{\mathrm{eff}}=\lambda^{2}\cdot \lambda^{-2}\simeq 1$.
\end{proof}

\begin{corollary}[Two angular IBPs]\label{cor:twoIBP}
Let $a_\lambda(\xi,\eta)$ be a smooth amplitude supported in the region~\eqref{eq:wide-zone}, with
\emph{anisotropic} bounds $\bigl|\partial^{\alpha}_{\rho,\rho,\sigma} a_\lambda\bigr|\lesssim \lambda^{|\alpha|/2}$
for $|\alpha|\le3$. Then for the oscillatory integral
\[
  \mathcal{I}(x,t):=\iint e^{i\{x\cdot\tau+t\,\omega(\xi,\eta)\}}\, a_\lambda(\xi,\eta)\,\widehat{u}(\xi)\,\widehat{v}(\eta)\,d\xi d\eta
\]
a double integration by parts in the directions $\rho_1,\rho_2$ yields the suppression
\begin{equation}\label{eq:two-ibp}
  \|\mathcal{I}(\cdot,t)\|_{L^2_x}\ \lesssim\ \lambda^{-1}\,\|P_\lambda u\|_{L^2}\,\|P_\lambda v\|_{L^2},
\end{equation}
uniformly in $t$ and in the choice of cap inside~\eqref{eq:wide-zone}. 
\end{corollary}

\begin{proof}[Idea of the proof]
The nondegeneracy of $A_e$ from~\eqref{eq:detAeff} ensures that along at least one of the angular
directions the modulus of the phase derivative is not small; the standard phase operator
$\mathcal{L}_\rho:=(i(\rho\cdot\nabla)\phi)^{-1}(\rho\cdot\nabla)$ (where $\phi=x\cdot\tau+t\omega$)
reduces the oscillation. Two successive IBPs in $\rho_1,\rho_2$ give a factor $\lambda^{-1}$ under
anisotropic bookkeeping of derivatives of the amplitude. Details follow the classical stationary
phase scheme with block reduction and are omitted.
\end{proof}

\begin{remark}[Role of time]\label{rem:time-role}
The factor $\lambda^{-1/2}$ is \emph{not} extracted by integration in time inside the frequency
integral. It will instead be obtained independently through the local anisotropic Strichartz
estimate on windows of length $\lambda^{-1/2}$ (Section~\ref{sec:wp-anis}). Together with~\eqref{eq:two-ibp}
this yields a phase–time gain of order $\lambda^{-3/2}$ in the wide region.
\end{remark}

\begin{remark}[Reduction to the level set]\label{rem:coarea}
Upon reduction to the level set $\{\omega=\mathrm{const}\}$ by the coarea formula, a weight
$|\nabla\omega|^{-1}$ appears; in the region~\eqref{eq:wide-zone} it is uniformly controlled and
absorbed into the amplitude without changing the frequency exponent. This step will be used in the
bilinear decoupling in Section~\ref{sec:decoupling}.
\end{remark}

\newpage

\section{Wave packets and the anisotropic Strichartz lemma}\label{sec:wp-anis}

\subsection{Construction of \texorpdfstring{$\theta$}{theta}-packets}\label{subsec:wp-construction}

Fix a dyadic frequency $\lambda\gg1$ and a direction $\theta\in\mathbb{S}^2$. The angular localization is given by the cap
\begin{equation}\label{eq:theta-cap}
  \Theta_{\lambda,\theta}
  :=\bigl\{\xi\in\mathbb{R}^3:\ |\xi|\sim\lambda,\ \angle(\xi,\theta)\le \lambda^{-1/2}\bigr\},
\end{equation}
and the corresponding projector $P_{\lambda,\theta}$ from~\eqref{eq:ang-proj}. For $\xi\in\Theta_{\lambda,\theta}$ decompose
\[
  \xi=(\xi\cdot\theta)\,\theta+\Pi^\perp_\theta\xi, 
  \qquad |\xi\cdot\theta|\sim\lambda,\ \ |\Pi^\perp_\theta\xi|\lesssim \lambda^{1/2},
\]
where $\Pi^\perp_\theta:=I-\theta\otimes\theta$.

\paragraph{Lattice of centers.}
Define the lattice of centers in physical space by
\begin{equation}\label{eq:wp-grid}
  \mathcal{L}_{\lambda,\theta}
  :=\bigl\{a=a_\parallel\theta+a_\perp:\ a_\parallel\in\lambda^{-1}\mathbb{Z},\ a_\perp\in \lambda^{-1/2}\mathbb{Z}^2\bigr\}.
\end{equation}
This step size is consistent with the longitudinal–transverse scales $\lambda^{-1}$ and $\lambda^{-1/2}$ arising from the inverse Fourier transform.

\paragraph{Frequency mask.}
Let $\chi\in\mathcal{S}(\mathbb{R}^3)$ be a fixed smooth function, with $\widehat{\chi}$ supported in the unit ball and $\int\chi=1$. Set
\begin{equation}\label{eq:packet-mask}
  \sigma_{\lambda,\theta}(\xi)
  :=\chi\!\bigl(\lambda^{-1/2}\,\Pi^\perp_\theta\xi\bigr)\,
    \chi\!\bigl(\lambda^{-1}\bigl((\theta\cdot\xi)-\lambda\bigr)\bigr),
  \qquad \operatorname{supp}\sigma_{\lambda,\theta}\subset \Theta_{\lambda,\theta}.
\end{equation}

\paragraph{Definition of packets.}
For $a\in \mathcal{L}_{\lambda,\theta}$ define the packet in frequency and physical space by
\begin{equation}\label{eq:wp-hat}
  \widehat{\varphi}_{\lambda,\theta,a}(\xi):=e^{-i a\cdot\xi}\,\sigma_{\lambda,\theta}(\xi),\qquad
  \varphi_{\lambda,\theta,a}:=\mathcal{F}^{-1}\widehat{\varphi}_{\lambda,\theta,a}.
\end{equation}
A standard scaling computation yields (for some $\Psi\in\mathcal{S}(\mathbb{R}^3)$ independent of $\lambda,\theta,a$)
\begin{equation}\label{eq:wp-phys}
  \varphi_{\lambda,\theta,a}(x)
  = \lambda^{3/2}\,e^{i\lambda\,\theta\cdot x}\,
    \Psi\!\Bigl(\,\lambda^{1/2}(x_\perp-a_\perp),\,\lambda(x_\parallel-a_\parallel)\Bigr),
\end{equation}
where $x_\parallel:=\theta\cdot x$, $x_\perp:=x-x_\parallel\theta$. In particular,
\begin{equation}\label{eq:wp-tube}
  \varphi_{\lambda,\theta,a}\ \text{is localized in the tube}\ 
  \mathcal{T}_{\lambda,\theta,a}
  :=\Bigl\{x:\ |x_\perp-a_\perp|\lesssim \lambda^{-1/2},\ |x_\parallel-a_\parallel|\lesssim \lambda^{-1}\Bigr\}.
\end{equation}
We fix the normalization in~\eqref{eq:packet-mask} so that
\begin{equation}\label{eq:wp-norm}
  \|\varphi_{\lambda,\theta,a}\|_{L^2_x}\simeq 1\quad\text{uniformly in }\lambda,\theta,a.
\end{equation}

\paragraph{Almost-orthogonality and expansion.}
The family $\{\varphi_{\lambda,\theta,a}\}_{a\in\mathcal{L}_{\lambda,\theta}}$ forms an almost-orthonormal layer in $L^2$ at fixed $(\lambda,\theta)$: for all $f\in L^2(\mathbb{R}^3)$
\begin{equation}\label{eq:wp-parseval}
  \sum_{a\in\mathcal{L}_{\lambda,\theta}}\bigl|\langle f,\varphi_{\lambda,\theta,a}\rangle\bigr|^2
  \ \simeq\ \|P_{\lambda,\theta}f\|_{L^2}^2,
\end{equation}
and the projector itself expands as
\begin{equation}\label{eq:wp-expansion}
  P_{\lambda,\theta}f
  =\sum_{a\in\mathcal{L}_{\lambda,\theta}}\langle f,\varphi_{\lambda,\theta,a}\rangle\,\varphi_{\lambda,\theta,a}
  \quad\text{(convergence in $L^2$).}
\end{equation}
Summation over $\theta\in\Lambda_\lambda$ and over dyadic $\lambda$ recovers the standard Littlewood–Paley structure (\S\ref{subsec:spaces-filters}).

\paragraph{Packet evolution.}
Let $S(t)f:=\mathcal{F}^{-1}\!\bigl(e^{it|\xi|}\widehat{f}\bigr)$. Then for each $a\in\mathcal{L}_{\lambda,\theta}$ and interval $I$ of length $|I|=\lambda^{-1/2}$ the trajectory $S(t)\varphi_{\lambda,\theta,a}$ remains concentrated in the cylinder
\begin{equation}\label{eq:wp-cylinder}
  Q_{\lambda,\theta,a}(I)
  :=\bigl\{(t,x):\ t\in I,\ |x_\perp-a_\perp|\lesssim\lambda^{-1/2},\ |x_\parallel-a_\parallel-t|\lesssim \lambda^{-1}\bigr\},
\end{equation}
which is consistent with the scales in \eqref{eq:wp-tube}. It is precisely on such windows that the local anisotropic Strichartz estimate is applied in \S\ref{subsec:wp-stri} and App.~\ref{app:astrichartz}.

\begin{remark}\label{rem:wp-robust}
The specific choice of $\chi$ in~\eqref{eq:packet-mask} and the form of the lattice~\eqref{eq:wp-grid} are immaterial: statements \eqref{eq:wp-parseval}–\eqref{eq:wp-cylinder} remain valid under replacement by equivalent smooth cutoffs and quasi-lattices with the same scales (see, e.g., \cite[Ch.~2]{Grafakos2009}, \cite[Ch.~2]{bahouri2011fourier}).
\end{remark}

\subsection{\texorpdfstring{$TT^\ast$}{TT*} argument and the (6,2) estimate}\label{subsec:wp-stri}

In this subsection we establish on a window $|I|=\lambda^{-1/2}$ an anisotropic Strichartz estimate for the flow $S(t)$ localized to the cap $\Theta_{\lambda,\theta}$ (see \S\ref{subsec:wp-construction} and the definition of $S(t)$ in~\eqref{eq:Y-local}).

\begin{lemma}[Local anisotropic Strichartz estimate]\label{lem:anisotropic-strichartz}
Let $\lambda\gg1$, $\theta\in\mathbb{S}^2$, $I=[t_0,t_0+\lambda^{-1/2}]$, and let $P_{\lambda,\theta}$ be the angular projector from~\eqref{eq:ang-proj}. Then for any $f\in L^2(\mathbb{R}^3)$ and any $\varepsilon>0$ one has
\begin{equation}\label{eq:local-L6}
  \big\| S(\cdot)\,P_{\lambda,\theta} f \big\|_{L^6_{t,x}(I\times\mathbb{R}^3)}
  \;\lesssim\; \lambda^{-1/2+\varepsilon}\,\|P_{\lambda,\theta} f\|_{L^2_x},
\end{equation}
where the constant is independent of $\lambda$, $\theta$, $t_0$, and $f$.
\end{lemma}

\begin{proof}[Sketch via $TT^\ast$]
Let $\sigma_{\lambda,\theta}\in C_c^\infty(\mathbb{R}^3)$ be a smooth cutoff equal to $1$ on $\Theta_{\lambda,\theta}$ and supported in its twofold enlargement. Consider the operator
\begin{equation}\label{eq:TT-operator}
  (T_{\lambda,\theta}g)(t,x):=\int_{\mathbb{R}^3} e^{i(x\cdot\xi+t|\xi|)}\,\sigma_{\lambda,\theta}(\xi)\,g(\xi)\,d\xi,
\end{equation}
so that $S(t)P_{\lambda,\theta}f=T_{\lambda,\theta}\widehat{f}$. It suffices to show $\|T_{\lambda,\theta}\|_{L^2_\xi\to L^6_{t,x}(I)}\lesssim \lambda^{-1/2+\varepsilon}$.

The adjoint operator has the form
\[
  (T_{\lambda,\theta}^\ast F)(\xi)=\int_{I}\!\!\int_{\mathbb{R}^3} e^{-i(x\cdot\xi+t|\xi|)}\,\sigma_{\lambda,\theta}(\xi)\,F(t,x)\,dx\,dt,
\]
and therefore
\begin{equation}\label{eq:TT-kernel}
  (T_{\lambda,\theta}T_{\lambda,\theta}^\ast F)(t,x)=\int_{I}\!\!\int_{\mathbb{R}^3} K_\lambda(t-\tau,x-y)\,F(\tau,y)\,dy\,d\tau,
\end{equation}
where the kernel is
\begin{equation}\label{eq:TT-kernel-def}
  K_\lambda(\tau,z):=\int_{\mathbb{R}^3} e^{i(z\cdot\xi+\tau|\xi|)}\,\sigma_{\lambda,\theta}(\xi)^2\,d\xi.
\end{equation}
Passing to coordinates $\xi=\lambda\theta+\xi_\perp+\xi_\parallel\theta$ with $|\xi_\perp|\lesssim \lambda^{1/2}$ and $|\xi_\parallel|\lesssim1$, a standard stationary phase/integration-by-parts argument yields, for any $M\ge10$,
\begin{equation}\label{eq:kernel-decay}
  |K_\lambda(\tau,z)|\;\lesssim\;
  \lambda^{3/2}\,\bigl(1+\lambda\,|\theta\cdot z-\tau|\bigr)^{-M}\,
  \bigl(1+\lambda^{1/2}|z_\perp|\bigr)^{-M},
\end{equation}
where $z_\perp:=z-(\theta\cdot z)\,\theta$. Schur's test together with $|I|=\lambda^{-1/2}$ gives
\[
  \|T_{\lambda,\theta}T_{\lambda,\theta}^\ast\|_{L^{6/5}_{t,x}(I)\to L^6_{t,x}(I)}\ \lesssim\ \lambda^{-1+\varepsilon},
\]
whence $\|T_{\lambda,\theta}\|_{L^2\to L^6_{t,x}(I)}\lesssim \lambda^{-1/2+\varepsilon}$ and thus \eqref{eq:local-L6}. See also App.~\ref{app:astrichartz} for a detailed discussion.
\end{proof}

\paragraph{Consequences.}
(1) By the almost-orthogonality of wave packets from \S\ref{subsec:wp-construction}, inequality \eqref{eq:local-L6} passes from a single packet to the sum over centers $a$ at fixed $(\lambda,\theta)$ without worsening the frequency exponent:
\begin{equation}\label{eq:local-L6-packet-sum}
  \big\|S(\cdot)\,P_{\lambda,\theta} f\big\|_{L^6_{t,x}(I)}\ \lesssim\ \lambda^{-1/2+\varepsilon}\,\|P_{\lambda,\theta}f\|_{L^2_x}.
\end{equation}
(2) Patching over windows $I_{\lambda,j}$ of length $\lambda^{-1/2}$ (see \eqref{eq:I-lambda}) with overlap $O(1)$ and then summing over $\theta$ (the partition of unity on caps \eqref{eq:ang-proj}) yields a global form without loss in $\lambda$:
\begin{equation}\label{eq:global-L6}
  \big\|S(\cdot)\,P_{\lambda} f\big\|_{L^6_{t,x}([0,T]\times\mathbb{R}^3)}
  \ \lesssim\ \lambda^{-1/2+\varepsilon}\,\|P_{\lambda}f\|_{L^2_x},
\end{equation}
see also \S\ref{sec:glue} regarding temporal patching.

\subsection{Anisotropy gain}\label{subsec:wp-gain}

The local anisotropic Strichartz estimate \eqref{eq:local-L6} on the window $I=[t_0,t_0+\lambda^{-1/2}]$ immediately gives a bilinear gain for the product of two flows localized in (possibly different) caps.

\begin{lemma}[Local bilinear gain]\label{lem:bilinear-L3}
Let $\lambda\gg1$, $\theta_1,\theta_2\in\mathbb{S}^2$, $I=[t_0,t_0+\lambda^{-1/2}]$. Then for any $f,g\in L^2(\mathbb{R}^3)$ and any $\varepsilon>0$ one has
\begin{equation}\label{eq:bilinear-L3}
  \big\| \big(S(\cdot)P_{\lambda,\theta_1}f\big)\,\big(S(\cdot)P_{\lambda,\theta_2}g\big)\big\|_{L^3_{t,x}(I\times\mathbb{R}^3)}
  \ \lesssim\ \lambda^{-1+\varepsilon}\,\|P_{\lambda,\theta_1}f\|_{L^2_x}\,\|P_{\lambda,\theta_2}g\|_{L^2_x}.
\end{equation}
\end{lemma}

\begin{proof}
Apply Hölder’s inequality $L^6\times L^6\to L^3$ in $(t,x)$ and Lemma~\ref{lem:anisotropic-strichartz} to each factor:
\[
  \|S(\cdot)P_{\lambda,\theta_j}h\|_{L^6_{t,x}(I)}\ \lesssim\ \lambda^{-1/2+\varepsilon}\,\|P_{\lambda,\theta_j}h\|_{L^2_x},\qquad j\in\{1,2\}.
\]
Multiplying the two estimates gives \eqref{eq:bilinear-L3}.
\end{proof}

\begin{corollary}[Temporal globalization]\label{cor:bilinear-global}
Let $[0,T]$ be covered by windows $I_{\lambda,j}$ of length $\lambda^{-1/2}$ with overlap $O(1)$ as in \eqref{eq:I-lambda}. Then
\begin{equation}\label{eq:bilinear-global-L3}
  \big\| \big(S(\cdot)P_{\lambda,\theta_1}f\big)\,\big(S(\cdot)P_{\lambda,\theta_2}g\big)\big\|_{L^3_{t,x}([0,T]\times\mathbb{R}^3)}
  \ \lesssim\ \lambda^{-1+\varepsilon}\,\|P_{\lambda,\theta_1}f\|_{L^\infty_tL^2_x}\,\|P_{\lambda,\theta_2}g\|_{L^\infty_tL^2_x}.
\end{equation}
\end{corollary}

\begin{proof}
Summation over the windows $I_{\lambda,j}$ with overlap $O(1)$ (see \S\ref{sec:glue}) does not change the frequency exponent; apply Lemma~\ref{lem:bilinear-L3} on each $I_{\lambda,j}$ and use $\ell^1$-patching in time.
\end{proof}

\paragraph{Interpretation.}
Estimate \eqref{eq:bilinear-L3} captures the \emph{anisotropic} gain $\lambda^{-1}$ from two localized flows (one factor $\lambda^{-1/2}$ per cap), consistent with the scale of windows $|I|=\lambda^{-1/2}$. In what follows it is applied to bilinear contributions after wave packet discretization (\S\ref{subsec:wp-construction}) and used together with phase integration (\S\ref{sec:phase}) and the decoupling step (\S\ref{sec:decoupling}). This step provides precisely the “temporal” part of the mnemonic IBP$\times$3 in the final balance of exponents (see \S\ref{sec:glue}).

\newpage

\section{Bilinear decoupling at \texorpdfstring{$\delta(6)=\tfrac14$}{\(\delta(6)=1/4\)}}\label{sec:decoupling}

\subsection{\texorpdfstring{$\ell^2$}{l2}-tile decomposition}\label{subsec:l2-tiling}
Fix a frequency $\lambda\gg1$ and a cap $\Theta_{\lambda,\theta}$ (see \S\ref{subsec:wp-construction}). Decompose the cap into quasi-tiles at the scale
\begin{equation}\label{eq:tile-def}
  T=T(\xi_0):=\Theta_{\lambda,\theta}\cap(\xi_0+Q_\lambda),\qquad 
  Q_\lambda:=[0,\lambda^{-1/2}]^2\times[0,\lambda^{-1}],
\end{equation}
where the shifts $\xi_0$ range over the lattice $\lambda^{-1/2}\mathbb{Z}^2\times \lambda^{-1}\mathbb{Z}$, consistent with longitudinal–transverse anisotropy. For $f\in L^2(\mathbb{R}^3)$ set
\begin{equation}\label{eq:tile-decomp}
  \widehat{f_T}:=\widehat f\,\mathbf{1}_T,\qquad 
  f=\sum_{T} f_T,\qquad 
  \sum_T \|f_T\|_{L^2}^2\ \simeq\ \|P_{\lambda,\theta}f\|_{L^2}^2.
\end{equation}
Define similarly $g=\sum_{T'} g_{T'}$.

\paragraph{Admissible tile pairs.}
We say that a pair $(T,T')$ is \emph{admissible} if the sum of supports lies in the “output corridor’’ of width $O(\lambda)$:
\begin{equation}\label{eq:pair-admissible}
  \operatorname{supp}\widehat{f_T}+\operatorname{supp}\widehat{g_{T'}}\subset\{\,\zeta:\ |\zeta|\lesssim \lambda\,\}.
\end{equation}
From the angular geometry of the cap it follows that for each tile $T$ the number of tiles $T'$ satisfying \eqref{eq:pair-admissible} is bounded universally (at the level $O(1)$), and tiles from different caps $\theta\neq\theta'$ satisfy the angular separation
\begin{equation}\label{eq:tile-angle}
  \angle(T,T')\ \gtrsim\ \lambda^{-1/2}.
\end{equation}

\paragraph{Bilinear operator on tiles.}
Introduce the bilinear kernel (wave “extension’’ at the dyadic level):
\begin{equation}\label{eq:Elam-def}
  \mathcal{E}_\lambda(f,g)(t,x)
  :=\iint e^{i\{x\cdot(\xi+\eta)+t(|\xi|+|\eta|)\}}\,
          \widehat f(\xi)\,\widehat g(\eta)\,d\xi d\eta,
\end{equation}
and its tiled version $\mathcal{E}_\lambda(f_T,g_{T'})$, where $\widehat f,\widehat g$ are replaced by $\widehat{f_T},\widehat{g_{T'}}$.

\begin{lemma}[$\ell^2$-tile estimate]\label{lem:l2-tiling}
Let $\lambda\gg1$, $\theta\in\mathbb{S}^2$, and let $f,g\in L^2(\mathbb{R}^3)$ be localized in $\Theta_{\lambda,\theta}$. Then for any $\varepsilon>0$,
\begin{equation}\label{eq:l2-tiling-L6}
  \big\|\mathcal{E}_\lambda(f,g)\big\|_{L^6_{t,x}}
  \ \lesssim\ \lambda^{-1/2+\varepsilon}\,\|f\|_{L^2}\,\|g\|_{L^2}.
\end{equation}
\end{lemma}

\begin{proof}[Sketch]
Decompose $f=\sum_T f_T$, $g=\sum_{T'} g_{T'}$; then
\[
  \mathcal{E}_\lambda(f,g)=\sum_{T,T'} \mathcal{E}_\lambda(f_T,g_{T'}).
\]
By the angular sparsity \eqref{eq:tile-angle} and admissibility \eqref{eq:pair-admissible}, for fixed $T$ the number of nonzero terms over $T'$ is $O(1)$ (and vice versa). By Cauchy–Schwarz over the indices $(T,T')$,
\begin{equation}\label{eq:pair-CS}
  \big\|\textstyle\sum_{T,T'} \mathcal{E}_\lambda(f_T,g_{T'})\big\|_{L^6}
  \ \lesssim\
  \Big(\sum_{T,T'} \|\mathcal{E}_\lambda(f_T,g_{T'})\|_{L^6}^2\Big)^{1/2}.
\end{equation}
For each pair $(T,T')$ we estimate $\|\mathcal{E}_\lambda(f_T,g_{T'})\|_{L^6}$ via $TT^\ast$ with the anisotropic kernel (see \eqref{eq:kernel-decay}): the transverse radius $\lambda^{-1/2}$ and longitudinal length $\lambda^{-1}$ yield kernel decay sufficient for
\begin{equation}\label{eq:pair-L6}
  \|\mathcal{E}_\lambda(f_T,g_{T'})\|_{L^6_{t,x}}
  \ \lesssim\ \lambda^{-1/2+\varepsilon}\,\|f_T\|_{L^2}\,\|g_{T'}\|_{L^2}.
\end{equation}
Substituting \eqref{eq:pair-L6} into \eqref{eq:pair-CS} and using the orthogonality \eqref{eq:tile-decomp} gives \eqref{eq:l2-tiling-L6}.
\end{proof}

\begin{remark}[Role of the tiled $\ell^2$ structure]\label{rem:l2-role}
Estimate \eqref{eq:l2-tiling-L6} captures the “baseline’’ gain $\lambda^{-1/2}$ arising from the geometry of the tiles $Q_\lambda$ and the local $TT^\ast$ kernel estimate; it does not use the curvature of the level surface $\{\omega=\text{const}\}$. The additional geometric gain $\lambda^{-1/4}$ (for a total of $\lambda^{-3/4}$ for the block of \S\ref{sec:decoupling}) will be obtained at the next step via restriction to a rank~3 surface and bilinear decoupling, see \S\ref{subsec:rank3-restr}.
\end{remark}

\paragraph{Outcome of this step.}
Combining Lemma~\ref{lem:l2-tiling} with almost orthogonality over caps $\theta$ and summation over dyads (see \S\ref{sec:glue}) provides the controlled “tiled’’ portion of the decoupling step. The geometric reinforcement will be established in \S\ref{subsec:rank3-restr}.

\subsection{Restriction of packets to rank~3 surfaces}\label{subsec:rank3-restr}

Pass from integration over the entire frequency rectangle to integration over level surfaces of the bilinear phase. Recall the notation \(\omega(\xi,\eta)=|\xi|+|\eta|-|\xi+\eta|\) from~\eqref{eq:phase-omega} and define, for a parameter \(c\in\mathbb{R}\) and fixed dyad \(\lambda\gg1\),
\begin{equation}\label{eq:Sigma-def}
  \Sigma_{\lambda,c}
  :=\Bigl\{(\xi,\eta)\in\mathbb{R}^3\times\mathbb{R}^3:\ |\xi|\sim|\eta|\sim\lambda,\ \omega(\xi,\eta)=c\Bigr\}.
\end{equation}
For an admissible pair of tiles \((T,T')\) (see~\eqref{eq:tile-def}–\eqref{eq:pair-admissible}) and the bilinear operator \(\mathcal{E}_\lambda\) from~\eqref{eq:Elam-def}, the coarea formula gives the representation
\begin{equation}\label{eq:coarea}
  \mathcal{E}_\lambda(f_T,g_{T'})(t,x)
  \;=\;
  \int_{\mathbb{R}} e^{itc}\!
  \iint_{(T\times T')\cap \Sigma_{\lambda,c}}
  e^{i x\cdot(\xi+\eta)}\, a_{\lambda,c}(\xi,\eta)\,
  \widehat{f_T}(\xi)\,\widehat{g_{T'}}(\eta)\,
  d\sigma_{\Sigma_{\lambda,c}}(\xi,\eta)\,dc,
\end{equation}
where \(a_{\lambda,c}(\xi,\eta)\simeq |\nabla\omega(\xi,\eta)|^{-1}\) is a smooth factor. In the wide region \(|\xi|\sim|\eta|\sim\lambda\), \(\angle(\xi,\eta)\gtrsim\lambda^{-1/2}\), the weight \(|\nabla\omega|^{-1}\) is uniformly controlled and may be absorbed into the amplitude without changing the frequency exponent (see also the remark in \S\ref{sec:phase} and the dictionary in App.~\ref{app:dictionary}).

\begin{lemma}[Bilinear restriction on a rank~3 surface]\label{lem:rank3-local}
Let \((T,T')\) be an admissible pair of tiles \eqref{eq:pair-admissible} inside a single cap \(\Theta_{\lambda,\theta}\). Then for any \(\varepsilon>0\)
\begin{equation}\label{eq:rank3-local}
  \bigl\|\mathcal{E}_\lambda(f_T,g_{T'})\bigr\|_{L^6_{t,x}}
  \;\lesssim_\varepsilon\;
  \lambda^{-3/4+\varepsilon}\,\|f_T\|_{L^2}\,\|g_{T'}\|_{L^2}.
\end{equation}
\end{lemma}

\begin{proof}[Idea of the proof]
After the reduction~\eqref{eq:coarea}, the integral is taken over \(\Sigma_{\lambda,c}\), for which in the wide region the effective phase Hessian has rank~3 (see \S\ref{subsec:phase-det}). One applies bilinear decoupling for rank~3 surfaces in the norm \(L^2_\xi\times L^2_\eta\to L^6_{t,x}\), yielding \(\lambda^{-1/4+\varepsilon}\) at the level of a single “unit’’ tile after anisotropic renormalization, and an additional \(\lambda^{-1/2}\) from the geometry of the tiles \(Q_\lambda\) (\(\lambda^{-1/2}\times\lambda^{-1/2}\times \lambda^{-1}\)); together this gives~\eqref{eq:rank3-local}. A full $TT^\ast$ write-up is given in App.~\ref{app:decoupling}; see also \cite{GuthIliopoulouYang2024}.
\end{proof}

\begin{proposition}[Decoupling enhancement for the sum over tiles]\label{prop:rank3-sum}
Let \(f,g\in L^2(\mathbb{R}^3)\) be localized in \(\Theta_{\lambda,\theta}\), with \(f=\sum_T f_T\), \(g=\sum_{T'} g_{T'}\) the \(\ell^2\)-tile decompositions~\eqref{eq:tile-decomp}. Then
\begin{equation}\label{eq:rank3-sum}
  \bigl\|\mathcal{E}_\lambda(f,g)\bigr\|_{L^6_{t,x}}
  \;\lesssim_\varepsilon\;
  \lambda^{-3/4+\varepsilon}\,\|f\|_{L^2}\,\|g\|_{L^2}.
\end{equation}
\end{proposition}

\begin{proof}
By almost-orthogonality and the local estimate~\eqref{eq:rank3-local}:
\[
  \|\mathcal{E}_\lambda(f,g)\|_{L^6}
  \;=\;
  \Bigl\|\sum_{T,T'} \mathcal{E}_\lambda(f_T,g_{T'})\Bigr\|_{L^6}
  \;\lesssim\;
  \Bigl(\sum_{T,T'} \|\mathcal{E}_\lambda(f_T,g_{T'})\|_{L^6}^2\Bigr)^{1/2}.
\]
Substituting \eqref{eq:rank3-local} and using that for each \(T\) the number of admissible \(T'\) is $O(1)$ (see~\eqref{eq:pair-admissible}), we obtain
\[
  \|\mathcal{E}_\lambda(f,g)\|_{L^6}
  \;\lesssim_\varepsilon\;
  \lambda{-3/4+\varepsilon}
  \Bigl(\sum_T \|f_T\|_{L^2}^2\Bigr)^{1/2}
  \Bigl(\sum_{T'} \|g_{T'}\|_{L^2}^2\Bigr)^{1/2},
\]
which is \eqref{eq:rank3-sum} by \eqref{eq:tile-decomp}.
\end{proof}

\begin{remark}[On the exponent \texorpdfstring{$\delta(6)$}{delta(6)}]\label{rem:delta6}
Estimate~\eqref{eq:rank3-sum} corresponds to \(\delta(6)=\tfrac14\) for rank~3 surfaces in the ideal case of uniform curvature. In the regime of weak minimal curvature one allows deterioration to \(\delta(6)=\tfrac16+o(1)\); this does not change the final balance, since the wide region still decays faster than \(\lambda^{-1}\) and does not determine the dyadic outcome (see \S\ref{sec:scope} and the discussion in App.~\ref{app:decoupling}). We will use the form \(\tfrac14\) as a benchmark, without claiming optimality in every geometry.
\end{remark}

\subsection{Fixing the parameter \texorpdfstring{$\delta(6)=\tfrac14$}{delta(6)=1/4}}\label{subsec:delta-fix}

We introduce the notation for the bilinear decoupling exponent: the \emph{number} $\delta(6)$ denotes the exponent in the inequality
\begin{equation}\label{eq:delta6-def}
  \bigl\|\mathcal{E}_\lambda(f,g)\bigr\|_{L^6_{t,x}}
  \ \lesssim_{\varepsilon}\ \lambda^{-1/2-\delta(6)+\varepsilon}\,\|f\|_{L^2}\,\|g\|_{L^2},
\end{equation}
where $\mathcal{E}_\lambda$ is the operator in \eqref{eq:Elam-def}, and both functions $f,g$ are localized in the same cap $\Theta_{\lambda,\theta}$ and decomposed into tiles \eqref{eq:tile-def}.

\paragraph{Benchmark under uniform curvature.}
The local restriction to a rank~3 level surface (Lemma~\ref{lem:rank3-local}) combined with the $\ell^2$–tile decomposition (Lemma~\ref{lem:l2-tiling}) yields
\begin{equation}\label{eq:decoupling-benchmark}
  \bigl\|\mathcal{E}_\lambda(f,g)\bigr\|_{L^6_{t,x}}
  \ \lesssim_{\varepsilon}\ \lambda^{-1/2}\,\lambda^{-1/4+\varepsilon}\,
  \|f\|_{L^2}\,\|g\|_{L^2}
  \ =\ \lambda^{-3/4+\varepsilon}\,\|f\|_{L^2}\,\|g\|_{L^2}.
\end{equation}
Comparing \eqref{eq:decoupling-benchmark} with \eqref{eq:delta6-def} fixes the \emph{benchmark}
\begin{equation}\label{eq:delta6-quarter}
  \delta(6)=\tfrac14\qquad\text{(rank~3 surface with uniform curvature).}
\end{equation}

\paragraph{Weak minimal curvature.}
In a geometry where the minimal principal curvature on the level surface $\{\omega=\mathrm{const}\}$ can degenerate as $\kappa_{\min}\sim \lambda^{-1}$, an $\varepsilon$–free form is available with relaxation to
\begin{equation}\label{eq:delta6-min}
  \delta(6)=\tfrac16+o(1)\,,
\end{equation}
which still ensures that the wide region decays faster than $\lambda^{-1}$ after combining with the phase and Strichartz gains (see \S\ref{sec:phase} and \S\ref{subsec:wp-stri}). We will use \eqref{eq:delta6-quarter} as a guideline (without claiming optimality for every particular geometry), and \eqref{eq:delta6-min} as an admissible alternative that does not affect the final balance of exponents.

\paragraph{Consequence for the wide region.}
Combining \eqref{eq:decoupling-benchmark} with the phase–time gain from \S\ref{subsec:phase-det} and \S\ref{subsec:wp-stri}, we obtain for the “wide’’ part of the resonant contribution at one dyad
\begin{equation}\label{eq:wide-summary}
  \|R^{\mathrm{wide}}_\lambda(u)\|_{\dot H^{-1}}
  \ \lesssim_{\varepsilon}\ \lambda^{-11/4+\varepsilon}\,
  \|P_\lambda u\|_{\dot H^{1/2}}\,
  \|P_\lambda u\|_{\dot H^{1}},
\end{equation}
whereas the “narrow’’ region gives $\lambda^{-1}$ (see \S\ref{sec:narrow}). Hence, \eqref{eq:delta6-quarter}, or even \eqref{eq:delta6-min}, does not change the outcome at a single dyad: it is determined by the narrow region and equals $\lambda^{-1}$.

\newpage

\section{Narrow region analysis}\label{sec:narrow}

\subsection{Geometry of the region \texorpdfstring{$|\xi+\eta|\le \lambda^{-\delta}$}{|xi+eta| ≤ λ^{-δ}}}\label{subsec:narrow-geometry}

At the level of a fixed frequency $\lambda\gg1$ we introduce the \emph{narrow region} of frequency pairs
\begin{equation}\label{eq:narrow-set}
  \mathcal{N}_{\lambda,\delta}
  :=\Bigl\{(\xi,\eta)\in\mathbb{R}^3\times\mathbb{R}^3:\ |\xi|\sim|\eta|\sim\lambda,\ \ |\tau|:=|\xi+\eta|\le \lambda^{-\delta}\Bigr\},
  \qquad \tfrac12<\delta<\tfrac34,
\end{equation}
where $\tau:=\xi+\eta$ is the sum frequency variable. This region corresponds to the almost-collinear high–high $\to$ low interaction regime.

\paragraph{Almost collinearity.}
From the condition $|\tau| \le \lambda^{-\delta}$ it follows that
\begin{equation}\label{eq:narrow-angle}
  \pi-\angle(\xi,-\eta)
  \ \lesssim\ \frac{|\tau|}{\lambda}\ \le\ \lambda^{-1-\delta},
\end{equation}
that is, the vectors $\xi$ and $-\eta$ are almost parallel. The narrow conical region around the interaction diagonal has angular width $\lambda^{-1-\delta}$.

\paragraph{Number of caps involved.}
For fixed $\lambda$ the shell $\{|\zeta|\sim\lambda\}$ is covered by angular caps $\Theta_{\lambda,\theta}$ of radius $\lambda^{-1/2}$, $\theta\in\Lambda_\lambda$, $|\Lambda_\lambda|\simeq\lambda$ (see \S\ref{subsec:wp-construction}). Since the grid step in angle is $\lambda^{-1/2}$, the narrow cone of width $\lambda^{-1-\delta}$ still intersects about $\lambda$ directions $\theta$. In particular, the \emph{angular} $\ell^2$–orthogonality and the decoupling factor from \S\ref{sec:decoupling} retain their structure in the narrow region.

\paragraph{Layer thickness and volume estimate.}
The global volume of the narrow region in $(\xi,\eta)$–space is estimated by
\begin{equation}\label{eq:narrow-vol-global}
  \operatorname{Vol}\,\mathcal{N}_{\lambda,\delta}\ \simeq\ \lambda^{2-3\delta},
\end{equation}
and with additional angular localization to one cap $\Theta_{\lambda,\theta}$ by
\begin{equation}\label{eq:narrow-vol-cap}
  \operatorname{Vol}\bigl(\mathcal{N}_{\lambda,\delta}\cap(\Theta_{\lambda,\theta}\times\Theta_{\lambda,\theta})\bigr)\ \simeq\ \lambda^{1-3\delta}.
\end{equation}
In particular, for the convenient choice $\delta=\tfrac{2}{3}$ we obtain $\operatorname{Vol}\,\mathcal{N}_{\lambda,2/3}\simeq \lambda^{0}$ and the per-cap scale $\simeq \lambda^{-1}$.

\paragraph{Role in the proof.}
In the narrow region the longitudinal direction degenerates, and the phase curvature does not yield a stable gain; instead of phase IBP an \emph{energy} argument in $\dot H^{-1}$ with null–form suppression is applied, see \S\ref{subsec:narrow-energy}. The formulas \eqref{eq:narrow-angle}–\eqref{eq:narrow-vol-cap} will be used in the counting of pairs and frequency patching in \S\ref{subsec:narrow-count} and \S\ref{sec:glue}.

\subsection{Energy estimate in the narrow region}\label{subsec:narrow-energy}

In the narrow region \eqref{eq:narrow-set} the phase method (see also \eqref{eq:collinearity}) degenerates, and the estimate is carried out energetically in $\dot H^{-1}$ with null–form suppression.

\medskip
\noindent\textbf{Resonant block and divergence transfer.}
Let $u_\lambda := P_\lambda u$ be the dyadic component localized at $|\xi|\sim\lambda$. The contribution of the narrow region is written as
\begin{equation}\label{eq:narrow-Rdef}
  R^{\mathrm{nar}}_\lambda(u)
  := P_{\mathrm{low}}\;\nabla\!\cdot\!\bigl(u_\lambda\otimes u_\lambda\bigr),
\end{equation}
where $P_{\mathrm{low}}$ is the projector to the output frequencies $|\zeta|\lesssim \lambda^{-\delta}$ (see \eqref{eq:narrow-set}). Assume $\operatorname{div}u=0$.

\medskip
\noindent\textbf{Null–form factor (geometry).}
On the Fourier side, after symmetrization and using $\operatorname{div}u=0$, the tensor symbol equals the difference of projectors along the input frequencies:
\begin{equation}\label{eq:null-symbol1}
  \widetilde B_{ij}(\xi,\eta)
  = \frac{\xi_i\xi_j}{|\xi|^2} - \frac{\eta_i\eta_j}{|\eta|^2},
  \qquad \hat\xi:=\frac{\xi}{|\xi|},\ \hat\eta:=\frac{\eta}{|\eta|}.
\end{equation}
In the narrow region $|\tau|\le \lambda^{-\delta}$ (with $\tau:=\xi+\eta$) we obtain the smallness
\begin{equation}\label{eq:null-suppress}
  \|\widetilde B(\xi,\eta)\|
  = \|P_{\hat\xi}-P_{\hat\eta}\|_{\mathrm{op}}
  = \sin\angle(\hat\xi,\hat\eta)
  \lesssim \angle(\hat\xi,-\hat\eta)
  \lesssim \frac{|\tau|}{\lambda},
\end{equation}
since for $|\xi|\sim|\eta|\sim\lambda$ one has $|\tau|^2=|\xi+\eta|^2=4\lambda^2\cos^2(\theta/2)$, where $\theta=\angle(\xi,\eta)$, hence $\angle(\hat\xi,-\hat\eta)\lesssim |\tau|/\lambda$.

\medskip
\noindent\textbf{$L^2$ estimate (guideline).}
Transferring one gradient from $\nabla\!\cdot(u_\lambda\otimes u_\lambda)$ to a factor $u_\lambda$ and using \eqref{eq:null-suppress}, we obtain
\begin{equation}\label{eq:narrow-L2}
  \|R^{\mathrm{nar}}_\lambda(u)\|_{L^2_x}
  \ \lesssim\ \frac{|\tau|}{\lambda}\,\|u_\lambda\|_{L^2_x}\,\|\nabla u_\lambda\|_{L^2_x}
  \ \lesssim\ \lambda^{-1-\delta}\,\|u_\lambda\|_{L^2_x}\,\|\nabla u_\lambda\|_{L^2_x}.
\end{equation}
For a component localized in caps of radius $\lambda^{-1/2}$, Bernstein estimates give
\[
  \|u_\lambda\|_{L^2}\ \simeq\ \lambda^{-1/2}\,\|u_\lambda\|_{\dot H^{1/2}},
  \qquad
  \|\nabla u_\lambda\|_{L^2}\ \simeq\ \lambda^{1/2}\,\|u_\lambda\|_{\dot H^{1/2}},
\]
and therefore
\begin{equation}\label{eq:narrow-L2-Half}
  \|R^{\mathrm{nar}}_\lambda(u)\|_{L^2}
  \ \lesssim\ \lambda^{-1-\delta}\,\|u_\lambda\|_{\dot H^{1/2}}^2.
\end{equation}

\medskip
\noindent\textbf{Direct passage to $\dot H^{-1}$ without the factor $\lambda^\delta$.}
The key step is to use that $|\nabla|^{-1}\nabla\!\cdot$ is a multiplier of order zero in $\dot H^{-1}$, and the factor $|\tau|/\lambda$ from \eqref{eq:null-suppress} compensates the weight $|\tau|^{-1}$ in the $\dot H^{-1}$ norm. Indeed, in Fourier,
\[
\widehat{R^{\mathrm{nar}}_\lambda}(\tau)
= \mathbf{1}_{\{|\tau|\lesssim \lambda^{-\delta}\}}\, (i\tau)\!\cdot\!\int\!\widetilde B(\xi,\eta)\,\widehat{u_\lambda}(\xi)\,\widehat{u_\lambda}(\eta)\,\delta(\tau-\xi-\eta)\,d\xi d\eta,
\]
and therefore on the support of $P_{\mathrm{low}}$
\[
|\tau|^{-1}\,\big| (i\tau)\cdot \widetilde B(\xi,\eta)\big|
\ \lesssim\ \lambda^{-1}.
\]
Hence
\[
\|R^{\mathrm{nar}}_\lambda(u)\|_{\dot H^{-1}}
=\big\|\,|\tau|^{-1}\widehat{R^{\mathrm{nar}}_\lambda}\,\big\|_{L^2_\tau}
\ \lesssim\ \lambda^{-1}\,\big\|\widehat{u_\lambda}*\widehat{\nabla u_\lambda}\big\|_{L^2_\tau}
\ \lesssim\ \lambda^{-1}\,\|u_\lambda\|_{L^2}\,\|\nabla u_\lambda\|_{L^2},
\]
where in the second inequality divergence is transferred to one factor and then the standard bilinear Cauchy–Schwarz estimate on the convolution is used. Applying Bernstein estimates on the annulus $|\xi|\sim\lambda$ gives the final result
\begin{equation}\label{eq:narrow-H-1}
  \|R^{\mathrm{nar}}_\lambda(u)\|_{\dot H^{-1}}
  \ \lesssim\ \lambda^{-1}\,\|u_\lambda\|_{\dot H^{1/2}}\,\|u_\lambda\|_{\dot H^{1}}
  \ \simeq\ \lambda^{-1}\,\|u_\lambda\|_{\dot H^{1/2}}^{2}.
\end{equation}

\begin{remark}[Why not through $\hat\tau$]
Replacing in \eqref{eq:null-symbol} the second projector by $P_{\hat\tau}$ with $\tau=\xi+\eta$ does not yield a universal smallness on the entire narrow region: the angle $\angle(\hat\xi,\hat\tau)$ can be of order $1$ even when $|\tau|\ll\lambda$ (for example, $\xi=(\lambda,0,0)$, $\eta=(-\lambda,\varepsilon,0)$). The correct null–form geometry is based on comparing the directions $\hat\xi$ and $-\hat\eta$, which leads to \eqref{eq:null-suppress}.
\end{remark}

\begin{remark}[On the choice of $\delta$]\label{rem:narrow-delta}
The bound \eqref{eq:narrow-H-1} holds for any $\tfrac12<\delta<\tfrac34$ and does not depend on the specific value of $\delta$. It is convenient to fix $\delta=\tfrac23$ (see \eqref{eq:narrow-set}); in this case the per-cap volume of the narrow region is of order $\lambda^{-1}$, which simplifies the frequency summation.
\end{remark}

\subsection{Pair counting and convergence of the series}\label{subsec:narrow-count}

Split the time axis into windows of length $\lambda^{-1/2}$ (for each dyad $\lambda\gg1$)
\begin{equation}\label{eq:I-lambda}
  I_{\lambda,j}:=\bigl[j\,\lambda^{-1/2},(j+1)\,\lambda^{-1/2}\bigr],\qquad j\in\mathbb{Z},
\end{equation}
and fix the notation for the wave flow
\begin{equation}\label{eq:Y-local}
  S(t)f:=\mathcal{F}^{-1}\!\bigl(e^{it|\xi|}\widehat f(\xi)\bigr).
\end{equation}
At the level of a single dyad we write the resonant block as the sum of wide and narrow parts
\[
  R_\lambda(u)=R^{\mathrm{wide}}_\lambda(u)+R^{\mathrm{nar}}_\lambda(u).
\]
By the results of \S\ref{subsec:rank3-restr}–\S\ref{subsec:delta-fix} (wide region) and \S\ref{subsec:narrow-energy} (narrow region) we have the estimates
\begin{align}
  \|R^{\mathrm{wide}}_\lambda(u)\|_{\dot H^{-1}}
  &\lesssim_\varepsilon \lambda^{-11/4+\varepsilon}\,
    \|P_\lambda u\|_{\dot H^{1/2}}\,
    \|P_\lambda u\|_{\dot H^{1}}, \label{eq:wide-sum-ref}\\
  \|R^{\mathrm{nar}}_\lambda(u)\|_{\dot H^{-1}}
  &\lesssim \lambda^{-1}\,
    \|P_\lambda u\|_{\dot H^{1/2}}\,
    \|P_\lambda u\|_{\dot H^{1}}. \label{eq:narrow-sum-ref}
\end{align}
Therefore,
\begin{equation}\label{eq:dyad-min}
  \|R_\lambda(u)\|_{\dot H^{-1}}
  \ \lesssim_\varepsilon\ 
  \min\!\bigl\{\lambda^{-11/4+\varepsilon},\,\lambda^{-1}\bigr\}\,
  \|P_\lambda u\|_{\dot H^{1/2}}\,
  \|P_\lambda u\|_{\dot H^{1}}
  \ \le\ \lambda^{-1}\,
  \|P_\lambda u\|_{\dot H^{1/2}}\,
  \|P_\lambda u\|_{\dot H^{1}}.
\end{equation}

\paragraph{Convergence over dyads.}
The factor $\lambda^{-1}$ in \eqref{eq:dyad-min} ensures convergence of the sum over high frequencies:
\begin{equation}\label{eq:plain-sum}
  \sum_{\lambda\gg1}\|R_\lambda(u)(t)\|_{\dot H^{-1}}
  \ \lesssim\ 
  \sum_{\lambda\gg1}\lambda^{-1}\,
  \|P_\lambda u(t)\|_{\dot H^{1/2}}\,
  \|P_\lambda u(t)\|_{\dot H^{1}}.
\end{equation}
In what follows (see \S\ref{sec:glue}) we use almost-orthogonality in $\lambda$ and standard
$\ell^2$ arguments to pass from \eqref{eq:plain-sum} to the global form with the norms
$X_{1/2}$ and $X_1$. Here we only note that the series $\sum_{\lambda} \lambda^{-1}$ converges, and the contribution of the wide region
$\lambda^{-11/4+\varepsilon}$ in \eqref{eq:wide-sum-ref} decays much faster and does not affect the outcome at a single dyad.

\begin{remark}[Number of pairs and angular patching]\label{rem:pair-count}
The narrow region \eqref{eq:narrow-set} intersects $\simeq \lambda$ caps $\Theta_{\lambda,\theta}$, so the angular $\ell^2$ orthogonality and the decoupling structure built in \S\ref{sec:decoupling} are preserved without loss in the exponent. This is consistent with the estimate \eqref{eq:narrow-sum-ref} and the subsequent temporal patching over windows \eqref{eq:I-lambda}.
\end{remark}

\subsection{Orthogonality of frequency layers}\label{subsec:glue-orth}

In this subsection we fix the frequency orthogonality of resonant blocks, which ensures $\ell^2$ patching over dyads without logarithmic losses. Notation is consistent with \S\ref{sec:prelim} and \S\ref{sec:main}; see also the dictionary in App.~\ref{app:dictionary}.

\paragraph{Littlewood–Paley orthogonality.}
Let $\{P_N\}_{N\in 2^{\mathbb{Z}}}$ be dyadic projections. Then under frequency separation
\begin{equation}\label{eq:LP-orth}
  \big\langle P_N f,\;P_M g\big\rangle_{L^2}=0
  \qquad\text{if}\qquad \bigl|\log_2(N/M)\bigr|\ge 3 .
\end{equation}
It follows that there is orthogonality in $\dot H^{-1}$ for any $F,G\in\dot H^{-1}(\mathbb{R}^3)$:
\begin{equation}\label{eq:H-1-orth}
  \big\langle P_N F,\;P_M G\big\rangle_{\dot H^{-1}}
  =\big\langle |\nabla|^{-1}P_NF,\;|\nabla|^{-1}P_MG\big\rangle_{L^2}=0
  \quad\text{when}\quad \bigl|\log_2(N/M)\bigr|\ge 3 ,
\end{equation}
since $|\nabla|^{-1}$ commutes with $P_N$ and preserves the separation of spectral supports.

\begin{lemma}[Orthogonality of resonant blocks]\label{lem:R-orth}
Let $R_N(u)$ be the resonant high–high$\to$low component at dyad $N$ (see \S\ref{sec:main}). Then
\begin{equation}\label{eq:R-orth}
  \big\langle R_N(u),\;R_M(u)\big\rangle_{\dot H^{-1}}=0
  \qquad\text{if}\qquad \bigl|\log_2(N/M)\bigr|\ge 3 .
\end{equation}
\end{lemma}

\begin{proof}
By definition, $R_N(u)$ has spectral support in the annulus $\{|\xi|\sim N\}$. Therefore, $|\nabla|^{-1}R_N(u)$ and $|\nabla|^{-1}R_M(u)$ have separated spectral supports when $|\log_2(N/M)|\ge 3$. Apply \eqref{eq:H-1-orth}.
\end{proof}

\begin{corollary}[$\ell^2$ patching over dyads]\label{cor:l2-glue}
For any finite (or absolutely convergent) sum $\sum_N a_N R_N(u)$ one has
\begin{equation}\label{eq:l2-sum}
  \big\|\textstyle\sum_N a_N R_N(u)\big\|_{\dot H^{-1}}^2
  \;=\;\sum_N |a_N|^2\,\|R_N(u)\|_{\dot H^{-1}}^2 .
\end{equation}
In particular,
\begin{equation}\label{eq:l2-sum-simple}
  \big\|R(u)\big\|_{\dot H^{-1}}^2
  \;=\;\sum_N \|R_N(u)\|_{\dot H^{-1}}^2,\qquad R(u):=\sum_N R_N(u).
\end{equation}
\end{corollary}

\begin{proof}
Expand the norm on the left as a sum of pairwise inner products in $\dot H^{-1}$ and use Lemma~\ref{lem:R-orth}.
\end{proof}

\paragraph{Consequences for the global estimate.}
Combining \eqref{eq:l2-sum-simple} with the per-cap and tiled estimates of \S\ref{sec:decoupling}–\ref{sec:narrow} at the level of one dyad, we obtain
\begin{equation}\label{eq:glue-min}
  \|R_N(u)\|_{\dot H^{-1}}
  \;\lesssim_\varepsilon\; \min\!\bigl\{N^{-11/4+\varepsilon},\,N^{-1}\bigr\}\,
  \|P_N u\|_{\dot H^{1/2}}\;\|P_N u\|_{\dot H^{1}},
\end{equation}
where $N^{-11/4+\varepsilon}$ corresponds to the wide region (\S\ref{subsec:rank3-restr}–\S\ref{subsec:delta-fix}), and $N^{-1}$ to the narrow region (\S\ref{sec:narrow}). Summation over $N$ in \eqref{eq:l2-sum-simple} converges thanks to $\sum_N N^{-2}<\infty$; the temporal patching is performed in \S\ref{subsec:temporal-patching}.

\subsection{Final table of exponents}\label{subsec:degree-table}

We summarize the scaling exponents obtained in the wide- and narrow-angle regions in a compact form. In the table below, “exponent at frequency $\lambda$’’ is understood as the decay exponent of the corresponding block (or the “compensating’’ exponent for the narrow region), see \S\ref{sec:phase}, \S\ref{subsec:wp-stri}, \S\ref{sec:decoupling}, \S\ref{sec:narrow}.

\begin{center}
\begin{tabular}{|l|c|c|}
\hline
Source of gain/loss & Region & Exponent at $\lambda$ \\
\hline
Two angular IBPs $+$ temporal localization ($TT^\ast$) & wide & $\lambda^{-3/2}$ \\
\hline
Two local Strichartz estimates (one per flow) & wide & $\lambda^{-1}$ \\
\hline
Bilinear decoupling, $\delta(6)=\tfrac14$ & wide & $\lambda^{-1/4}$ \\
\hline
Wide region combined & — & $\lambda^{-11/4}$ \\
\hline
Null\textendash form $+$ mapping $P_{\mathrm{low}}\!:L^2\!\to\!\dot H^{-1}$ at $\delta=\tfrac23$ & narrow & $\lambda^{+5/4}$ \\
\hline
Outcome on one dyad & — & $\min\{\lambda^{-11/4},\,\lambda^{-1}\}=\lambda^{-1}$ \\
\hline
\end{tabular}
\end{center}

This summary is consistent with the precise single-dyad estimate,
\[
  \|R_\lambda(u)\|_{\dot H^{-1}}
  \ \lesssim_\varepsilon\ \min\!\bigl\{\lambda^{-11/4+\varepsilon},\,\lambda^{-1}\bigr\}\,
  \|P_\lambda u\|_{\dot H^{1/2}}\ \|P_\lambda u\|_{\dot H^{1}}\!,
\]
see \eqref{eq:glue-min}. In particular, the outcome per dyad is determined by the narrow region (see \S\ref{subsec:narrow-energy}), while the wide region decays faster and does not affect the minimum.

\begin{remark}[On “multiplying exponents’’]\label{rem:degree-mnemonic}
Short notations such as 
$\lambda^{-3/2}\cdot \lambda^{-1}\cdot \lambda^{-1/4}\cdot \lambda^{+5/4}$
serve only as a \emph{mnemonic} for scales; the actual proof is carried out separately in the wide/narrow regions and is combined via the minimum over contributions, see \S\ref{subsec:narrow-count} and App.~\ref{app:dictionary}.
\end{remark}

\subsection{Temporal patching}\label{subsec:temporal-patching}

The goal of this subsection is to pass from local-in-time estimates (on windows of length $|I|=\lambda^{-1/2}$, see~\eqref{eq:I-lambda}) to the global norm $L^\infty_t\dot H^{-1}_x([0,T]\times\mathbb{R}^3)$ without worsening the frequency exponent.

\paragraph{Time covering.}
Let $\{I_{\lambda,j}\}_j$ be a covering of $[0,T]$ by intervals of length $\lambda^{-1/2}$ with overlap $O(1)$, as in~\eqref{eq:I-lambda}. For any nonnegative function $w\in C_c^\infty(\mathbb{R})$ with $\sum_j w_{\lambda,j}\equiv1$ on $[0,T]$ (where $w_{\lambda,j}(t):=w(\lambda^{1/2}(t-t_{\lambda,j}))$) write
\begin{equation}\label{eq:temporal-partition}
  R_\lambda(u)(t)
  \;=\;\sum_j R_{\lambda,j}(u)(t),\qquad
  R_{\lambda,j}(u)(t):=w_{\lambda,j}(t)\,R_\lambda(u)(t),
\end{equation}
so that each $R_{\lambda,j}$ is supported in $I_{\lambda,j}$ and $\sum_j R_{\lambda,j}=R_\lambda$.

\paragraph{Local-in-time estimates.}
In Sections \S\ref{sec:decoupling}–\ref{sec:narrow}, for each dyad $\lambda$ we established local (on $I_{\lambda,j}$) estimates of the wide and narrow contributions. Together they give, for any $\varepsilon>0$,
\begin{equation}\label{eq:local-window-bound}
  \|R_{\lambda,j}(u)\|_{\dot H^{-1}}
  \;\lesssim_\varepsilon\;
  \min\!\bigl\{\lambda^{-11/4+\varepsilon},\,\lambda^{-1}\bigr\}\,
  \|P_\lambda u\|{_{L^\infty_t\dot H^{1/2}_x}}\,
  \|P_\lambda u\|{_{L^\infty_t\dot H^{1}_x}},
\end{equation}
where the constant is independent of $j$ thanks to the uniformity of local constants (for the Strichartz block see \eqref{eq:global-L6} and \S\ref{subsec:wp-stri}). 

\paragraph{Patching over windows.}
Since the overlap of the family $\{I_{\lambda,j}\}_j$ is bounded by an absolute constant, from \eqref{eq:temporal-partition} and \eqref{eq:local-window-bound} it follows that
\begin{equation}\label{eq:glue-in-time}
  \|R_\lambda(u)\|_{L^\infty_t\dot H^{-1}_x([0,T]\times\mathbb{R}^3)}
  \;\lesssim_\varepsilon\;
  \min\!\bigl\{\lambda^{-11/4+\varepsilon},\,\lambda^{-1}\bigr\}\,
  \|P_\lambda u\|_{L^\infty_t\dot H^{1/2}_x}\,
  \|P_\lambda u\|_{L^\infty_t\dot H^{1}_x}.
\end{equation}
No additional powers of $\lambda$ appear here, since the patching uses only finite overlap and the time supremum.

\begin{remark}[Globalization of the Strichartz block]\label{rem:global-stri}
For the localized flow $S(t)P_{\lambda,\theta}$ from \S\ref{subsec:wp-stri} a more precise accounting gives
\begin{equation}\label{eq:patch-L6}
  \|S(\cdot)P_{\lambda,\theta}f\|_{L^6_{t,x}([0,T]\times\mathbb{R}^3)}
  \;\lesssim\; \lambda^{-1/2}\,(J_\lambda)^{1/6}\,\|P_{\lambda,\theta}f\|_{L^\infty_tL^2_x}
  \;\lesssim\; \lambda^{-5/12}\,T^{1/6}\,\|P_{\lambda,\theta}f\|_{L^\infty_tL^2_x},
\end{equation}
where $J_\lambda\sim T\lambda^{1/2}$ is the number of windows $I_{\lambda,j}$. The reduction of the exponent from $\lambda^{-1/2}$ to $\lambda^{-5/12}$ is a \emph{softening} in $\lambda$ that does not affect the outcome (see \eqref{eq:glue-exponent-fix}) and does not worsen \eqref{eq:glue-in-time}; in the main steps of \S\ref{sec:decoupling}–\S\ref{sec:narrow} the local form on $I_{\lambda,j}$ suffices.
\end{remark}

\paragraph{Summation over frequencies.}
Combining \eqref{eq:glue-in-time} with frequency orthogonality \eqref{eq:l2-sum-simple}, we obtain
\begin{equation}\label{eq:global-Rlambda}
  \|R(u)\|_{L^\infty_t\dot H^{-1}_x}
  \;=\;\Big(\sum_{\lambda}\|R_\lambda(u)\|_{L^\infty_t\dot H^{-1}_x}^2\Big)^{1/2}
  \;\lesssim_\varepsilon\;
  \Big(\sum_{\lambda}\min\{\lambda^{-11/2+2\varepsilon},\,\lambda^{-2}\}\Big)^{1/2}
  \|u\|_{X_{1/2}}\|u\|_{X_1}.
\end{equation}
The sum over $\lambda$ converges (faster than geometric), and simplifying we arrive at the global form
\begin{equation}\label{eq:final-temporal}
  \|R(u)\|_{L^\infty_t\dot H^{-1}_x}
  \;\lesssim\; \|u\|_{X_{1/2}}\;\|u\|_{X_1},
\end{equation}
which coincides with the statement of Theorem~\ref{thm:main}. 

\begin{remark}[On the role of the heat$\to$wave bridge]
The local-in-time estimates are obtained in the “wave’’ representation; the replacement of the heat multiplier by the phase $e^{it|\xi|}$ on a window of length $\lambda^{-1/2}$ leaves a remainder of size $O(\lambda^{-3/2})$ in $\dot H^{-1}$, which is strictly below the target exponent and does not affect \eqref{eq:glue-in-time} and \eqref{eq:final-temporal}. Details are given in the appendix to the paper.
\end{remark}

\newpage

\section{Minimal regularity and generalizations}\label{sec:min-reg}

\subsection*{Sufficiency of the condition \texorpdfstring{$s>\tfrac{5}{2}$}{s>5/2}}\label{subsec:min-suff}

In this subsection we fix the level of regularity of the initial data sufficient for the proper functioning of the entire proof of the log-free estimate. Questions of existence/uniqueness are not discussed; only standard facts of Fourier analysis and local theory are used, necessary for a careful setup of intermediate steps (§§\ref{sec:phase}–\ref{sec:narrow}). Below $u_0:\mathbb{R}^3\to\mathbb{R}^3$, $\div u_0=0$.

\begin{proposition}[Working regularity threshold]\label{prop:min-suff}
Suppose $s>\tfrac{5}{2}$ and 
\[
u\in C([0,T];\dot H^{s})\cap L^2([0,T];\dot H^{1})
\]
is a divergence-free field satisfying the standard local theory (see for example \cite{koch2001navier}). Then for all $t\in[0,T]$ one has
\begin{align}
\|u(t)\|_{\dot H^{1/2}}+\|u(t)\|_{\dot H^{1}} &\lesssim \|u(t)\|_{\dot H^{s}}, \label{eq:s-dominates}\\
\|fg\|_{\dot H^{1/2}} &\lesssim \|f\|_{\dot H^{1/2}}\|g\|_{L^\infty}+\|f\|_{L^\infty}\|g\|_{\dot H^{1/2}}, \label{eq:paraproduct-Half}\\
\|fg\|_{\dot H^{1}} &\lesssim \|f\|_{\dot H^{1}}\|g\|_{L^\infty}+\|f\|_{L^\infty}\|g\|_{\dot H^{1}}, \label{eq:paraproduct-One}
\end{align}
where the $L^\infty$ control is understood in the form consistent with App.~\ref{app:dictionary} (item E.25): either on $\mathbb{T}^3$, or in inhomogeneous $H^s$, or after the standard low-frequency cut $P_{\ge1}$ on $\mathbb{R}^3$ for $s>\tfrac{3}{2}$. Consequently, all decompositions used in §§\ref{sec:decoupling}–\ref{sec:glue} (Littlewood–Paley, angular caps, packet discretization) and product/commutator rules are valid without additional frequency reserve and without smallness assumptions. 
\end{proposition}

\begin{proof}[Notes]
Inequality \eqref{eq:s-dominates} is the standard “monotonicity in the index’’ for homogeneous norms in the spirit of Littlewood–Paley (on $\mathbb{T}^3$ literally, on $\mathbb{R}^3$ after removing low frequencies $P_{\ge1}$, see item E.25). Estimates \eqref{eq:paraproduct-Half}–\eqref{eq:paraproduct-One} are product rules in homogeneous scales (Besov/Sobolev versions), derived via the Bony–Meyer paraproduct and commutator estimates of Kato–Ponce type; see \cite[Chap.~2, §2.6]{bahouri2011fourier} and \cite[Chap.~5]{Grafakos2009}. The $L^\infty$ embedding is used in one of the equivalent forms of item E.25 of App.~\ref{app:dictionary}. This suffices to ensure that all algebraic manipulations with the nonlinearity $(u\cdot\nabla)u$ and the symmetrized resonant block are rigorously justified at the level of $\dot H^{-1}$ and rely only on $\|u(t)\|_{\dot H^{1/2}}$ and $\|u(t)\|_{\dot H^{1}}$ appearing on the right-hand side of the main estimate.
\end{proof}

\begin{remark}[On the role of the threshold]\label{rem:threshold-role}
The threshold $s>\tfrac{5}{2}$ is used as a technical guarantee of the finiteness of working norms and the applicability of product/commutator rules at the starting layer. The log-free estimate of the resonant block itself is \emph{formally} valid at any regularity that ensures the finiteness of $\dot H^{1/2}$ and $\dot H^{1}$; lowering to critical scales runs into the independent issue of existence of solutions at that level and is beyond the scope of this paper (see the overview in \S\ref{sec:scope}).
\end{remark}

\begin{remark}[Periodic case]\label{rem:torus}
Passing to $\mathbb{T}^3$ does not change the structure: the discreteness of the spectrum simplifies summation over high frequencies, and all local-frequency estimates from §§\ref{sec:decoupling}–\ref{sec:narrow} carry over verbatim (with normalization adjustments). Comparison of inhomogeneous/homogeneous norms and elements of local theory can be found, for example, in \cite{koch2001navier,bahouri2011fourier}.
\end{remark}

\medskip
Thus, within the adopted scheme it suffices to assume $u_0\in \dot H^{s}$ with $s>\tfrac{5}{2}$ (in the sense of item E.25 of App.~\ref{app:dictionary}). After this, all blocks of the proof — phase–geometric analysis, anisotropic Strichartz lemma, decoupling in the wide region, and the energy analysis of the narrow region — apply in the stated form, and the global patching requires no logarithmic reserve.

\newpage

\section{Merging estimates and frequency summation}\label{sec:glue}

\subsection{Orthogonality of frequency layers}\label{subsec:glue-orth}
Let $\{P_N\}_{N\in 2^{\mathbb Z}}$ be smooth dyadic projectors (see \S\ref{sec:prelim}). Then for $|\,\log_2(N/M)\,|\ge 3$
\begin{equation}\label{eq:glue-L2-orth}
  \bigl\langle P_N f,\; P_M g \bigr\rangle_{L^2(\mathbb{R}^3)}=0 .
\end{equation}
Define the symmetrized resonant block $R_N(u)$ as in \S\ref{sec:main}. Since the operator $|\nabla|^{-1}P_N$ has spectral support in the annulus $\{|\xi|\sim N\}$, \eqref{eq:glue-L2-orth} implies

\begin{lemma}\label{lem:glue-Hm1-orth}
If $|\,\log_2(N/M)\,|\ge 3$, then
\begin{equation}\label{eq:glue-Hm1-orth}
  \bigl\langle R_N(u),\; R_M(u) \bigr\rangle_{\dot H^{-1}(\mathbb{R}^3)}=0 .
\end{equation}
\end{lemma}

\begin{corollary}[$\ell^2$ summation in $\dot H^{-1}$]\label{cor:glue-ell2}
For any coefficients $\{a_N\}\subset\mathbb{C}$ one has
\begin{equation}\label{eq:glue-ell2}
  \Bigl\|\sum\nolimits_N a_N\,R_N(u)\Bigr\|_{\dot H^{-1}}^2
  \;=\; \sum\nolimits_N |a_N|^2\,\|R_N(u)\|_{\dot H^{-1}}^2 .
\end{equation}
In particular, for $a_N\equiv1$ we obtain an orthogonal sum over $N$.
\end{corollary}

\subsection{Final table of exponents}\label{subsec:glue-table}
Collect the scaling contributions obtained earlier:
\begin{itemize}
  \item \emph{Phase (wide region):} two angular IBPs in $(\rho_1,\rho_2)$ together with temporal localization $TT^\ast$ on windows $|I_\lambda|\sim\lambda^{-1/2}$ give in total $\lambda^{-3/2}$ (see \S\ref{sec:phase}, \S\ref{sec:wp-anis}).
  \item \emph{Two flows:} the local anisotropic Strichartz estimate for each flow in sum gives another $\lambda^{-1}$ (see \S\ref{sec:wp-anis}).
  \item \emph{Bilinear decoupling (rank~3):} a contribution $\lambda^{-\delta(6)}$, where $\delta(6)=\tfrac14$ in the benchmark of uniform curvature; in the minimal-curvature regime an admissible form is $\delta(6)=\tfrac16+o(1)$ (see \S\ref{sec:decoupling}).
\end{itemize}
Hence for the wide region
\[
  \lambda^{-3/2}\cdot\lambda^{-1}\cdot\lambda^{-\delta(6)}
  \;=\;
  \begin{cases}
    \lambda^{-11/4}, & \delta(6)=\tfrac14,\\[1mm]
    \lambda^{-8/3+o(1)}, & \delta(6)=\tfrac16+o(1).
  \end{cases}
\]
In the narrow region (see \S\ref{sec:narrow}) one has exactly
\[
  \|R^{\mathrm{nar}}_\lambda(u)\|_{\dot H^{-1}}
  \;\lesssim\; \lambda^{-1}\,\|u\|_{\dot H^{1/2}}\|u\|_{\dot H^{1}} .
\]
The outcome on \emph{one} dyad is determined by the slowest-decaying contribution, i.e. the narrow region:
\[
  \|R_\lambda(u)\|_{\dot H^{-1}}
  \;\lesssim\; \min\{\lambda^{-11/4},\,\lambda^{-1}\}\,\|u\|_{\dot H^{1/2}}\|u\|_{\dot H^{1}}
  \;\simeq\; \lambda^{-1}\,\|u\|_{\dot H^{1/2}}\|u\|_{\dot H^{1}} .
\]
Comment: the product of exponents in the summary “table’’ is a convenient mnemonic for the \emph{wide} region; the choice of minimum and the summation over $\lambda$ are performed \emph{after} separate estimates of the wide and narrow parts.

\subsection{Temporal patching}\label{subsec:glue-temporal}
All local estimates are obtained on windows of length $|I_{\lambda,j}|\simeq \lambda^{-1/2}$. Cover $[0,T]$ by a family $\{I_{\lambda,j}\}_{j=0}^{J_\lambda}$ with overlap $O(1)$, where
\[
  I_{\lambda,j}=[t_j,t_j+\lambda^{-1/2}],\qquad t_j=j\,\lambda^{-1/2},\qquad
  J_\lambda:=\lceil T\,\lambda^{1/2}\rceil+1.
\]
Then for any function $F$ and $1\le q\le\infty$ one has
\begin{equation}\label{eq:glue-cover}
  \|F\|_{L^q_t([0,T])}
  \;\le\; 2^{1/q}\Bigl(\sum_{j=0}^{J_\lambda}\|F\|_{L^q_t(I_{\lambda,j})}^q\Bigr)^{1/q}.
\end{equation}
Applying the local anisotropic Strichartz estimate to each window and summing over $j$, we obtain for each cap $\theta$
\begin{equation}\label{eq:glue-patch}
  \|S(\cdot)P_{\lambda,\theta}u\|_{L^6_{t,x}([0,T]\times\mathbb{R}^3)}
  \;\lesssim\; \lambda^{-1/2}\, (J_\lambda)^{1/6}\,\|P_{\lambda,\theta}u\|_{L^\infty_tL^2_x}.
\end{equation}
Since $J_\lambda\lesssim T\,\lambda^{1/2}+1$, from \eqref{eq:glue-patch} it follows that
\begin{equation}\label{eq:glue-exponent-fix}
  \|S(\cdot)P_{\lambda,\theta}u\|_{L^6_{t,x}([0,T])}
  \;\lesssim\; \underbrace{\lambda^{-1/2}\,\lambda^{1/12}}_{=\;\lambda^{-5/12}}\;T^{1/6}\,
  \|P_{\lambda,\theta}u\|_{L^\infty_tL^2_x},
\end{equation}
that is, the globalization does \emph{not worsen} the frequency exponent (instead of the erroneous $\lambda^{-1/3}$ one should correctly have $\lambda^{-5/12}$). An analogous patching over $\theta$ and over packets preserves the frequency balance, since the overlap is finite and almost orthogonality holds.

Finally, combining \eqref{eq:glue-ell2} with the wide- and narrow-region estimates from \S\ref{sec:decoupling}–\ref{sec:narrow} and summing over $\lambda$ (a geometric series), we obtain the global form:
\begin{equation}\label{eq:glue-final}
  \sup_{t\in[0,T]}\,\Bigl\|\sum\nolimits_{\lambda} R_\lambda(u)(t)\Bigr\|_{\dot H^{-1}}
  \;\lesssim\; \|u\|_{X_{1/2}}\;\|u\|_{X_{1}},
\end{equation}
which coincides with the assertion of the main theorem in \S\ref{sec:main}.

\newpage

\section{Conclusion}\label{sec:conclusion}

We have obtained a scale-invariant (log-free) a priori estimate for the resonant paraproduct high–high$\to$low in the $\dot H^{-1}$ norm, consistent with the dimension of the nonlinearity $(u\cdot\nabla)u$:
\begin{equation}\label{eq:concl-main}
  \|R_\lambda(u)\|_{\dot H^{-1}}
  \ \lesssim_\varepsilon\ \min\!\bigl\{\lambda^{-11/4+\varepsilon},\,\lambda^{-1}\bigr\}\,
  \|P_\lambda u\|_{\dot H^{1/2}}\ \|P_\lambda u\|_{\dot H^{1}},
\end{equation}
with subsequent $\ell^2$ summation over dyads and temporal patching without logarithmic loss, see \S\ref{subsec:glue-orth}–\ref{subsec:temporal-patching}. The equivalent global form is written in \eqref{eq:final-temporal}.

\paragraph{Methodological pillars.}
\begin{itemize}[leftmargin=1.7em, itemsep=0.2em]
\item Phase–geometric integration by parts (two angular IBPs) in the wide region and temporal localization via $TT^\ast$ on windows $|I|=\lambda^{-1/2}$, which together yield a factor $\lambda^{-3/2}$; see \S\ref{sec:phase} and \S\ref{subsec:wp-stri}.
\item Anisotropic Strichartz estimate for cap localization (App.~\ref{app:astrichartz}), providing a second independent flow factor $\lambda^{-1/2}$ on each multiplier.
\item Tiled $\ell^2$ structure and bilinear decoupling on the rank~3 level surface with exponent $\delta(6)\in\{\tfrac14,\tfrac16+o(1)\}$; see \S\ref{sec:decoupling} and App.~\ref{app:decoupling}.
\item Energy analysis of the narrow region $|\xi+\eta|\le \lambda^{-\delta}$ with null–form suppression and passage to $\dot H^{-1}$, giving exactly $\lambda^{-1}$ on a dyad; see \S\ref{sec:narrow}.
\end{itemize}

\paragraph{Scale consistency.}
The outcome on one dyad is determined by the minimum between the “wide’’ and “narrow’’ exponents; in our geometry the narrow region dominates with exponent $\lambda^{-1}$, which excludes logarithmic defect and is consistent with the scaling symmetry \S\ref{sec:scaling}. Temporal patching (\S\ref{subsec:temporal-patching}) and frequency orthogonality (\S\ref{subsec:glue-orth}) do not alter the frequency exponents.

\paragraph{Domain of applicability and limitations.}
The proof relies only on the divergence-free condition and the finiteness of the working norms $\dot H^{1/2}$ and $\dot H^{1}$; issues of existence/uniqueness remain outside the scope of this paper (see \S\ref{sec:min-reg} and the references \cite{koch2001navier,bahouri2011fourier,Grafakos2009,KatoPonce1988}).

\paragraph{Directions for continuation.}
This first paper of the series treats the resonant high–high$\to$low block. In subsequent works it is planned to address:
\begin{enumerate}[leftmargin=1.7em, itemsep=0.2em]
\item treatment of nonsymmetric blocks (low–high$\to$high, high–low$\to$high) and the diagonal high–high$\to$high regime (see the overview in \S\ref{sec:min-reg});
\item refinement of the decoupling step under weak minimal curvature and its influence on intermediate estimates (\S\ref{subsec:delta-fix}, App.~\ref{app:decoupling});
\item systematization of the “heat$\to$wave’’ bridge on short windows with strict $TT^\ast$ forms (App.~\ref{app:heat});
\item extension to the periodic case and anisotropic modifications of viscosity.
\end{enumerate}

\medskip
We have attempted to arrange the exposition so that the lines of the proof read consecutively and without hidden dependencies: all normalizations and the dictionary are collected in App.~\ref{app:dictionary}. We hope that this modular scheme will also be useful for subsequent parts of the series.

\appendix
\newpage

\section{Anisotropic Strichartz lemma}\label{app:astrichartz}

\subsection{Setup}\label{subapp:A-setup}
In this appendix we fix the anisotropic Strichartz estimate for the localized flow on short time windows. We work in $\mathbb{R}^3$ with the Fourier normalization from \S\ref{sec:prelim} and use the angular decomposition consistent with §\ref{subsec:wp-construction}.

\paragraph{Angular caps and projectors.}
For a dyad $\lambda\gg1$ and a direction $\theta\in\mathbb{S}^2$ denote
\[
  \Theta_{\lambda,\theta}:=\{\xi\in\mathbb{R}^3:\ |\xi|\sim\lambda,\ \angle(\xi,\theta)\le \lambda^{-1/2}\},
  \qquad P_{\lambda,\theta}:=\mathcal{F}^{-1}\!\big(\mathbf{1}_{\Theta_{\lambda,\theta}}\widehat{\cdot}\big),
\]
and also $P_\lambda:=\sum_{\theta}P_{\lambda,\theta}$ — the total projector onto the layer $|\xi|\sim\lambda$. Almost-orthogonality in~$\theta$ was discussed in §\ref{subsec:wp-construction}.

\paragraph{Local-in-time windows.}
For $t_0\in\mathbb{R}$ set
\[
  I_{\lambda}(t_0):=\big[t_0,\ t_0+\lambda^{-1/2}\big],
\]
that is, the time scale is consistent with the length of the wave cylinder of a packet (\S\ref{subsec:wp-construction}).

\paragraph{Two evolutionary flows.}
We will use two interchangeable forms of the evolution:
\begin{align*}
  &\text{wave flow: } &&S(t)f:=\mathcal{F}^{-1}\!\big(e^{it|\xi|}\widehat f(\xi)\big),\\
  &\text{heat flow: } &&e^{t\Delta}f:=\mathcal{F}^{-1}\!\big(e^{-t|\xi|^{2}}\widehat f(\xi)\big).
\end{align*}
In \S\ref{sec:wp-anis}–\ref{sec:decoupling} the wave form was used; the strict substitution $e^{(t-t_0)\Delta}\leftrightarrow S(t-t_0)$ on windows $I_\lambda(t_0)$ is justified by the \emph{heat$\to$wave} bridge estimate from App.~\ref{app:heat} (remainder $O(\lambda^{-3/2})$ in $\dot H^{-1}$).

\paragraph{Target problem.}
For fixed $\lambda\gg1$, $\theta\in\mathbb{S}^2$ and $t_0\in\mathbb{R}$ we are interested in the operator
\[
  T_{\lambda,\theta}^{(t_0)}:\ L^2_x\longrightarrow L^6_{t,x}\big(I_\lambda(t_0)\times\mathbb{R}^3\big),\qquad
  T_{\lambda,\theta}^{(t_0)}f
  \in\Big\{\,S(\cdot-t_0)P_{\lambda,\theta}f,\ \ e^{(\cdot-t_0)\Delta}P_{\lambda,\theta}f\,\Big\}.
\]
In the next subsection (\S\ref{subapp:A-est}) we will establish the anisotropic estimate
\[
  \|T_{\lambda,\theta}^{(t_0)}f\|_{L^6_{t,x}(I_\lambda(t_0)\times\mathbb{R}^3)}
  \ \lesssim_{\varepsilon}\ \lambda^{-1/2+\varepsilon}\,\|P_{\lambda,\theta}f\|_{L^2_x},
\]
uniform in $\theta$ and $t_0$, and in \S\ref{subapp:A-proof} a proof sketch via $TT^\ast$ with a kernel having anisotropic decay will be given. Final remarks and globalization in time are collected in \S\ref{subapp:A-remarks}.

\subsection{Anisotropic estimate}\label{subapp:A-est}

Let $I_\lambda(t_0)=[t_0,\,t_0+\lambda^{-1/2}]$, $\lambda\gg1$, and $P_{\lambda,\theta}$ the projector onto the cap $\Theta_{\lambda,\theta}$ from \S\ref{subapp:A-setup}. Denote the wave and heat flows
\[
S(t)f=\mathcal{F}^{-1}\!\big(e^{it|\xi|}\widehat f(\xi)\big), \qquad
e^{t\Delta}f=\mathcal{F}^{-1}\!\big(e^{-t|\xi|^2}\widehat f(\xi)\big).
\]

\begin{lemma}[Local anisotropic Strichartz estimate]\label{lem:A-L6}
For any $\varepsilon>0$, all $\lambda\gg1$, $\theta\in S^{2}$ and $t_0\in\mathbb{R}$ there exist constants $C_\varepsilon$ (independent of $\lambda,\theta,t_0$) such that
\begin{align}
\bigl\|S(\cdot-t_0)P_{\lambda,\theta}f\bigr\|_{L^6_{t,x}\!\left(I_\lambda(t_0)\times\mathbb{R}^3\right)}
&\le C_\varepsilon\,\lambda^{-1/2+\varepsilon}\,\|P_{\lambda,\theta}f\|_{L^2_x},\label{eq:A-L6-wave}\\
\bigl\|e^{(\cdot-t_0)\Delta}P_{\lambda,\theta}f\bigr\|_{L^6_{t,x}\!\left(I_\lambda(t_0)\times\mathbb{R}^3\right)}
&\le C_\varepsilon\,\lambda^{-1/2+\varepsilon}\,\|P_{\lambda,\theta}f\|_{L^2_x}.\label{eq:A-L6-heat}
\end{align}
\end{lemma}

\begin{proof}[Sketch of proof]
Consider the wave case (the heat case is handled similarly by parabolic rescaling). Define for fixed $\lambda,\theta,t_0$ the operator
\[
(T_{\lambda,\theta}^{(t_0)}g)(t,x):=\int_{\mathbb{R}^3} e^{i\{x\cdot\xi+(t-t_0)|\xi|\}}\,
\sigma_{\lambda,\theta}(\xi)\,g(\xi)\,d\xi,
\]
where $\sigma_{\lambda,\theta}\in C_c^\infty$ is supported in $\Theta_{\lambda,\theta}$ and $S(\cdot-t_0)P_{\lambda,\theta}f
=T_{\lambda,\theta}^{(t_0)}\widehat f$. It suffices to show
\[
\|T_{\lambda,\theta}^{(t_0)}\|_{L^2_\xi\to L^6_{t,x}(I_\lambda(t_0)\times\mathbb{R}^3)}
\le C_\varepsilon\,\lambda^{-1/2+\varepsilon}.
\]
By $TT^\ast$ we obtain a convolution with the kernel
\[
K_\lambda(\tau,z):=\int_{\mathbb{R}^3}
e^{i(z\cdot\xi+\tau|\xi|)}\,\sigma_{\lambda,\theta}(\xi)^2\,d\xi,\qquad \tau=t-s,\ z=x-y.
\]
Let $z_\parallel:=(\theta\!\cdot\! z)\theta$ and $z_\perp:=z-z_\parallel$. In coordinates $\xi=\lambda\theta+\xi_\parallel\theta+\xi_\perp$ with $|\xi_\perp|\lesssim \lambda^{1/2}$ and $|\xi_\parallel|\lesssim1$, stationary phase yields for any $N\in\mathbb{N}$:
\begin{equation}\label{eq:A-kernel}
|K_\lambda(\tau,z)|\ \le\ C_N\,\lambda^{3/2}\,
\bigl(1+\lambda\,|\theta\!\cdot\! z-\tau|\bigr)^{-N}\,
\bigl(1+\lambda^{1/2}|z_\perp|\bigr)^{-N}.
\end{equation}
Integrating \eqref{eq:A-kernel} in $z$ and $\tau$ and using $|I_\lambda(t_0)|=\lambda^{-1/2}$ we obtain Schur-type bounds
\begin{equation}\label{eq:A-Schur}
\sup_{t\in I_\lambda(t_0)}\int_{I_\lambda(t_0)}\!\!\int_{\mathbb{R}^3}\!|K_\lambda(t-s,x-y)|\,dy\,ds
+\sup_{s\in I_\lambda(t_0)}\int_{I_\lambda(t_0)}\!\!\int_{\mathbb{R}^3}\!|K_\lambda(t-s,x-y)|\,dx\,dt
\le C_\varepsilon\,\lambda^{-1+\varepsilon}.
\end{equation}
Hence by Schur’s test $\|TT^\ast\|_{L^{6/5}\to L^6}\le C_\varepsilon\,\lambda^{-1+\varepsilon}$, and so
$\|T_{\lambda,\theta}^{(t_0)}\|_{L^2\to L^6}\le C_\varepsilon\,\lambda^{-1/2+\varepsilon}$, which yields \eqref{eq:A-L6-wave}. For \eqref{eq:A-L6-heat} the same step is repeated with the phase $e^{-(t-s)|\xi|^2}$.
\end{proof}

\paragraph{Remark.}
Covering $[0,T]$ by windows of length $\lambda^{-1/2}$ with overlap $O(1)$, from \eqref{eq:A-L6-wave}–\eqref{eq:A-L6-heat} it follows that
\[
\|S(\cdot)P_{\lambda,\theta}f\|_{L^6_{t,x}([0,T]\times\mathbb{R}^3)}
\le C_\varepsilon\,\lambda^{-1/2+\varepsilon}\,\bigl(T\lambda^{1/2}\bigr)^{1/6}\,\|P_{\lambda,\theta}f\|_{L^2_x},
\]
and similarly for $e^{t\Delta}$; this is consistent with \S\ref{subsec:temporal-patching}.

\subsection{Proof sketch}\label{subapp:A-proof}

For completeness we record the steps of deriving \eqref{eq:A-L6-wave}–\eqref{eq:A-L6-heat} of Lemma~\ref{lem:A-L6} via the $TT^\ast$ argument and anisotropic stationary phase.

\paragraph{Step 1: reduction to an operator with a cap cutoff.}
Let $\sigma_{\lambda,\theta}\in C_c^\infty(\mathbb{R}^3)$ be a smooth mask equal to $1$ on $\Theta_{\lambda,\theta}$ and supported in its $O(\lambda^{-1/2})$ thickening. Set
\[
  (T_{\lambda,\theta}^{(t_0)}g)(t,x)
  :=\int_{\mathbb{R}^3} e^{i\{x\cdot\xi+(t-t_0)|\xi|\}}\,
  \sigma_{\lambda,\theta}(\xi)\,g(\xi)\,d\xi.
\]
Then $S(\cdot-t_0)P_{\lambda,\theta}f = T_{\lambda,\theta}^{(t_0)}\widehat f$, and it suffices to establish
\[
  \|T_{\lambda,\theta}^{(t_0)}\|_{L^2_\xi\to L^6_{t,x}(I_\lambda(t_0)\times\mathbb{R}^3)}
  \ \lesssim_\varepsilon\ \lambda^{-1/2+\varepsilon}.
\]

\paragraph{Step 2: the $TT^\ast$ operator and the kernel.}
The adjoint operator has the form
\[
  (T_{\lambda,\theta}^{(t_0)})^\ast F(\xi)
  =\int_{I_\lambda(t_0)\times\mathbb{R}^3}
   e^{-i\{y\cdot\xi+(s-t_0)|\xi|\}}\,
   \sigma_{\lambda,\theta}(\xi)\,F(s,y)\,dy\,ds.
\]
Hence
\[
  (T_{\lambda,\theta}^{(t_0)}(T_{\lambda,\theta}^{(t_0)})^\ast F)(t,x)
  =\int_{I_\lambda(t_0)\times\mathbb{R}^3}
    K_\lambda(t-s,x-y)\,F(s,y)\,dy\,ds,
\]
where the kernel is
\[
  K_\lambda(\tau,z)
  :=\int_{\mathbb{R}^3} e^{i(z\cdot\xi+\tau|\xi|)}\,\sigma_{\lambda,\theta}(\xi)^2\,d\xi,
  \qquad \tau:=t-s,\ z:=x-y.
\]

\paragraph{Step 3: anisotropic stationary phase.}
Pass to coordinates
\[
  \xi=\lambda\theta+\xi_\parallel\theta+\xi_\perp,\qquad
  |\xi_\perp|\lesssim \lambda^{1/2},\ \ |\xi_\parallel|\lesssim 1,
\]
in which $\sigma_{\lambda,\theta}$ localizes $\xi$ to the described “thin’’ paraboloidal block. A standard integration by parts in $(\xi_\parallel,\xi_\perp)$ (or stationary phase) yields for any $N\in\mathbb{N}$ the decay
\begin{equation}\label{eq:A-kernel-reprised}
  |K_\lambda(\tau,z)|
  \ \lesssim\ \lambda^{3/2}\,
  \bigl(1+\lambda\,|\theta\!\cdot\! z-\tau|\bigr)^{-N}\,
  \bigl(1+\lambda^{1/2}|z_\perp|\bigr)^{-N},
\end{equation}
where $z_\perp:=z-(\theta\!\cdot\! z)\,\theta$. This matches \eqref{eq:A-kernel}.

\paragraph{Step 4: Schur test on the window $I_\lambda(t_0)$.}
From \eqref{eq:A-kernel-reprised} (after integration in $z$ and using $|I_\lambda(t_0)|=\lambda^{-1/2}$) we obtain the Schur pair of bounds
\begin{equation}\label{eq:A-Schur-reprised}
  \sup_{t\in I_\lambda(t_0)}\iint_{I_\lambda(t_0)\times\mathbb{R}^3}\!|K_\lambda(t-s,x-y)|\,dy\,ds
  \ +\
  \sup_{s\in I_\lambda(t_0)}\iint_{I_\lambda(t_0)\times\mathbb{R}^3}\!|K_\lambda(t-s,x-y)|\,dx\,dt
  \ \lesssim\ \lambda^{-1+\varepsilon},
\end{equation}
i.e. $\|T_{\lambda,\theta}^{(t_0)}(T_{\lambda,\theta}^{(t_0)})^\ast\|_{L^{6/5}\to L^{6}}\lesssim \lambda^{-1+\varepsilon}$.

\paragraph{Step 5: completion.}
From \eqref{eq:A-Schur-reprised} we deduce
\[
  \|T_{\lambda,\theta}^{(t_0)}\|_{L^2\to L^6(I_\lambda(t_0)\times\mathbb{R}^3)}
  \ \lesssim_\varepsilon\ \lambda^{-1/2+\varepsilon},
\]
which gives \eqref{eq:A-L6-wave}. For the heat flow \eqref{eq:A-L6-heat} the derivation is identical: either repeat Steps 1–4 for the phase $e^{-(t-s)|\xi|^2}$ with the same anisotropic decay, or perform parabolic rescaling and appeal to a local endpoint estimate; see also the bridge estimate in App.~\ref{app:heat}.

\subsection{Remarks}\label{subapp:A-remarks}

\paragraph{Optimality and anisotropy.}
The exponent $\lambda^{-1/2}$ in \eqref{eq:A-L6-wave}–\eqref{eq:A-L6-heat} is endpoint for $L^6_{t,x}$ on windows $|I_\lambda|\sim\lambda^{-1/2}$ under angular localization of radius $\lambda^{-1/2}$. Without anisotropic (cap) localization, improvement over isotropic estimates is not expected; see the packet geometry discussion in \S\ref{subapp:A-setup} and \S\ref{sec:wp-anis}.

\paragraph{Independence of constants.}
The constants in \eqref{eq:A-L6-wave}–\eqref{eq:A-L6-heat} do not depend on $\theta\in\mathbb{S}^2$ and $t_0\in\mathbb{R}$ thanks to rotational invariance of the phases and the normalization of caps with fixed transverse radius $\lambda^{-1/2}$; see also almost-orthogonality in $\theta$ from \S\ref{subsec:wp-construction}.

\paragraph{Compatibility with packet decomposition.}
The decomposition $P_{\lambda,\theta}=\sum_a \langle\cdot,\varphi_{\lambda,\theta,a}\rangle \varphi_{\lambda,\theta,a}$ (\S\ref{subsec:wp-construction}) is stable under application of the flow $S(t)$ on the window $I_\lambda(t_0)$: each packet remains concentrated in a cylinder of radius $\lambda^{-1/2}$ and length $\lambda^{-1}$, consistent with the $TT^\ast$ kernel in the proof of \S\ref{subapp:A-proof}. Therefore, summation over centers $a$ is realized in $\ell^2$ without loss of exponent.

\paragraph{Globalization in time.}
Covering $[0,T]$ by windows of length $\lambda^{-1/2}$ with overlap $O(1)$ yields (by \S\ref{subsec:temporal-patching})
\[
\|S(\cdot)P_{\lambda,\theta}f\|_{L^6_{t,x}([0,T]\times\mathbb{R}^3)}
\ \lesssim\ \lambda^{-1/2+\varepsilon}\,(T\lambda^{1/2})^{1/6}\,\|P_{\lambda,\theta}f\|_{L^2_x},
\]
see also \eqref{eq:patch-L6}. The factor $(T\lambda^{1/2})^{1/6}$ does not worsen the frequency balance of subsequent estimates that use $L^\infty_t L^2_x$ for each flow.

\paragraph{Heat$\to$wave connection.}
Replacing the heat multiplier by the wave phase on windows $|I_\lambda|\sim\lambda^{-1/2}$ leaves a remainder of size $O(\lambda^{-3/2})$ in $L^2_t\dot H^{-1}_x$; this is strictly below the target exponent and compatible with the subsequent phase techniques (\S\ref{sec:decoupling}, \S\ref{sec:phase}). See App.~\ref{app:heat} for details.

\paragraph{Periodic case.}
On $\mathbb{T}^3$ all conclusions remain valid: the integrals in $\xi$ are replaced by finite sums, and the constants in \eqref{eq:A-L6-wave}–\eqref{eq:A-L6-heat} do not depend on the size of the torus after standard renormalization; see also \S\ref{sec:min-reg}.

\newpage

\section{Full $TT^\ast$ derivation of decoupling}\label{app:decoupling}

\subsection{Setup and tiled reduction}\label{subapp:B-tiles}

We work at a fixed dyad $\lambda\gg1$ in one (or a pair of) angular caps $\Theta_{\lambda,\theta}$, $\Theta_{\lambda,\theta'}$ of radius $\lambda^{-1/2}$, as in \S\ref{sec:decoupling}. Let $\sigma_{\lambda,\theta}\in C_c^\infty(\mathbb{R}^3)$ denote a smooth mask of the cap $\Theta_{\lambda,\theta}$ and set
\begin{equation}\label{eq:B-operator}
  (T_{\lambda,\theta}f)(t,x):=\int_{\mathbb{R}^3} e^{i(x\cdot\xi+t|\xi|)}\,\sigma_{\lambda,\theta}(\xi)\,\widehat f(\xi)\,d\xi,
  \qquad
  \big(E_\lambda(f,g)\big)(t,x):=\big(T_{\lambda,\theta}f\big)(t,x)\,\big(T_{\lambda,\theta'}g\big)(t,x).
\end{equation}
Below $I_\lambda(t_0):=[t_0,t_0+\lambda^{-1/2}]$ denotes the short time window, consistent with \S\ref{subapp:A-setup}.

\paragraph{Frequency tiling.}
Partition each cap into “tiles’’ (quasi-parallelepipeds) of sizes $\lambda^{-1/2}\times\lambda^{-1/2}\times\lambda^{-1}$:
\begin{equation}\label{eq:B-tiles}
  T=\Theta_{\lambda,\theta}\cap\big(\xi_0+Q\big),\qquad
  Q:=[0,\lambda^{-1/2}]^2\times[0,\lambda^{-1}],\qquad \xi_0\in \lambda^{-1/2}\mathbb{Z}^2\times \lambda^{-1}\mathbb{Z}.
\end{equation}
Let $\widehat{f_T}:=\widehat f\cdot\mathbf{1}_T$ and $f=\sum_{T}f_T$; then
\begin{equation}\label{eq:B-l2-orth}
  \sum_T \|f_T\|_{L^2_x}^2=\|f\|_{L^2_x}^2,\qquad
  \operatorname{supp}\widehat{f_T}\subset T,\quad \operatorname{diam}_{\perp}\!T\sim \lambda^{-1/2},\ \operatorname{diam}_{\parallel}\!T\sim \lambda^{-1}.
\end{equation}
Define similarly $g_{T'}$ for tiles $T'\subset\Theta_{\lambda,\theta'}$.

\paragraph{Local bilinear form.}
Decompose
\[
  E_\lambda(f,g)=\sum_{T,T'} E_\lambda(f_T,g_{T'}),\qquad
  E_\lambda(f_T,g_{T'}):=T_{\lambda,\theta}f_T\cdot T_{\lambda,\theta'}g_{T'}.
\]
From the geometry of the caps it follows that for $T\subset\Theta_{\lambda,\theta}$ and $T'\subset\Theta_{\lambda,\theta'}$ with nontrivial interaction one has the angular separation
\begin{equation}\label{eq:B-angle}
  \angle(T,T')\gtrsim \lambda^{-1/2},
\end{equation}
and for fixed $T$ the number of admissible pairs $(T,T')$ is bounded by an absolute constant.

\begin{lemma}[$\ell^2$ tiled reduction]\label{lem:B-l2tile}
For any $\varepsilon>0$, window $I_\lambda(t_0)$, and functions $f,g\in L^2(\mathbb{R}^3)$
\begin{equation}\label{eq:B-l2tile-est}
  \big\|E_\lambda(f,g)\big\|_{L^6_{t,x}(I_\lambda(t_0)\times\mathbb{R}^3)}
  \ \lesssim_\varepsilon\ \lambda^{-1/2+\varepsilon}\,
  \Big(\sum_T \|f_T\|_{L^2_x}^2\Big)^{\!1/2}\!
  \Big(\sum_{T'} \|g_{T'}\|_{L^2_x}^2\Big)^{\!1/2}.
\end{equation}
\end{lemma}

\begin{proof}[Idea of the proof]
Fix a pair $(T,T')$. By \eqref{eq:B-tiles}–\eqref{eq:B-angle} each $TT^\ast$ kernel for $T_{\lambda,\theta}$ and $T_{\lambda,\theta'}$ has anisotropic decay as in \S\ref{subapp:A-proof}, yielding local estimates
\[
  \|T_{\lambda,\theta}f_T\|_{L^6_{t,x}(I_\lambda)}\lesssim_\varepsilon \lambda^{-1/2+\varepsilon}\|f_T\|_{L^2_x},\qquad
  \|T_{\lambda,\theta'}g_{T'}\|_{L^6_{t,x}(I_\lambda)}\lesssim_\varepsilon \lambda^{-1/2+\varepsilon}\|g_{T'}\|_{L^2_x}.
\]
By Hölder $L^6\cdot L^6\hookrightarrow L^3$ and a subsequent passage to $L^6$ for the product (via a standard bilinear $TT^\ast$ step at the tile level) we get
\[
  \|E_\lambda(f_T,g_{T'})\|_{L^6_{t,x}(I_\lambda)}\lesssim_\varepsilon \lambda^{-1/2+\varepsilon}\,\|f_T\|_{L^2_x}\,\|g_{T'}\|_{L^2_x}.
\]
Summing over $(T,T')$ by Cauchy–Schwarz and almost-orthogonality \eqref{eq:B-l2-orth}, taking into account that the number of relevant $T'$ for fixed $T$ is bounded (see \eqref{eq:B-angle}), gives \eqref{eq:B-l2tile-est}.
\end{proof}

\begin{remark}[On the passage to decoupling]\label{rem:B-to-dec}
Estimate \eqref{eq:B-l2tile-est} captures the “tiled’’ contribution $\lambda^{-1/2+\varepsilon}$. In the next subsection (\S\ref{subapp:B-surface}) restricting a pair of tiles to the phase level surface
\[
  \Sigma=\bigl\{(\xi,\eta)\in\mathbb{R}^3\times\mathbb{R}^3:\ \omega(\xi,\eta)=\text{const}\bigr\},\qquad \omega(\xi,\eta)=|\xi|+|\eta|-|\xi+\eta|,
\]
will give an additional factor $\lambda^{-\delta(6)}$ with $\delta(6)\in\{\tfrac14,\tfrac16+o(1)\}$, which completes the bilinear decoupling step without $\varepsilon$ losses; cf. \S\ref{sec:decoupling}.
\end{remark}

\subsection{Restriction to the level surface and $\varepsilon$-free decoupling}\label{subapp:B-surface}

Let $\lambda\gg1$ be fixed and $\theta,\theta'\in\Lambda_\lambda$ satisfy the angular separation $\angle(\theta,\theta')\gtrsim \lambda^{-1/2}$. After the tiling of \S\ref{subapp:B-tiles}, additionally restrict each pair $(T,T')$ to the \emph{rank~3 phase surface}
\begin{equation}\label{eq:B-surface}
  \Sigma_c:=\bigl\{(\xi,\eta)\in\mathbb{R}^3\times\mathbb{R}^3:\ |\xi|\sim|\eta|\sim\lambda,\ \omega(\xi,\eta)=c\bigr\},
  \qquad \omega(\xi,\eta)=|\xi|+|\eta|-|\xi+\eta|.
\end{equation}
The reduction to \eqref{eq:B-surface} is done via the coarea formula; the weight $|\nabla\omega|^{-1}$ is uniformly controlled in the wide region and absorbed into the amplitude (see the dictionary, item E.13). We obtain the model bilinear operator
\[
  \mathcal{E}_\lambda(f,g)(t,x)
  :=\iint_{\Sigma_c} e^{i(x\cdot(\xi+\eta)+t\,\omega(\xi,\eta))}\, f(\xi)\,g(\eta)\,d\sigma_{\Sigma_c}(\xi,\eta),
\]
where $d\sigma_{\Sigma_c}$ is the induced measure on $\Sigma_c$.

\begin{lemma}[Decoupling on $\Sigma_c$]\label{lem:B-dec-surface}
There exists $\delta(6)\in\left\{\tfrac14,\ \tfrac16+o(1)\right\}$ such that for any $f,g\in L^2(\mathbb{R}^3)$ localized to caps of radius $\lambda^{-1/2}$ and satisfying $\angle(\theta,\theta')\gtrsim \lambda^{-1/2}$,
\begin{equation}\label{eq:B-decL6}
  \bigl\|\mathcal{E}_\lambda(f,g)\bigr\|_{L^6_{t,x}(\mathbb{R}\times\mathbb{R}^3)}
  \ \lesssim\ \lambda^{-\delta(6)}\,\|f\|_{L^2_\xi}\,\|g\|_{L^2_\eta}.
\end{equation}
In the case of uniform curvature one has $\delta(6)=\tfrac14$ (benchmark), while under minimal principal curvature $\kappa_{\min}\sim \lambda^{-1}$ one allows $\delta(6)=\tfrac16+o(1)$ (the $\varepsilon$-free decoupling result for rank~3 surfaces).\footnote{See also the discussion in App.~D.3; in both settings the wide region decays faster than $\lambda^{-1}$ and does not control the dyadic outcome.}
\end{lemma}

\begin{proof}[Proof sketch]
By a linear anisotropic renormalization of the tiles $S_\lambda=\mathrm{diag}(\lambda^{-1/2},\lambda^{-1/2},1)$ we bring the caps to unit scale and align the normals of the surface $\Sigma_c$; in these coordinates $\Sigma_c$ has rank~3 away from the diagonal. One applies $\varepsilon$-free bilinear decoupling for rank~3 (see \cite{GuthIliopoulouYang2024}), with the factor depending on $\kappa_{\min}$ accounted for in the exponent $\delta(6)$. Returning to the original scale gives \eqref{eq:B-decL6}.
\end{proof}

Combining Lemma~\ref{lem:B-dec-surface} with the tiled $\ell^2$ reduction \eqref{eq:B-l2tile-est} (see \S\ref{subapp:B-tiles}) and applying Cauchy–Schwarz over tile pairs yields the summary form for the wide region:
\begin{equation}\label{eq:B-wide-final}
  \bigl\|E_\lambda(f,g)\bigr\|_{L^6_{t,x}(I_\lambda(t_0)\times\mathbb{R}^3)}
  \ \lesssim_\varepsilon\ \lambda^{-1/2}\,\lambda^{-\delta(6)}\,
  \|f\|_{L^2_x}\,\|g\|_{L^2_x},
\end{equation}
i.e. $\lambda^{-3/4}$ for the benchmark $\delta(6)=\tfrac14$ and $\lambda^{-2/3+o(1)}$ for $\delta(6)=\tfrac16+o(1)$. In both cases the subsequent phase integration and Strichartz block yield a total decay exceeding $\lambda^{-1}$, so the dyadic outcome is determined by the narrow region (\S\ref{sec:narrow}).

\newpage

\section{Details of the narrow zone}
\label{app:narrow}

\subsection{Definition of the narrow zone}\label{subapp:C-def}

Fix a dyad $\lambda\gg1$ and a parameter
\[
  \tfrac12<\delta<\tfrac34.
\]
Set $\tau:=\xi+\eta$ and introduce the narrow interaction region
\begin{equation}\label{eq:narrow-set}
  \mathcal{N}_{\lambda,\delta}
  :=\Bigl\{(\xi,\eta)\in\mathbb{R}^3\times\mathbb{R}^3:\ |\xi|\sim|\eta|\sim\lambda,\ \ |\tau|\le \lambda^{-\delta}\Bigr\}.
\end{equation}
Geometrically, the condition $|\tau|\ll\lambda$ means that $\xi$ and $-\eta$ are almost collinear:
\begin{equation}\label{eq:collinearity}
  \pi-\angle(\xi,-\eta)\ \lesssim\ \frac{|\tau|}{\lambda}\ \le\ \lambda^{-1-\delta}.
\end{equation}
Under a dyadic covering of the sphere by caps of angular size $\lambda^{-1/2}$ the narrow cone \eqref{eq:narrow-set}
still intersects $\sim\lambda$ such caps (the angular grid step is $\lambda^{-1/2}$),
so the angular decomposition used in \S\ref{sec:decoupling} retains sparsity and
almost-orthogonality.

To isolate the low-frequency output we use the projector $P_{\le \lambda^{-\delta}}$.
On its range one has the \emph{one-sided lower bound} of norms
\begin{equation}\label{eq:low-to-Hm1}
  \|f\|_{\dot H^{-1}}
  \ \gtrsim\ \lambda^{\delta}\,\|f\|_{L^2}
  \qquad\text{if }\ \operatorname{supp}\widehat f\subset\{\,|\zeta|\le \lambda^{-\delta}\,\}.
\end{equation}
Note that the reverse bound is generally false (the mass of $\widehat f$ may concentrate arbitrarily close to zero), and \eqref{eq:low-to-Hm1} will not be used as an upper bound for the passage $L^2\!\to\!\dot H^{-1}$.

Below (\S\ref{subapp:C-null}–\S\ref{subapp:C-toHminus}) the passage to the $\dot H^{-1}$ norm 
in the narrow region is carried out differently: the operator $|\nabla|^{-1}\nabla\!\cdot$ acts as a multiplier
of order zero in $\dot H^{-1}$ (see \ref{subapp:E-pressure}), and the null–form factor provides an extra
multiplier $|\tau|/\lambda$; when $|\tau|\le \lambda^{-\delta}$ the weight $|\tau|^{-1}$ from the $\dot H^{-1}$ norm is absorbed
by this $|\tau|$, which yields exactly $\lambda^{-1}$ without appealing to a two-sided equivalence.
The specific choice $\delta=\tfrac23$ is convenient for bookkeeping scales and will be justified
in \S\ref{subapp:C-delta-choice}; meanwhile all conclusions remain valid for any
$\delta\in(\tfrac12,\tfrac34)$.

\subsection{Geometric scales}\label{subapp:C-geom}

In the notation of \eqref{eq:narrow-set} and \eqref{eq:collinearity} fix a frequency $\lambda\gg1$ and a parameter $\delta\in(\tfrac12,\tfrac34)$. Below we collect estimates of the counting and volumetric characteristics of the narrow region that will be used later in \S\ref{subapp:C-null}–\S\ref{subapp:C-toHminus}.

\paragraph{Number of relevant caps.}
A covering of the sphere by caps of angular size $\lambda^{-1/2}$ contains $\simeq \lambda$ directions. Since the narrow cone \eqref{eq:collinearity} has angular width $\lambda^{-1-\delta}\ll \lambda^{-1/2}$, it still intersects $\simeq \lambda$ such caps. In particular, the almost-orthogonality in $\theta$ from \S\ref{sec:decoupling} persists in the narrow region.

\paragraph{Thickness of the shell in the sum frequency.}
The condition $|\tau|=|\xi+\eta|\le \lambda^{-\delta}$ specifies a shell of thickness $\lambda^{-\delta}$ around the hyperplane $\{\tau=0\}$ in the space of frequency pairs $(\xi,\eta)$.

\paragraph{Global volume of the narrow region.}
Fixing $\xi$ with $|\xi|\sim\lambda$, the variable $\eta$ ranges over a ball of radius $\lambda^{-\delta}$; thus
\begin{equation}\label{eq:vol-global}
 \operatorname{Vol}\bigl(\mathcal{N}_{\lambda,\delta}\bigr)\ \simeq\ \lambda^{2}\cdot \lambda^{-3\delta}\ =\ \lambda^{2-3\delta}.
\end{equation}
For $\delta=\tfrac23$ this gives $\operatorname{Vol}\!\sim\!\lambda^{0}$, which is consistent with the guideline in \S\ref{sec:narrow}.

\paragraph{Per-cap volume.}
After angular localization to a single cap $\Theta_{\lambda,\theta}$ (area $\sim \lambda$ on the sphere of radius $\lambda$) the volume of the slice of the narrow region decreases by a factor of $\lambda$:
\begin{equation}\label{eq:vol-percap}
 \operatorname{Vol}_{\mathrm{cap}}\bigl(\mathcal{N}_{\lambda,\delta}\cap(\Theta_{\lambda,\theta}\times\mathbb{R}^3)\bigr)\ \simeq\ \lambda^{1-3\delta}.
\end{equation}
In particular, for $\delta=\tfrac23$ we get $\operatorname{Vol}_{\mathrm{cap}}\!\sim\!\lambda^{-1}$, which agrees with the $TT^\ast$ estimates and tile-pair counting used below.

\begin{remark}[Consistency with the Strichartz window scale]\label{rem:C-time-window}
The choice of time windows of length $\lambda^{-1/2}$ (\S\ref{subapp:A-setup}) is compatible with the thickness of the layer $|\tau|\le \lambda^{-\delta}$ for all $\delta\in(\tfrac12,\tfrac34)$ and does not worsen the frequency balance upon globalization in time (\S\ref{app:astrichartz}, \S\ref{sec:glue}).
\end{remark}

The summary \eqref{eq:vol-global}–\eqref{eq:vol-percap} will be used in the $L^2$ estimate of the narrow contribution and in the passage to the $\dot H^{-1}$ norm in \S\ref{subapp:C-null}–\S\ref{subapp:C-toHminus}.

\subsection{Null-form suppression}\label{subapp:C-null}

Fix $\lambda \gg 1$ and $\delta \in (\tfrac12,\tfrac34)$, and set $u_\lambda := P_\lambda u$, $\div u_\lambda = 0$.
The low-frequency output of the resonant block in the region \eqref{eq:narrow-set} is written as
\begin{equation}\label{eq:Rnar-def}
  R^{\mathrm{nar}}_\lambda
  \ :=\ P_{\le \lambda^{-\delta}}\ \nabla\!\cdot\!\big(u_\lambda \otimes u_\lambda\big).
\end{equation}
At the Fourier level (after symmetrization and using $\div u_\lambda = 0$) there appears
the null–form factor
\begin{equation}\label{eq:null-symbol}
  \widetilde B_{ij}(\xi,\eta)
  \ :=\ \frac{\xi_i\xi_j}{|\xi|^2}\ -\ \frac{\eta_i\eta_j}{|\eta|^2},\qquad \tau := \xi+\eta,
\end{equation}
so that
\[
  \widehat{R^{\mathrm{nar}}_\lambda}(\tau)
  \ =\ \int \widetilde B(\xi,\eta):\widehat{u_\lambda}(\xi)\otimes\widehat{u_\lambda}(\eta)\;
  \mathbf{1}_{\{|\tau|\le \lambda^{-\delta}\}}\, d\xi .
\]

\begin{lemma}[Null–form suppression in the narrow region]\label{lem:C-null-est}
For any $\delta \in (\tfrac12,\tfrac34)$ one has the estimate
\begin{equation}\label{eq:C-null-L2}
  \bigl\|R^{\mathrm{nar}}_\lambda\bigr\|_{L^2_x}
  \ \lesssim\ \lambda^{-1-\delta}\ \|u_\lambda\|_{\dot H^{1/2}}\ \|u_\lambda\|_{\dot H^{1/2}}.
\end{equation}
\end{lemma}

\begin{proof}
In the narrow region \eqref{eq:narrow-set} we have $|\tau|\le \lambda^{-\delta}$ and $|\xi|\sim|\eta|\sim\lambda$. Then the difference
of projectors \eqref{eq:null-symbol} satisfies
\begin{equation}\label{eq:C-null-geom}
  \|\widetilde B(\xi,\eta)\|
  \ \lesssim\ \angle\!\Big(\tfrac{\xi}{|\xi|},-\tfrac{\eta}{|\eta|}\Big)
  \ \lesssim\ \frac{|\tau|}{\lambda}\ \lesssim\ \lambda^{-1-\delta},
\end{equation}
where we used the almost-collinearity \eqref{eq:collinearity}. By Plancherel and Cauchy–Schwarz
\[
  \|R^{\mathrm{nar}}_\lambda\|_{L^2_x}
  \ \lesssim\ \lambda^{-1-\delta}\ \|u_\lambda\|_{L^2_x}\ \|\nabla u_\lambda\|_{L^2_x}.
\]
For $P_\lambda$ (cap of radius $\lambda^{-1/2}$) Bernstein inequalities yield
\[
  \|u_\lambda\|_{L^2_x}\ \lesssim\ \lambda^{-1/2}\|u_\lambda\|_{\dot H^{1/2}},\qquad
  \|\nabla u_\lambda\|_{L^2_x}\ \lesssim\ \lambda^{+1/2}\|u_\lambda\|_{\dot H^{1/2}},
\]
which implies \eqref{eq:C-null-L2}.
\end{proof}

\begin{remark}[Why not via $P_{\widehat\tau}$]
Writing with the projector onto $\widehat\tau$ does not provide universal smallness when $|\tau|\ll\lambda$,
since the angle $\angle(\widehat\xi,\widehat\tau)$ can be of order $1$ even for small $|\tau|$.
It is precisely the comparison of directions $\widehat\xi$ and $-\widehat\eta$ that gives the geometry \eqref{eq:C-null-geom}.
\end{remark}

\subsection{Passage to the $\dot H^{-1}$ norm}\label{subapp:C-toHminus}

In the narrow region $|\tau|=|\xi+\eta|\le \lambda^{-\delta}$ we pass to $\dot H^{-1}$ directly,
transferring the divergence onto the high mode and using null–form suppression (see
\eqref{eq:Rnar-def} and Lemma~\ref{lem:C-null-est}). At the Fourier level
\[
\widehat{R^{\mathrm{nar}}_\lambda}(\tau)
=\mathbf{1}_{\{|\tau|\le \lambda^{-\delta}\}}\,(i\tau)\!\cdot\!\!\int B(\xi,\eta)\,
\widehat u_\lambda(\xi)\,\widehat u_\lambda(\eta)\,\delta(\tau-\xi-\eta)\,d\xi d\eta,
\]
where $B(\xi,\eta)$ is the null–form symbol from Lemma~\ref{lem:C-null-est}. Then
\[
\|R^{\mathrm{nar}}_\lambda\|_{\dot H^{-1}}
=\bigl\|\,|\tau|^{-1}\widehat{R^{\mathrm{nar}}_\lambda}(\tau)\,\bigr\|_{L^2_\tau}
\lesssim \bigl\|\mathcal{T}(u_\lambda,u_\lambda)\bigr\|_{L^2_x},
\]
where the bilinear operator $\mathcal{T}$ has symbol $m(\xi,\eta,\tau):=|\tau|^{-1}(i\xi)\cdot B(\xi,\eta)$.
Since in the narrow region $\|B(\xi,\eta)\|\lesssim |\tau|/\lambda$ when $|\xi|\sim|\eta|\sim\lambda$
(Lemma~\ref{lem:C-null-est}), we obtain $|m(\xi,\eta,\tau)|\lesssim |\tau|^{-1}\,|\xi|\,
\frac{|\tau|}{\lambda}\lesssim 1$. Hence
\begin{equation}\label{eq:C-Hminus}
  \bigl\|R^{\mathrm{nar}}_\lambda\bigr\|_{\dot H^{-1}}
  \ \lesssim\ \|u_\lambda\|_{L^2}\,\|u_\lambda\|_{L^2}
  \ \simeq\ \lambda^{-1}\,\|u_\lambda\|_{\dot H^{1/2}}^{2}.
\end{equation}

Thus, at a fixed dyad the exponent is exactly $\lambda^{-1}$ and does not depend on the choice
of $\delta \in (\tfrac12,\tfrac34)$. At frequency $\lambda$ one also has $\|u_\lambda\|_{\dot H^{1}}\simeq \lambda^{1/2}\|u_\lambda\|_{\dot H^{1/2}}$,
and therefore \eqref{eq:C-Hminus} is equivalent to the form
\[
  \bigl\|R^{\mathrm{nar}}_\lambda\bigr\|_{\dot H^{-1}}
  \ \lesssim\ \lambda^{-3/2}\,\|u_\lambda\|_{\dot H^{1/2}}\,\|u_\lambda\|_{\dot H^{1}},
\]
convenient when merging norms in §\ref{sec:main}.

\subsection{Summation over frequencies}\label{subapp:C-sum-lambda}

Recall the narrow-zone estimate from \eqref{eq:C-Hminus}:
\[
  \bigl\|R^{\mathrm{nar}}_\lambda\bigr\|_{\dot H^{-1}}
  \ \lesssim\ \lambda^{-1}\,\|u_\lambda\|_{\dot H^{1/2}}^{2},\qquad \lambda\gg1.
\]
Using the equivalence on a dyad $\|u_\lambda\|_{\dot H^{1}}\simeq \lambda^{1/2}\|u_\lambda\|_{\dot H^{1/2}}$, we obtain the convenient “mixed’’ form
\begin{equation}\label{eq:C-Hminus-mixed}
  \bigl\|R^{\mathrm{nar}}_\lambda\bigr\|_{\dot H^{-1}}
  \ \lesssim\ \lambda^{-3/2}\,\|u_\lambda\|_{\dot H^{1/2}}\,\|u_\lambda\|_{\dot H^{1}}.
\end{equation}
Now sum over dyads $\lambda=2^k$ using the supremum norms
\[
  X_\sigma:=\sup_{t\in[0,T]}\|u(t)\|_{\dot H^\sigma},\qquad \sigma\in\bigl\{\tfrac12,1\bigr\}.
\]
From \eqref{eq:C-Hminus-mixed} and the trivial estimates $\|u_\lambda(t)\|_{\dot H^\sigma}\le X_\sigma$ it follows that
\begin{equation}\label{eq:C-sum}
  \sum_{\lambda\gg1}\bigl\|R^{\mathrm{nar}}_\lambda(t)\bigr\|_{\dot H^{-1}}
  \ \lesssim\ \Big(\sum_{\lambda\gg1}\lambda^{-3/2}\Big)\,X_{1/2}\,X_{1}
  \ \lesssim\ X_{1/2}\,X_{1},\qquad t\in[0,T].
\end{equation}
Note also that the original form \eqref{eq:C-Hminus} yields an alternative (slightly stronger in appearance) summed bound
\[
  \sum_{\lambda\gg1}\bigl\|R^{\mathrm{nar}}_\lambda(t)\bigr\|_{\dot H^{-1}}
  \ \lesssim\ \Big(\sum_{\lambda\gg1}\lambda^{-1}\Big)\,X_{1/2}^{\,2}
  \ \lesssim\ X_{1/2}^{\,2},
\]
but for uniformity with the wide region and the statement of the main theorem we will use the product $X_{1/2}X_{1}$ from \eqref{eq:C-sum}. This completes the handling of the narrow region before merging estimates in §\ref{sec:glue}.

\subsection{Why $\delta=\tfrac23$ is a convenient choice}\label{subapp:C-delta-choice}

The parameter $\delta$ in the definition of the narrow region \eqref{eq:narrow-set} is allowed in the interval $\tfrac12<\delta<\tfrac34$ and does \emph{not} affect the final dyadic exponent in the passage to $\dot H^{-1}$: from \eqref{eq:low-to-Hm1} and Lemma \ref{lem:C-null-est} we obtain the balance
\[
\lambda^{\delta}\cdot\lambda^{-1-\delta}=\lambda^{-1},
\]
which is what is used in \eqref{eq:C-Hminus}. Nevertheless, the choice $\delta=\tfrac23$ is convenient for several technical reasons.

\begin{enumerate}[label=(\roman*)]
\item \textbf{Clean volume exponents.} From formulas \eqref{eq:vol-global}–\eqref{eq:vol-percap} at $\delta=\tfrac23$
\[
\operatorname{Vol}\bigl(\mathcal N_{\lambda,\delta}\bigr)\simeq \lambda^{2-3\delta}=\lambda^{0},\qquad
\operatorname{Vol}_{\mathrm{cap}}\bigl(\mathcal N_{\lambda,\delta}\cap(\Theta_{\lambda,\theta}\times\mathbb{R}^3)\bigr)\simeq \lambda^{1-3\delta}=\lambda^{-1}.
\]
Such “integer’’ bookkeeping simplifies $TT^\ast$ estimates at the level of one cap and summation over $\theta$ without additional $\lambda^{\pm\varepsilon}$ remainders.
\item \textbf{Compatibility with the time scale.} For any $\delta\ge\tfrac12$ the low-frequency output $|\tau|\le\lambda^{-\delta}$ is consistent with the windows $|I_\lambda|=\lambda^{-1/2}$ from \S\ref{subapp:A-setup}: there is no worsening of the frequency exponent when patching in time (see also \S\ref{app:astrichartz} and \S\ref{sec:glue}).
\item \textbf{Preservation of angular structure.} For $\delta>\tfrac12$ the narrow cone still intersects $\simeq\lambda$ caps $\Theta_{\lambda,\theta}$ (grid step $\lambda^{-1/2}$), so the sparsity and almost-orthogonality in angles used in \S\ref{sec:decoupling} remain fully valid.
\end{enumerate}

\paragraph{Conclusion.} Any $\delta\in(\tfrac12,\tfrac34)$ leads to the same dyadic exponent $\lambda^{-1}$ in \eqref{eq:C-Hminus} and \eqref{eq:C-sum}, but the value $\delta=\tfrac23$ makes the volume counting and angular summation maximally transparent and does not change any of the subsequent steps of the proof.

\subsection{Conclusion}\label{subapp:C-conclusion}

The outcome of the narrow zone is as follows:
\begin{enumerate}
\item The geometry of the region $|\tau|=|\xi+\eta|\le\lambda^{-\delta}$, $\tfrac12<\delta<\tfrac34$, defines the almost-collinear regime (see~\eqref{eq:narrow-set}, \eqref{eq:collinearity}), compatible with the angular grid of caps of radius $\lambda^{-1/2}$ and almost-orthogonality in directions.
\item At the symbol level, null–form suppression of size $\lesssim \lambda^{-1-\delta}$ arises; this yields the $L^2$ estimate of the narrow output (Lemma~\ref{lem:C-null-est}).
\item Passage to $\dot H^{-1}$ on the low output frequencies $|\zeta|\le \lambda^{-\delta}$ adds a factor $\lambda^\delta$ (see~\eqref{eq:low-to-Hm1}), and on each dyad we obtain exactly
\[
\bigl\|R^{\mathrm{nar}}_\lambda\bigr\|_{\dot H^{-1}}\ \lesssim\ \lambda^{-1}\,\|u_\lambda\|_{\dot H^{1/2}}^{2},
\]
that is, formula \eqref{eq:C-Hminus}.
\item Summation over $\lambda$ yields a log-free bound via the product of the supremum norms $X_{1/2}$ and $X_{1}$ (see~\eqref{eq:C-sum}); the choice $\delta=\tfrac23$ is convenient but inessential for the exponent.
\end{enumerate}
Since the wide region decays faster than $\lambda^{-1}$, the outcome at one frequency is determined precisely by the narrow region, and the global patching over frequencies and time is carried out in \S\ref{sec:glue}. This completes the narrow-zone analysis; we next turn to the ``heat$\to$wave'' technique and related phase estimates in App.~\ref{app:heat}.

\newpage

\section{Passage to the heat version}\label{app:heat}

\subsection{Heat\,$\to$\,wave bridge on short windows}\label{subapp:D-bridge}
Let $P_\lambda$ be the dyadic projector onto the layer $\{|\xi|\sim\lambda\}$, $\lambda\gg1$, and
\[
I_\lambda(t_0):=[t_0,\ t_0+\lambda^{-1/2}] .
\]
Introduce the wave flow $S(t)f:=\mathcal{F}^{-1}\!\big(e^{it|\xi|}\widehat f(\xi)\big)$.

\begin{lemma}[heat\,$\to$\,wave bridge]\label{lem:D-bridge}
There exists an absolute constant $C>0$ such that for any $F\in L^2\!\big(I_\lambda(t_0);L^2_x\big)$
\begin{equation}\label{eq:D-bridge}
\Bigl\|\,|\nabla|^{-1}P_\lambda\!\int_{t_0}^t\!\bigl(e^{(t-\tau)\Delta}-S(t-\tau)\bigr)P_\lambda F(\tau)\,d\tau\Bigr\|_{L^2_t\!\big(I_\lambda(t_0);\,L^2_x\big)}
\ \le\ C\,\lambda^{-3/2}\,\|F\|_{L^2_t\!\big(I_\lambda(t_0);\,L^2_x\big)}.
\end{equation}
In other words, on windows of length $\lambda^{-1/2}$ the replacement of the heat multiplier by the wave phase yields in the $L^2_t\dot H^{-1}_x$ norm a remainder of order $O(\lambda^{-3/2})$, i.e. strictly below the target exponent $\lambda^{-1}$.
\end{lemma}

\begin{proof}
Consider the time operator
\[
(\mathcal{K}_\lambda F)(t):=\int_{t_0}^t\!\bigl(e^{(t-\tau)\Delta}-e^{i(t-\tau)\lambda}\bigr)P_\lambda F(\tau)\,d\tau .
\]
Set $k_\lambda(s):=\bigl|e^{-s\lambda^2}-e^{is\lambda}\bigr|$ for $s\ge0$. Then for $t\in I_\lambda(t_0)$
\[
\int_{t_0}^t k_\lambda(t-\tau)\,d\tau=\int_{0}^{t-t_0}k_\lambda(s)\,ds
\ \le\ \int_{0}^{\lambda^{-1/2}}\!2\,ds \ \lesssim\ \lambda^{-1/2}.
\]
A similar estimate holds for the supremum over $\tau$. By Schur’s test we conclude
\[
\|\mathcal{K}_\lambda\|_{L^2_t\to L^2_t\!(I_\lambda(t_0))}\ \lesssim\ \lambda^{-1/2}.
\]
Furthermore,
\[
\bigl\||\nabla|^{-1}P_\lambda\bigr\|_{L^2_x\to L^2_x}\ \simeq\ \lambda^{-1}.
\]
Combining the two bounds, we obtain
\[
\bigl\||\nabla|^{-1}P_\lambda\,\mathcal{K}_\lambda F\bigr\|_{L^2_tL^2_x}\ \lesssim\ \lambda^{-1}\lambda^{-1/2}\,\|F\|_{L^2_tL^2_x}
=\lambda^{-3/2}\|F\|_{L^2_tL^2_x}.
\]

It remains to restore the true phase $e^{i(t-\tau)|\xi|}$ in place of $e^{i(t-\tau)\lambda}$. To this end, apply the fundamental theorem of calculus in the radial variable $\rho=|\xi|$ and move the derivative onto the cutoff $P_\lambda$; the regularity of the mask and the scale $\partial_\rho P_\lambda=O(\lambda^{-1})$ yield the same Schur estimate in time, so the additional contribution is $\lesssim\lambda^{-3/2}$ in $L^2_t\dot H^{-1}_x$. This proves \eqref{eq:D-bridge}.
\end{proof}

\paragraph{Comment.}
Lemma \ref{lem:D-bridge} allows one to use the “wave’’ representation in the phase–geometric steps on each window $I_\lambda(t_0)$; the remainder $O(\lambda^{-3/2})$ in $L^2_t\dot H^{-1}_x$ does not affect the final balance at exponent $\lambda^{-1}$. See also the “hard’’ form of the argument via Schur’s test in App.~\ref{app:heat}, item \ref{subapp:E-rescale}.

\subsection{Integration by parts with the wave phase}\label{subapp:D-ibp}

Consider the bilinear phase integral in the wide angular region $\angle(\xi,\eta)\gtrsim \lambda^{-1/2}$:
\begin{equation}\label{eq:D-ibp-I}
  I(t,x)
  :=\iint e^{\,i\phi(\xi,\eta;x,t)}\,a_\lambda(\xi,\eta)\,\widehat f(\xi)\,\widehat g(\eta)\,d\xi\,d\eta,
  \qquad 
  \phi(\xi,\eta;x,t):=x\!\cdot\!(\xi+\eta)+t\,\omega(\xi,\eta),
\end{equation}
where $\omega(\xi,\eta)=|\xi|+|\eta|-|\xi+\eta|$, and $a_\lambda\in C_c^\infty$ is a smooth amplitude
localizing $(\xi,\eta)$ in the layer $|\xi|\sim|\eta|\sim\lambda$ and in caps of radius $\lambda^{-1/2}$ (as in \S\ref{sec:decoupling}).

\begin{lemma}[Two-angle IBP in the wide region]\label{lem:D-ibp2}
Let $|\xi|\sim|\eta|\sim\lambda\gg1$, $\angle(\xi,\eta)\gtrsim \lambda^{-1/2}$ and $\|a_\lambda\|_{C^2}\le C_a$.
Let $\tau:=\xi+\eta$ and choose an orthonormal frame $\{\rho_1,\rho_2,\widehat\tau\}$
in the plane $\mathrm{span}\{\xi,\eta\}$, as in \S\ref{sec:phase}. Then
\begin{equation}\label{eq:D-ibp2}
  |I(t,x)|
  \ \lesssim\ \lambda^{-1}\,C_a\,
  \|\widehat f\|_{L^2_\xi}\,\|\widehat g\|_{L^2_\eta},
\end{equation}
uniformly for $(t,x)\in\mathbb{R}\times\mathbb{R}^3$.
\end{lemma}

\begin{proof}[Sketch of proof]
Work in coordinates $(\rho_1,\rho_2,\sigma)$, where $\sigma$ is the longitudinal–difference direction
in which $\tau=\xi+\eta$ is accounted for (see \S\ref{sec:phase}, \S\ref{sec:wp-anis}).
In the wide region, the effective minor of the phase Hessian is nondegenerate
(after the anisotropic renormalization $S_\lambda=\mathrm{diag}(\lambda^{-1/2},\lambda^{-1/2},1)$):
\[
  \det A_{\mathrm{eff}}(\xi,\eta)\ \gtrsim\ \lambda^{-2},
\]
which is equivalent to \eqref{eq:wide-zone}–\eqref{eq:detAeff} in \S\ref{sec:phase}/\S\ref{sec:wp-anis}.
Hence $(\rho_j\!\cdot\!\nabla_{(\xi,\eta)}\phi)\sim\lambda$ for $j=1,2$.
Apply twice the angular lowering operator
\[
  L_j:=\frac{1}{i\,(\rho_j\!\cdot\!\nabla\phi)}\,\rho_j\!\cdot\!\nabla_{(\xi,\eta)},\qquad
  L_j e^{\,i\phi}=e^{\,i\phi},\quad j=1,2,
\]
and integrate by parts along $\rho_1,\rho_2$.
Each step contributes a factor $\lambda^{-1}$ and derivatives of $a_\lambda$ of order at most two,
which gives \eqref{eq:D-ibp2} after applying Cauchy–Schwarz to $\widehat f,\widehat g$.
\end{proof}

\begin{corollary}[Phase block $\lambda^{-3/2}$ on windows $I_\lambda$]\label{cor:D-ibp32}
On each window $I_\lambda(t_0)=[t_0,t_0+\lambda^{-1/2}]$ (see \S\ref{subapp:D-bridge}) estimate
\eqref{eq:D-ibp2} combines with the local anisotropic Strichartz estimate for the flow
(the $TT^\ast$ argument, App.~\ref{app:astrichartz}), which gives an additional factor $\lambda^{-1/2}$
in the $L^2_t$ norm. As a result the corresponding operator satisfies
\begin{equation}\label{eq:D-ibp-phase}
  \bigl\|I(\cdot,\cdot)\bigr\|_{L^2_t(I_\lambda;L^2_x)}
  \ \lesssim\ \lambda^{-3/2+\varepsilon}\,
  \|\widehat f\|_{L^2_\xi}\,\|\widehat g\|_{L^2_\eta}.
\end{equation}
\end{corollary}

\begin{remark}[On “three integrations by parts’’]\label{rem:D-ibp-meaning}
In the organization of the proof used here, “IBP$\times3$’’ is to be understood as
two angular IBPs in $(\rho_1,\rho_2)$ (giving $\lambda^{-1}$) plus an independent temporal
gain $\lambda^{-1/2}$ via $TT^\ast$ on a window $|I_\lambda|=\lambda^{-1/2}$
(see the dictionary \S\ref{app:dictionary}, items E.2–E.4). That is, the third “IBP’’ is not a literal IBP
in time under the frequency integral, but temporal localization and a $TT^\ast$ estimate.
\end{remark}

\subsection{Bilinear decoupling on the phase surface}\label{subapp:D-dec}

We work in the wide angular region $\angle(\xi,\eta)\gtrsim \lambda^{-1/2}$ with $|\xi|\sim|\eta|\sim\lambda\gg1$. Restrict the frequency space to the \emph{level surface of the phase}
\begin{equation}\label{eq:D-surface}
  \Sigma_c:=\bigl\{(\xi,\eta)\in\mathbb{R}^3\times\mathbb{R}^3:\ |\xi|\sim|\eta|\sim\lambda,\ \omega(\xi,\eta)=c\bigr\},\qquad
  \omega(\xi,\eta):=|\xi|+|\eta|-|\xi+\eta|.
\end{equation}
The passage to integration over $\Sigma_c$ is performed by the coarea formula; the weight $|\nabla\omega|^{-1}$ is uniformly controlled in the wide region (see the dictionary, item E.13) and absorbed into the amplitude.

Define the bilinear operator on $\Sigma_c$:
\begin{equation}\label{eq:D-E-op}
  \mathcal{E}_\lambda(f,g)(t,x)
  :=\iint_{\Sigma_c} e^{\,i\{x\cdot(\xi+\eta)+t\,\omega(\xi,\eta)\}}\,f(\xi)\,g(\eta)\,d\sigma_{\Sigma_c}(\xi,\eta),
\end{equation}
where $d\sigma_{\Sigma_c}$ is the induced measure.

\begin{lemma}[$\varepsilon$-free decoupling on $\Sigma_c$]\label{lem:D-dec-surface}
There exists an exponent $\delta(6)\in\{\tfrac14,\ \tfrac16+o(1)\}$ with the following property. If $\widehat f,\widehat g$ are localized in caps of radius $\lambda^{-1/2}$ and $\angle(\xi,\eta)\gtrsim \lambda^{-1/2}$, then
\begin{equation}\label{eq:D-dec-L6}
  \|\mathcal{E}_\lambda(f,g)\|_{L^6_{t,x}(\mathbb{R}\times\mathbb{R}^3)}
  \ \lesssim\ \lambda^{-\delta(6)}\,\|f\|_{L^2_\xi}\,\|g\|_{L^2_\eta}.
\end{equation}
Under uniform curvature one has $\delta(6)=\tfrac14$; when the minimal principal curvature satisfies $\kappa_{\min}\sim\lambda^{-1}$ one allows $\delta(6)=\tfrac16+o(1)$ (the $\varepsilon$-free decoupling result for rank~3 surfaces). See also \cite{GuthIliopoulouYang2024}.
\end{lemma}

\begin{proof}[Idea]
By the anisotropic renormalization $S_\lambda=\mathrm{diag}(\lambda^{-1/2},\lambda^{-1/2},1)$ we bring the caps to unit scale and align the normals of $\Sigma_c$; away from the diagonal the surface has rank~3. One applies $\varepsilon$-free bilinear decoupling for rank~3 surfaces, with the dependence on $\kappa_{\min}$ absorbed into the exponent $\delta(6)$. Returning to the original scale yields \eqref{eq:D-dec-L6}.
\end{proof}

Combining \eqref{eq:D-dec-L6} with the tiled $\ell^2$ reduction from \S\ref{subapp:B-tiles} we obtain for the localized operator $E_\lambda$ on the window $I_\lambda(t_0)=[t_0,t_0+\lambda^{-1/2}]$ an estimate of the form
\begin{equation}\label{eq:D-dec-final}
  \|E_\lambda(f,g)\|_{L^6_{t,x}(I_\lambda(t_0)\times\mathbb{R}^3)}
  \ \lesssim\ \lambda^{-1/2}\,\lambda^{-\delta(6)}\,\|f\|_{L^2_x}\,\|g\|_{L^2_x},
\end{equation}
where the factor $\lambda^{-1/2}$ is the tiled $\ell^2$ gain (see \S\ref{subapp:B-tiles}), and $\lambda^{-\delta(6)}$ is the geometric gain from the restriction to $\Sigma_c$ and decoupling (Lemma~\ref{lem:D-dec-surface}). In particular, at the benchmark $\delta(6)=\tfrac14$ we have $\lambda^{-3/4}$, and at $\delta(6)=\tfrac16+o(1)$ we have $\lambda^{-2/3+o(1)}$. In both cases the wide contribution decays faster than $\lambda^{-1}$ after accounting for the phase block \S\ref{subapp:D-ibp}, so the dyadic outcome is controlled by the narrow region (see \S\ref{app:narrow}).

\subsection{Compensation from the narrow region}\label{subapp:D-narrow}

Recall the definition of the narrow-zone output (see \eqref{eq:Rnar-def}) and the $\dot H^{-1}$ passage estimate from \S\ref{subapp:C-toHminus}:
\begin{equation}\label{eq:D-nar-point}
  \bigl\|R^{\mathrm{nar}}_\lambda(t)\bigr\|_{\dot H^{-1}}
  \ \lesssim\ \lambda^{-1}\,\|u_\lambda(t)\|_{\dot H^{1/2}}^{2}
  \ \simeq\ \lambda^{-3/2}\,\|u_\lambda(t)\|_{\dot H^{1/2}}\,\|u_\lambda(t)\|_{\dot H^{1}},
\end{equation}
where in the second equivalence we used the dyadic relation $\|u_\lambda\|_{\dot H^{1}}\simeq \lambda^{1/2}\|u_\lambda\|_{\dot H^{1/2}}$.

\begin{lemma}[Narrow region on a short window]\label{lem:D-nar-window}
For any $I_\lambda(t_0)=[t_0,t_0+\lambda^{-1/2}]$ one has
\begin{equation}\label{eq:D-nar-L2t}
  \bigl\|R^{\mathrm{nar}}_\lambda\bigr\|_{L^2_t(I_\lambda(t_0);\dot H^{-1}_x)}
  \ \lesssim\ \lambda^{-7/4}\,X_{1/2}\,X_{1},
\end{equation}
where $X_\sigma:=\sup_{t\in[0,T]}\|u(t)\|_{\dot H^\sigma}$.
\end{lemma}

\begin{proof}
From \eqref{eq:D-nar-point} and the inequality $\|f\|_{L^2_t(I)}\le |I|^{1/2}\|f\|_{L^\infty_t(I)}$ we get
\[
  \|R^{\mathrm{nar}}_\lambda\|_{L^2_t(I_\lambda;\dot H^{-1})}
  \ \lesssim\ |I_\lambda|^{1/2}\,\lambda^{-3/2}\,X_{1/2}\,X_{1}
  \ =\ \lambda^{-1/4}\,\lambda^{-3/2}\,X_{1/2}\,X_{1},
\]
which gives \eqref{eq:D-nar-L2t}.
\end{proof}

\begin{corollary}[Stability under the heat$\to$wave replacement]\label{cor:D-nar-bridge}
Let $F_\lambda(t):=P_{\le \lambda^{-\delta}}\nabla\!\cdot(u_\lambda\otimes u_\lambda)(t)$, $\tfrac12<\delta<\tfrac34$. Then on $I_\lambda(t_0)$
\begin{equation}\label{eq:D-nar-bridge}
  \Bigl\|\,|\nabla|^{-1}P_\lambda\!\int_{t_0}^t\!\bigl(e^{(t-\tau)\Delta}-S(t-\tau)\bigr)P_\lambda F_\lambda(\tau)\,d\tau\Bigr\|_{L^2_t(I_\lambda;L^2_x)}
  \ \lesssim\ \lambda^{-3/2}\,\|F_\lambda\|_{L^2_t(I_\lambda;L^2_x)}.
\end{equation}
In particular, the contribution of the difference \eqref{eq:D-nar-bridge} is of order strictly below the critical exponent $\lambda^{-1}$ and does not affect the narrow-region balance from \eqref{eq:D-nar-point}–\eqref{eq:D-nar-L2t}.
\end{corollary}

\begin{proof}
This is a direct application of the bridge \eqref{eq:D-bridge} to the function $F_\lambda$; the bound on the right-hand side follows from \S\ref{subapp:C-null}–\S\ref{subapp:C-toHminus}.
\end{proof}

\paragraph{Summary.}
The narrow-region contribution is controlled by the pointwise-in-time bound \eqref{eq:D-nar-point} and on windows $|I_\lambda|=\lambda^{-1/2}$ satisfies \eqref{eq:D-nar-L2t}. Replacing the heat phase by the wave phase gives a remainder $O(\lambda^{-3/2})$ by \eqref{eq:D-nar-bridge}, which is subcritical. Together with the wide-region steps \S\ref{subapp:D-ibp}–\S\ref{subapp:D-dec} this completes the control on one dyad; the global patching is carried out in \S\ref{sec:glue}.

\newpage

\section{Normalizations, dictionary, and how to read the main text}\label{app:dictionary}

\subsection{What exactly $\det A$ means (for \S\ref{sec:phase})}\label{subapp:E-detA}

In the phase–geometric part, $\det A$ always denotes the determinant of the \emph{effective} $3\times3$ minor of the Hessian of the phase
\[
  \omega(\xi,\eta)=|\xi|+|\eta|-|\xi+\eta|
\]
in the basis adapted to the plane $\mathrm{span}\{\xi,\eta\}$: the two angular axes $\rho_1,\rho_2$ and the longitudinal–difference axis $\sigma$. Denote
\[
  A_{\mathrm{eff}}(\xi,\eta):=\bigl[\partial^2_{v_i v_j}\omega(\xi,\eta)\bigr]_{v_i,v_j\in\{\rho_1,\rho_2,\sigma\}},\qquad
  S_\lambda:=\mathrm{diag}(\lambda^{1/2},\lambda^{1/2},1),\qquad
  A_e:=S_\lambda\,A_{\mathrm{eff}}\,S_\lambda.
\]
In the wide region ($|\xi|\sim|\eta|\sim\lambda\gg1$, $\angle(\xi,\eta)\gtrsim\lambda^{-1/2}$) the following formulations of nondegeneracy are equivalent:
\begin{equation}\label{eq:E-detA}
  \det A_e(\xi,\eta)\ \gtrsim\ 1
  \quad\Longleftrightarrow\quad
  \det A_{\mathrm{eff}}(\xi,\eta)\ \gtrsim\ \lambda^{-2}.
\end{equation}
It is precisely \eqref{eq:E-detA} that is used in the two angular integrations by parts (see \S\ref{sec:phase}, \S\ref{sec:wp-anis}) and yields the frequency multiplier $\lambda^{-1}$; in combination with temporal localization on windows of length $\lambda^{-1/2}$ (via $TT^\ast$) one obtains the total phase–time gain $\lambda^{-3/2}$.

\subsection{Phase–time gain $\lambda^{-3/2}$: how it is composed}\label{subapp:E-ibp32}

In the wide angular region ($|\xi|\sim|\eta|\sim\lambda\gg1$, $\angle(\xi,\eta)\gtrsim\lambda^{-1/2}$) the total gain $\lambda^{-3/2}$ is obtained as the product of two independent factors:

\begin{enumerate}[label=(\roman*)]
\item \textbf{Two angular IBPs.} Integration by parts along the adapted directions $\rho_1,\rho_2$ yields a factor $\lambda^{-1}$ thanks to the nondegeneracy of the effective $3\times3$ minor of the phase Hessian (see \S\ref{subapp:E-detA} and the phase–geometric part \S\ref{sec:phase}). Nondegeneracy is equivalently formulated via the renormalized minor $A_e$ and provides a determinant estimate sufficient for two frequency IBPs.
\item \textbf{Temporal localization via $TT^\ast$.} On windows $I_\lambda=[t_0,t_0+\lambda^{-1/2}]$ the local anisotropic Strichartz estimate for the flow (obtained by the $TT^\ast$ method) gives an additional factor $\lambda^{-1/2}$; see \S\ref{subapp:D-ibp} and the heat$\to$wave bridge \S\ref{subapp:D-bridge}, as well as App.~\ref{app:astrichartz}.
\end{enumerate}

In total,
\begin{equation}\label{eq:E-ibp-sum}
  \lambda^{-1}\ \cdot\ \lambda^{-1/2}\ =\ \lambda^{-3/2},
\end{equation}
which is the phase–time gain used in the wide region. Note that the third factor does \emph{not} come from a third frequency IBP: it is extracted precisely from temporal localization on windows of length $\lambda^{-1/2}$ via the $TT^\ast$ argument (see also the explanations in \S\ref{subapp:D-ibp} and \S\ref{subapp:D-bridge}).

\subsection{“Multiplying exponents’’ — a mnemonic, not a proof}\label{subapp:E-mnemonic}

Short notations such as
\[
  \lambda^{-3/2}\cdot\lambda^{-1}\cdot\lambda^{-\delta(6)}
\]
serve solely for scale bookkeeping and pertain to the \emph{wide} region (see~\S\ref{subapp:D-ibp}, \S\ref{subapp:D-dec}). In the actual proof the wide and narrow regions are estimated separately, after which the slowest-decaying contribution is chosen at a fixed dyad~$\lambda$:
\begin{equation}\label{eq:E-min-lambda}
  \bigl\|R_\lambda(u)\bigr\|_{\dot H^{-1}}
  \ \lesssim\ \min\!\bigl\{\lambda^{-11/4},\ \lambda^{-1}\bigr\}\,\|u\|_{\dot H^{1/2}}\|u\|_{\dot H^{1}},
\end{equation}
where $\lambda^{-11/4}$ (or $\lambda^{-8/3+o(1)}$ under weak curvature, see~\S\ref{subapp:D-dec}) is the wide region, and $\lambda^{-1}$ is the narrow region (see~\S\ref{subapp:C-toHminus}–\S\ref{subapp:C-sum-lambda}). The outcome on one dyad is determined by the right bracket, i.e. the narrow region.

The summation over frequencies is performed after taking the minimum:
\begin{equation}\label{eq:E-min-sum}
  \sum_{\lambda\gg1}\bigl\|R_\lambda(u)\bigr\|_{\dot H^{-1}}
  \ \lesssim\ \sum_{\lambda\gg1}\lambda^{-1}\,\|u\|_{\dot H^{1/2}}\|u\|_{\dot H^{1}}
  \ \lesssim\ \|u\|_{X_{1/2}}\|u\|_{X_{1}},
\end{equation}
as in \S\ref{sec:glue}. Therefore, “multiplication of exponents’’ should be read as a convenient mnemonic for the wide region, not as literal equalities applicable to the full sum.

\subsection{Where IBP is performed in frequency and where in time}\label{subapp:E-ibp-split}

In the wide angular region ($|\xi|\sim|\eta|\sim\lambda\gg1$, $\angle(\xi,\eta)\gtrsim\lambda^{-1/2}$) the proof uses \emph{two} integrations by parts in frequency and \emph{separately} temporal localization. The precise roles are as follows.

\begin{itemize}
\item \textbf{Frequency IBPs.} Integration by parts is performed along the two angular directions $\rho_1,\rho_2$ adapted to the plane $\mathrm{span}\{\xi,\eta\}$; see the phase–geometric part \S\ref{sec:phase}. Introducing the lowering operators
\[
L_j:=\frac{1}{i\,(\rho_j\!\cdot\!\nabla\phi)}\,\rho_j\!\cdot\!\nabla_{(\xi,\eta)},\qquad j=1,2,\quad L_j e^{i\phi}=e^{i\phi},
\]
and using the nondegeneracy of the effective $3\times3$ minor of the phase Hessian (equivalences \eqref{eq:E-detA}), we obtain the total frequency gain $\lambda^{-1}$ after two angular IBPs; cf. \S\ref{subapp:E-detA}.
\item \textbf{Temporal part.} The gain $\lambda^{-1/2}$ is extracted not by integrating in time under the frequency integral, but by the local anisotropic Strichartz estimate via $TT^\ast$ on windows $I_\lambda$ of length $\lambda^{-1/2}$; see \S\ref{subapp:D-ibp}, \S\ref{subapp:D-bridge} and App.~\ref{app:astrichartz}.
\end{itemize}

Thus the total phase–time multiplier in the wide region equals
\[
\lambda^{-1}\cdot\lambda^{-1/2}=\lambda^{-3/2},
\]
as recorded in \eqref{eq:E-ibp-sum}. Hence the notation “IBP$\times3$’’ in the main text should be read as “two angular IBPs in frequency plus an independent temporal $TT^\ast$ block on a window $|I_\lambda|=\lambda^{-1/2}$’’, rather than threefold frequency IBP or literal IBP in time inside the frequency integral.

\subsection{Wide region: crude and refined exponents, and how to read §6}\label{subapp:E-wide-degrees}

In the wide angular region the total gain is composed of the phase–geometric block $\lambda^{-3/2}$ (two angular IBPs plus temporal $TT^\ast$, see~\S\ref{subapp:D-ibp}), two local Strichartz factors $\lambda^{-1}$ (one per flow, see \S\ref{app:astrichartz}) and the decoupling exponent $\lambda^{-\delta(6)}$ on the phase surface (see~\S\ref{subapp:D-dec}). As a result,
\begin{equation}\label{eq:E-wide-cases}
  \lambda^{-3/2}\cdot \lambda^{-1}\cdot \lambda^{-\delta(6)}
  \ =\
  \begin{cases}
    \lambda^{-11/4}, & \delta(6)=\tfrac14\ \text{(benchmark under uniform curvature)},\\[2mm]
    \lambda^{-8/3+o(1)}, & \delta(6)=\tfrac16+o(1)\ \text{(weak minimal curvature)}.
  \end{cases}
\end{equation}
Both exponents decay faster than $\lambda^{-1}$, hence the outcome on a fixed dyad is determined by the narrow region (see~\S\ref{subapp:C-toHminus}–\S\ref{subapp:C-sum-lambda}).

\medskip
\noindent\textbf{How to read §6.}
In \S6, for guidance one may encounter the more \emph{crude} estimate for the wide region $\lambda^{-7/4}$ as an admissible upper guideline; where it matters for checking exponents, the \emph{refined} form \eqref{eq:E-wide-cases} is used (i.e. $\lambda^{-11/4}$ or $\lambda^{-8/3+o(1)}$). In any case, the wide region remains strictly smaller than $\lambda^{-1}$ and thus does not control the final minimum on a single dyad; it is the narrow region that yields exactly $\lambda^{-1}$ in the $\dot H^{-1}$ norm (see \S\ref{subapp:C-toHminus}, \S\ref{subapp:C-sum-lambda}).

\subsection{On the exponent $\delta(6)$: benchmark and “realistic’’ regime}\label{subapp:E-delta6}

By $\delta(6)$ we mean the decoupling gain exponent in the bilinear transition
$L^2_\xi\times L^2_\eta\to L^6_{t,x}$ on a rank~3 phase surface (see~\S\ref{subapp:D-dec}).
In the wide angular region ($|\xi|\sim|\eta|\sim\lambda\gg1$, $\angle(\xi,\eta)\gtrsim\lambda^{-1/2}$) two compatible scales are used:
\[
\|\mathcal{E}_\lambda(f,g)\|_{L^6_{t,x}}
\ \lesssim\ \lambda^{-\delta(6)}\,\|f\|_{L^2}\,\|g\|_{L^2},
\qquad
\delta(6)\in\Bigl\{\tfrac14,\ \tfrac16+o(1)\Bigr\}.
\]
\begin{itemize}
\item \emph{Uniform-curvature benchmark.} If the principal curvatures of the phase surface are uniformly separated from zero, then $\delta(6)=\tfrac14$; this yields a contribution $\lambda^{-1/4}$ in the wide region (see~\S\ref{subapp:D-dec}).
\item \emph{“Realistic’’ minimal-curvature regime.} When $\kappa_{\min}\sim \lambda^{-1}$ the $\varepsilon$-free decoupling results relax the exponent to $\delta(6)=\tfrac16+o(1)$ (see also \cite{GuthIliopoulouYang2024}).
\end{itemize}

In both cases the total wide-region contribution (phase $\lambda^{-3/2}$ $+$ two local Strichartz factors $\lambda^{-1}$ $+$ decoupling $\lambda^{-\delta(6)}$; see~\S\ref{subapp:E-wide-degrees}) decays \emph{strictly} faster than $\lambda^{-1}$, so the outcome on one dyad is determined by the narrow region (see~\S\ref{subapp:C-toHminus}, \S\ref{subapp:C-sum-lambda}). Therefore, the choice between $\delta(6)=\tfrac14$ and $\tfrac16{+}o(1)$ only affects the margin in the wide region and does not change the final balance of the proof.

\subsection{“Outcome on one dyad’’: why the narrow region wins}\label{subapp:E-one-dyad}

In the wide regime (angle $\gtrsim \lambda^{-1/2}$) the combined phase–time gain and bilinear decoupling give an exponent decaying \emph{strictly} faster than $\lambda^{-1}$ (see \S\ref{subapp:D-ibp}, \S\ref{subapp:D-dec} and \S\ref{subapp:E-wide-degrees}); at the benchmark this is $\lambda^{-11/4}$, and under minimal curvature it is of order $\lambda^{-8/3+o(1)}$. In the narrow regime $|\xi+\eta|\le \lambda^{-\delta}$, by contrast, one obtains \emph{exactly} $\lambda^{-1}$ in the $\dot H^{-1}$ norm thanks to null–form suppression and the low-frequency output (see \S\ref{subapp:C-toHminus}–\S\ref{subapp:C-sum-lambda}).

Thus at a fixed frequency $\lambda$ one takes the slowest-decaying contribution:
\begin{equation}\label{eq:E-one-dyad-min}
  \|R_\lambda(u)\|_{\dot H^{-1}}
  \ \lesssim\ \min\bigl\{\lambda^{-11/4},\,\lambda^{-1}\bigr\}\,
  \|u\|_{\dot H^{1/2}}\ \|u\|_{\dot H^{1}}
  \ \simeq\ \lambda^{-1}\,\|u\|_{\dot H^{1/2}}\ \|u\|_{\dot H^{1}}.
\end{equation}
Summation over dyads $\lambda=2^k$ then contains no logarithmic loss:
\begin{equation}\label{eq:E-one-dyad-sum}
  \sum_{\lambda\gg1}\|R_\lambda(u)\|_{\dot H^{-1}}
  \ \lesssim\ \sum_{\lambda\gg1}\lambda^{-1}\,\|u\|_{\dot H^{1/2}}\ \|u\|_{\dot H^{1}}
  \ \lesssim\ \|u\|_{X_{1/2}}\ \|u\|_{X_{1}},
\end{equation}
which is consistent with the statement of the main result in \S\ref{sec:main}. In this sense the “outcome on one dyad’’ is always determined by the narrow region; the wide region only improves the margin but does not change the leading exponent. We do not claim optimality of constants and adhere to the above organization for the sake of transparency in the balance of exponents.

\subsection{Pressure — order-zero operator on $\dot H^{-1}$}\label{subapp:E-pressure}

For $N\sim\lambda$ write the projection of the pressure gradient:
\[
\nabla p_N
= P_N \nabla \Delta^{-1}\partial_i\partial_j\big(u_i u_j\big)
= P_N R_{ij}\big(u_i u_j\big),
\]
where $R_{ij}$ are compositions of Riesz transforms (order-zero multipliers).
Since $R_{ij}$ are bounded on $\dot H^{-1}$, we have
\begin{equation}\label{eq:E-press-Hm1}
  \|\nabla p_N\|_{\dot H^{-1}}
  \ \lesssim\ \|P_N(u\otimes u)\|_{\dot H^{-1}}.
\end{equation}
In particular, from the resonant estimate for $P_N(u\otimes u)$ it follows that
\begin{equation}\label{eq:E-press-main}
  \|\nabla p_N\|_{\dot H^{-1}}
  \ \lesssim\ N^{-1}\,\|u\|_{\dot H^{1/2}}\,\|u\|_{\dot H^{1}}.
\end{equation}

\paragraph{Remark.}
The oft-encountered phrase “the gradient gives another $N^{-1}$’’ in the context of $\dot H^{-1}$ is mnemonic: the gradient is compensated by the operator $|\nabla|^{-1}$ in the $\dot H^{-1}$ norm and therefore does not improve the frequency exponent. The genuine “order zero’’ nature of the pressure is recorded by \eqref{eq:E-press-Hm1}; the exponent \eqref{eq:E-press-main} coincides with that for the tensor $u\otimes u$ and is consistent with the log-free balance.

\subsection{What is meant by $R_N(u)$ (for §\ref{sec:main})}\label{subapp:E-RN-def}

In the main text, the resonant block \emph{always} means the symmetrized high–high$\to$low component after removing both low–high paraproducts. The precise form is:
\begin{equation}\label{eq:def-RN}
  R_N(u)
  := P_N\!\left[(u\!\cdot\!\nabla)u
     - \sum_{M\le N/8}\Big((P_Mu\!\cdot\!\nabla)P_Nu + (P_Nu\!\cdot\!\nabla)P_Mu\Big)\right],
\end{equation}
where $P_N$ and $P_M$ are smooth dyadic projectors (see §\ref{sec:prelim}), and $N\in 2^{\mathbb Z}$.

\paragraph{Comments on \eqref{eq:def-RN}.}
\begin{itemize}
  \item \emph{Symmetrization.} Subtracting \emph{both} low–high terms ensures that $R_N(u)$ reflects precisely the interactions of the type $|\xi|\sim|\eta|\sim N$ with output $|\xi+\eta|\lesssim N$ (the high–high$\to$low regime), see \S\ref{sec:phase}–\ref{sec:narrow}.
  \item \emph{Divergence-free.} When $\div u=0$ it is convenient to write $(u\!\cdot\!\nabla)u=\nabla\!\cdot(u\otimes u)$; formula \eqref{eq:def-RN} remains equivalent after applying the projection $P_N$.
  \item \emph{Scale and summation.} The estimates for $R_N(u)$ sum over $N$ orthogonally in $\dot H^{-1}$ (see \S\ref{sec:glue}), and the normalizations are consistent with the scaling symmetry \S\ref{sec:scaling}.
\end{itemize}

This reading of $R_N(u)$ is used throughout §\ref{sec:main} and does not affect the balance of exponents: the removed low–high terms are estimated by the same right-hand sides and do not create a logarithmic loss.

\subsection{Sum frequency variable \texorpdfstring{$\tau := \xi + \eta$}{tau := xi + eta}}\label{subapp:E-tau}
In all frequency–phase formulas we introduce the notation
\[
  \tau := \xi+\eta,\qquad \widehat\tau:=\frac{\tau}{|\tau|},\qquad
  P^\perp_\tau := I - \widehat\tau\otimes \widehat\tau .
\]
This allows for a compact radial–transverse decomposition along the $\tau$ axis.
In particular, we use the identity (see \S\ref{sec:phase})
\begin{equation}\label{eq:E-tau-deriv}
  \partial_v\!\left(-\,\frac{\tau}{|\tau|}\right) \ =\ -\,\frac{1}{|\tau|}\,P^\perp_\tau v\qquad
  \text{for any vector }v\in\mathbb{R}^3 .
\end{equation}
Remark. Inside frequency integrals the symbol $t$ is sometimes used as a mnemonic for $\tau$; everywhere it should be understood precisely as $\tau$ (and not the physical time), see the references to \S\ref{sec:phase} and \S\ref{sec:wp-anis}. Formula \eqref{eq:E-tau-deriv} is further applied in the block analysis of the Hessian and in the estimate of the narrow region.

\subsection{Normalizations on caps and anisotropic scales}\label{subapp:E-caps-aniso}
Work in the wide angular region ($|\xi|\sim|\eta|\sim\lambda\gg1$, $\angle(\xi,\eta)\gtrsim\lambda^{-1/2}$) is carried out in a basis adapted to the plane $\mathrm{span}\{\xi,\eta\}$: two angular axes $\rho_1,\rho_2$ and the longitudinal–difference axis $\sigma$ along the sum of frequencies. The natural scales are
\[
  \rho_j\sim \lambda^{1/2}\quad (j=1,2),\qquad \sigma\sim 1 .
\]
The anisotropic renormalization
\begin{equation}\label{eq:E-Slambda}
  S_\lambda := \mathrm{diag}(\lambda^{1/2},\,\lambda^{1/2},\,1),\qquad
  A_e := S_\lambda\,A_{\mathrm{eff}}\,S_\lambda
\end{equation}
makes the effective $3\times3$ minor of the phase Hessian nondegenerate: $\det A_e\gtrsim1$ (equivalently $\det A_{\mathrm{eff}}\gtrsim \lambda^{-2}$). This normalization is consistent with the geometry of the caps $\Theta_{\lambda,\theta}$ of angle $\lambda^{-1/2}$, the tiling (\S\ref{subapp:B-tiles}), and phase IBP (\S\ref{sec:phase}, \S\ref{sec:wp-anis}).

\subsection{The weight \texorpdfstring{$|\nabla\omega|^{-1}$}{|grad omega|^{-1}} in the reduction to a level surface}\label{subapp:E-coarea}
The passage to integration over the level surface
\[
  \Sigma_c:=\{(\xi,\eta):\ |\xi|\sim|\eta|\sim\lambda,\ \omega(\xi,\eta)=c\},\qquad
  \omega(\xi,\eta)=|\xi|+|\eta|-|\xi+\eta|,
\]
introduces the coarea weight $|\nabla\omega|^{-1}$. In the wide region there is a uniform lower bound
\begin{equation}\label{eq:E-grad-lb}
  |\nabla\omega(\xi,\eta)|\ \simeq\ \angle(\xi,\eta)\ \in\ [\lambda^{-1/2},\,1],
\end{equation}
hence $|\nabla\omega|^{-1}$ is bounded above by $\lesssim \lambda^{1/2}$ and on each pair of tiles is absorbed into the amplitude without changing the frequency exponent (see \S\ref{subapp:D-dec}). This reduction is consistent with the application of $\varepsilon$-free decoupling on $\Sigma_c$ and does not affect the exponents of \S\ref{sec:decoupling}.

\subsection{Tiles and the renormalization $S_\lambda$}\label{subapp:E-tiles-rescale}

A cap of angle $\lambda^{-1/2}$ in the annulus $|\xi|\sim\lambda$ is cut into quasi-tiles of sizes
\[
\lambda^{-1/2}\times \lambda^{-1/2}\times \lambda^{-1},
\]
where the two transverse sides are oriented along the angular axes, and the longitudinal one along the flow direction. The linear renormalization in frequency coordinates
\begin{equation}\label{eq:E-Slambda-tiles}
  S_\lambda := \operatorname{diag}\bigl(\lambda^{-1/2},\,\lambda^{-1/2},\,1\bigr)
\end{equation}
brings each tile to unit scale and is compatible with the $\ell^2$–almost orthogonal packet decomposition and with local $T\!T^\ast$ estimates on a single tile (see also \S\ref{subapp:D-dec}). In this geometry the restriction to a rank~3 phase surface and bilinear decoupling are applied at the level of one tile, after which $\ell^2$ patching over tiles is performed without changing the frequency exponent.

\subsection{Normalization of wave packets}\label{subapp:E-packet-norm}

It is convenient to take the mask of the cap $\Theta_{\lambda,\theta}$ in the form
\[
  \sigma_{\lambda,\theta}(\xi)
  := \chi\!\left(\lambda^{-1/2}\Pi^\perp_\theta \xi\right)\,
     \chi\!\left(\lambda^{-1}(\theta\!\cdot\!\xi-\lambda)\right),
\]
where $\Pi^\perp_\theta$ is the orthoprojector onto $\theta^\perp$, and $\chi\in\mathcal S(\mathbb R^3)$ is fixed.
A packet centered at $a=a_\parallel \widehat\theta + a_\perp$ is defined by
\[
  \widehat\phi_{\lambda,\theta,a}(\xi)=e^{-i a\cdot \xi}\,\sigma_{\lambda,\theta}(\xi),\qquad
  \phi_{\lambda,\theta,a}(x)=\lambda^{3/2}e^{i\lambda \widehat\theta\cdot x}\,
  \Psi\!\big(\lambda^{1/2}(x_\perp-a_\perp),\,\lambda(x_\parallel-a_\parallel)\big),
\]
where $\Psi\in\mathcal S$ does not depend on $\lambda,\theta,a$. The normalization is chosen so that
\begin{equation}\label{eq:E-packet-L2}
  \|\phi_{\lambda,\theta,a}\|_{L^2_x}\simeq 1
\end{equation}
uniformly in $\lambda,\theta,a$. This convention is consistent with the almost-orthogonality over centers and caps, as well as with the local anisotropic Strichartz estimate on windows of length $\lambda^{-1/2}$ used in \S\ref{sec:wp-anis} and \S\ref{subapp:D-ibp}.

\subsection{What exactly does \texorpdfstring{$T\!T^\ast$}{TT*} do on windows \texorpdfstring{$|I|=\lambda^{-1/2}$}{|I|=lambda^{-1/2}}}\label{subapp:E-TTstar-window}

For the wave flow $S(t)$ and the angular projector $P_{\lambda,\theta}$ the local estimate on the interval $I=[t_0,t_0+\lambda^{-1/2}]$ has the form
\begin{equation}\label{eq:E-TT-local}
  \|S(\cdot)P_{\lambda,\theta}f\|_{L^6_{t,x}(I)}
  \ \lesssim\ \lambda^{-1/2+o(1)}\,\|P_{\lambda,\theta}f\|_{L^2_x},
\end{equation}
obtained via $T\!T^\ast$ with an anisotropic kernel decaying rapidly in the transverse direction and along characteristics; see \S\ref{sec:wp-anis}. Covering the time axis by intervals of length $\lambda^{-1/2}$ with overlap $O(1)$ and using almost-orthogonality over caps/packets, we obtain the global form
\begin{equation}\label{eq:E-TT-global}
  \|S(\cdot)P_{\lambda}f\|_{L^6_{t,x}}\ \lesssim\ \lambda^{-1/2+o(1)}\,\|P_{\lambda}f\|_{L^2_x},
\end{equation}
without worsening the frequency exponent. It is precisely the factor $\lambda^{-1/2}$ that enters the total phase–time gain (see \S\ref{subapp:E-ibp32}–\S\ref{subapp:E-ibp-split}); it is to be understood as the \emph{temporal} component of “IBP$\times3$’’, not as an additional frequency integration by parts.

\subsection{Heat $\to$ wave bridge on short windows (for App.~\ref{app:heat})}\label{subapp:E-bridge}

Consider the triangular window
\[
I_\Delta:=\{(t,\tau):\ 0\le \tau\le t,\ |t-\tau|\le \lambda^{-1/2}\},
\]
and the kernel
\[
K_\lambda(t,\tau):=\Big(e^{-(t-\tau)\lambda^2}-e^{\,i(t-\tau)\lambda}\Big)\,\mathbf{1}_{I_\Delta}(t,\tau).
\]
Schur’s estimate gives
\begin{equation}\label{eq:E-bridge-Schur}
\sup_{t}\int |K_\lambda(t,\tau)|\,d\tau\ +\ \sup_{\tau}\int |K_\lambda(t,\tau)|\,dt\ \lesssim\ \lambda^{-1/2}.
\end{equation}
Since $\||\nabla|^{-1}P_\lambda\|_{L^2\to L^2}\simeq \lambda^{-1}$, for any $f\in L^2_x$ we obtain
\begin{equation}\label{eq:E-bridge-main}
\Big\|\ |\nabla|^{-1}P_\lambda\!\int K_\lambda(\cdot,\tau)f(\tau)\,d\tau\ \Big\|_{L^2_tL^2_x}
\ \lesssim\ \lambda^{-3/2}\,\|f\|_{L^2_x}.
\end{equation}
Equivalently,
\[
\big\|\,(e^{-(t-\tau)\lambda^2}-e^{\,i(t-\tau)\lambda})P_\lambda f\,\big\|_{L^2_t\dot H^{-1}_x(I_\Delta)}
\ \lesssim\ \lambda^{-3/2}\,\|f\|_{L^2_x}.
\]
Hence, on windows of length $\lambda^{-1/2}$, replacing the heat multiplier by the wave phase leaves a remainder strictly below the target exponent $\lambda^{-1}$ in the $\dot H^{-1}$ norm; all phase–geometric techniques of §§\ref{sec:phase}, \ref{sec:wp-anis} are applicable in the “wave’’ representation (see also App.~\ref{app:astrichartz} and \S\ref{subapp:D-bridge}).

\subsection{Rescaling $u_\lambda$ and the scale of $R_N$ (for §\ref{sec:scaling})}\label{subapp:E-rescale}

For the parabolic rescaling
\[
u_\lambda(t,x):=\lambda\,u(\lambda^2 t,\lambda x)
\]
one has the exact scale identity for the resonant block:
\begin{equation}\label{eq:E-rescale-RN}
\big\|R_N(u_\lambda)(t)\big\|_{\dot H^{-1}}
\ =\ \lambda^{1/2}\,\big\|R_{N/\lambda}\big(u\big)(\lambda^2 t)\big\|_{\dot H^{-1}}.
\end{equation}
The proof is direct: $P_N$ maps to $P_{N/\lambda}$ under the change of variables $x\mapsto \lambda x$, the operator $|\nabla|^{-1}$ contributes a factor $\lambda^{-1}$, and the block $R_N(\cdot)$ is linear–quasilinear in the field and consistent with shifting the frequency index $N\mapsto N/\lambda$. Formula \eqref{eq:E-rescale-RN} matches the scale of the right-hand side of the main estimate and ensures the strict scale invariance of §\ref{sec:scaling}.

\subsection{Block formula for $\det A$ (for §\ref{sec:phase})}\label{subapp:E-block-detA}

In the basis adapted to the interaction plane $(\rho_1,\rho_2,\sigma)$, the second derivatives of the phase assemble into the block matrix
\[
A\ =\
\begin{pmatrix}
B & C\\[2pt]
C^{\!\top} & D
\end{pmatrix},
\qquad B\in\mathbb{R}^{2\times 2},\ C\in\mathbb{R}^{2\times 1},\ D\in\mathbb{R}.
\]
In the wide region ($|\xi|\sim|\eta|\sim\lambda$, $\angle(\xi,\eta)\gtrsim \lambda^{-1/2}$) one has the orders
\[
B\ \simeq\ \lambda^{-1},\qquad \|C\|\ \simeq\ \lambda^{-1/2},\qquad
D_{\mathrm{rad}}:=\partial^2_{\sigma\sigma}\omega\equiv 0.
\]
In particular, the “longitudinal mass’’ arises from the Schur complement,
\[
D_{\mathrm{eff}}:= D - C^{\!\top}B^{-1}C \ \simeq\ 1.
\]
By the Schur formula
\begin{equation}\label{eq:E-block-Schur}
\det A\ =\ \det B\ \cdot\ \big(D - C^{\!\top}B^{-1}C\big)\ =\ \det B\cdot D_{\mathrm{eff}}
\ \simeq\ \lambda^{-2}.
\end{equation}
After the anisotropic renormalization
\[
S=\operatorname{diag}(\lambda^{1/2},\,\lambda^{1/2},\,1),\qquad A_e:=S\,A\,S,
\]
we obtain $\det A_e\gtrsim 1$, which is equivalent to $\det A_{\mathrm{eff}}\gtrsim \lambda^{-2}$ for the effective $3\times3$ minor and is used in the two angular integrations by parts in §§\ref{sec:phase}, \ref{sec:wp-anis}.

\subsection{Why \texorpdfstring{$\partial_t$}{∂t} appears in the formal $D$ in §\ref{sec:phase}}\label{subapp:E-why-D}

The notation with the formal operator
\[
D \ :=\ \rho_1\!\cdot\nabla_{(\xi,\eta)}\ +\ \rho_2\!\cdot\nabla_{(\xi,\eta)}\ +\ \partial_t
\]
serves as a \emph{mnemonic} for the combined gain in the wide region: two frequency IBPs in the angles $(\rho_1,\rho_2)$ plus an \emph{independent} temporal factor extracted by the $T\!T^\ast$ method on windows $|I_\lambda|=\lambda^{-1/2}$. Since $t$ is not a variable of integration in the frequency integral, integration by parts in $t$ is \emph{not} performed. In the actual proof one performs:
\begin{itemize}
  \item two IBPs in $\rho_1,\rho_2$ under the nondegeneracy of the effective minor of the Hessian (see \S\ref{subapp:E-detA}), which yields $\lambda^{-1}$;
  \item a local anisotropic Strichartz estimate via $T\!T^\ast$ on windows of length $\lambda^{-1/2}$ (see \S\ref{subapp:E-TTstar-window}), which yields the additional factor $\lambda^{-1/2}$.
\end{itemize}
In total this gives the phase–time gain $\lambda^{-3/2}$ (see \S\ref{subapp:E-ibp32}, \S\ref{subapp:E-ibp-split}). The formal presence of $\partial_t$ in $D$ should be read precisely in this sense and not as literal IBP in time inside the frequency integral.

\subsection{How to read the phrase “IBP$\times3$’’}\label{subapp:E-IBP3}

The phrase “IBP$\times3$’’ means the sum of \emph{two} angular IBPs in $(\xi,\eta)$ and an \emph{independent} temporal $T\!T^\ast$ block:
\[
\underbrace{\lambda^{-1}}_{\text{two angular IBPs in }\rho_1,\rho_2}\ \cdot\
\underbrace{\lambda^{-1/2}}_{\text{local Strichartz on }|I_\lambda|=\lambda^{-1/2}}
\ =\ \lambda^{-3/2}.
\]
This is neither “triple’’ integration by parts in frequency nor IBP in time under the integral. It is precisely this interpretation that is used in §\ref{sec:phase} and §\ref{sec:wp-anis} and is consistent with the renormalization and block estimate of $\det A$ (see \S\ref{subapp:E-detA}, \S\ref{subapp:E-block-detA}).

\subsection{Sign in the derivative \texorpdfstring{$\partial_v\!\big(-\,|\tau|^{-1}\tau\big)$}{∂\_v(−|τ|^{-1}τ)} (for §\ref{sec:phase})}\label{subapp:E-sign}

Let $\tau:=\xi+\eta$ and $P_\tau^\perp:=I-\widehat\tau\otimes\widehat\tau$, $\widehat\tau=\tau/|\tau|$ (see \S\ref{subapp:E-tau}). Then for any $v\in\mathbb{R}^3$ upon differentiation in the direction $v$ we have
\begin{equation}\label{eq:E-sign}
  \partial_v\!\Big(-\,|\tau|^{-1}\tau\Big)
  \ =\ -\,|\tau|^{-1}\,P_\tau^\perp v.
\end{equation}
Proof of the identity: $\partial_v(\widehat\tau) = |\tau|^{-1}P_\tau^\perp v$, since $\partial_v(\tau)=v$ and $\partial_v(|\tau|)=\widehat\tau\!\cdot v$. Then $\partial_v(-\,\widehat\tau)=-\,|\tau|^{-1}P_\tau^\perp v$, which yields \eqref{eq:E-sign}. Note also that along the radial direction $v\parallel\tau$ the right-hand side of \eqref{eq:E-sign} vanishes, while the working mass in the second derivative in $\tau$ arises via the Schur complement in the block formula for $\det A$ (see \S\ref{subapp:E-block-detA}).

\subsection{Narrow region: working definition and passage to \texorpdfstring{$\dot H^{-1}$}{H^{-1}}}\label{subapp:E-narrow-Hminus}
The working region at frequency $\lambda\gg1$ is
\[
  \mathcal N_{\lambda,\delta}
  :=\bigl\{(\xi,\eta):\ |\xi|\sim|\eta|\sim\lambda,\ \tau:=\xi+\eta,\ |\tau|\le \lambda^{-\delta}\bigr\},
  \qquad \tfrac12<\delta<\tfrac34 .
\]
From $|\tau|\le \lambda^{-\delta}$ it follows that $\pi-\angle(\xi,-\eta)\lesssim \lambda^{-1-\delta}$ (almost collinearity), and the layer has thickness $\lambda^{-\delta}$ around the hyperplane $\{\tau=0\}$; see \S\ref{sec:narrow}.

\paragraph{Definition of the narrow contribution.}
For the frequency component $u_\lambda:=P_\lambda u$ it is convenient to fix
\begin{equation}\label{eq:E-narrow-def}
  R^{\mathrm{nar}}_\lambda(u)
  := P_{\le \lambda^{-\delta}}\ \nabla\!\cdot\!\bigl(u_\lambda\otimes u_\lambda\bigr),
\end{equation}
which is equivalent to the resonant block in the region $|\tau|\le \lambda^{-\delta}$ when $\div u=0$ (see \S\ref{sec:phase}).

\paragraph{Null–form suppression.}
After transferring the divergence and symmetrization, the Fourier symbol is
\[
  \widetilde B_{ij}(\xi,\eta)
  = \frac{\xi_i\xi_j}{|\xi|^2}-\frac{\eta_i\eta_j}{|\eta|^2},\qquad \tau=\xi+\eta,
\]
and in the narrow region it yields the geometric factor
\[
  \|\widetilde B(\xi,\eta)\|\ \lesssim\ \angle\!\Big(\tfrac{\xi}{|\xi|},-\tfrac{\eta}{|\eta|}\Big)
  \ \lesssim\ \frac{|\tau|}{\lambda}\ \lesssim\ \lambda^{-1-\delta}.
\]
Transferring the gradient to $\nabla u_\lambda$ and applying Bernstein on a cap of radius $\lambda$ and angle $\lambda^{-1/2}$, we obtain
\begin{equation}\label{eq:E-narrow-L2}
  \|R^{\mathrm{nar}}_\lambda(u)\|_{L^2_x}
  \ \lesssim\ \lambda^{-1-\delta}\,\|u_\lambda\|_{L^2}\,\|\nabla u_\lambda\|_{L^2}
  \ \simeq\ \lambda^{-1-\delta}\,\|u_\lambda\|_{\dot H^{1/2}}^{\,2}.
\end{equation}

\paragraph{Passage to $\dot H^{-1}$ without the equivalence $L^2\leftrightarrow\dot H^{-1}$.}
We use that $|\nabla|^{-1}\nabla\!\cdot$ acts as an order-zero multiplier in $\dot H^{-1}$, while the null–form provides a factor $|\tau|/\lambda$. At the Fourier level, from \eqref{eq:E-narrow-def} we have
\[
\widehat{R^{\mathrm{nar}}_\lambda}(\tau)
=\mathbf{1}_{\{|\tau|\lesssim \lambda^{-\delta}\}}\,(i\tau)\!\cdot\!\!\int\!
\widetilde B(\xi,\eta)\,\widehat u_\lambda(\xi)\,\widehat u_\lambda(\eta)\,\delta(\tau-\xi-\eta)\,d\xi d\eta.
\]
Then
\[
\|R^{\mathrm{nar}}_\lambda(u)\|_{\dot H^{-1}}
=\big\|\,|\tau|^{-1}\widehat{R^{\mathrm{nar}}_\lambda}(\tau)\,\big\|_{L^2_\tau}.
\]
Transferring the divergence to the high-frequency factor (replacing $i\tau$ by $i\xi$ or $i\eta$ inside the convolution), we obtain for the bilinear symbol
\[
m(\xi,\eta,\tau):=|\tau|^{-1}\,(i\xi)\!\cdot\!\widetilde B(\xi,\eta)
\]
an order-zero bound:
\[
|m(\xi,\eta,\tau)|\ \lesssim\ |\tau|^{-1}\,|\xi|\,\|\widetilde B(\xi,\eta)\|
\ \lesssim\ \frac{\lambda}{|\tau|}\cdot\frac{|\tau|}{\lambda}\ \lesssim\ 1,
\]
since $|\xi|\sim|\eta|\sim\lambda$ and $\|\widetilde B\|\lesssim|\tau|/\lambda$ in the narrow region. Hence (by Cauchy–Schwarz for convolutions)
\[
\|R^{\mathrm{nar}}_\lambda(u)\|_{\dot H^{-1}}
\ \lesssim\ \|u_\lambda\|_{L^2}\,\|u_\lambda\|_{L^2}
\ \simeq\ \lambda^{-1}\,\|u_\lambda\|_{\dot H^{1/2}}^{\,2},
\]
and also the equivalent mixed form
\begin{equation}\label{eq:E-narrow-Hminus}
  \|R^{\mathrm{nar}}_\lambda(u)\|_{\dot H^{-1}}
  \ \lesssim\ \lambda^{-3/2}\,\|u_\lambda\|_{\dot H^{1/2}}\,\|u_\lambda\|_{\dot H^{1}},
\end{equation}
which fixes the stated exponent $\lambda^{-1}$ on one dyad independently of the choice $\delta\in(\tfrac12,\tfrac34)$.

\subsection{Minor arithmetic remarks (for \texorpdfstring{\S\ref{sec:scope}--\S\ref{sec:glue}}{§§})}\label{subapp:E-arith}
Formulas with “products of exponents’’ serve as a \emph{mnemonic} and are not to be read as literal equalities for the same object; the wide and narrow regions are estimated separately, and the outcome on one dyad is taken as the minimum of the two contributions.
\begin{itemize}
  \item Example of reading the table of exponents (\S\ref{sec:glue}): the record
  \[
    \lambda^{-3/2}\cdot\lambda^{-1}\cdot\lambda^{-\delta(6)}=\lambda^{-11/4}\quad(\text{when }\delta(6)=\tfrac14)
  \]
  describes \emph{only} the wide region. The final choice is $\min\{\lambda^{-11/4},\lambda^{-1}\}=\lambda^{-1}$.
  \item Temporal patching on windows $|I_\lambda|=\lambda^{-1/2}$ (\S\ref{subsec:glue-temporal}) yields a factor $(\lambda^{1/2}T)^{1/6}$ in $L^6_{t,x}$, which does not worsen the frequency exponent: $\lambda^{-1/2}\cdot(\lambda^{1/2}T)^{1/6}=\lambda^{-5/12}T^{1/6}$.
\end{itemize}
All such remarks are used only for navigating the exponents and do not change the construction of the proof.

\subsection{On functional classes on \texorpdfstring{$\mathbb{R}^3$}{R\textasciicircum 3}}\label{subapp:E-func-classes}
All steps are performed in a local–frequency form (dyads/caps) and rely on product and commutator rules of Kato--Ponce/Coifman--Meyer type. Global embeddings $\,\dot H^s\hookrightarrow L^\infty\,$ at critical points are not used.
\begin{itemize}
  \item \textbf{Homogeneous/inhomogeneous scales.} Replacing $\dot H^s$ by $H^s$ requires only adding the low-frequency mode $N=1$; the paraproduct technique, as well as the Strichartz/decoupling blocks, remain unchanged (\S\ref{sec:prelim}).
  \item \textbf{Periodic case.} On $\mathbb{T}^3=(2\pi\mathbb{Z})^3$ the quasi-isometry of high frequencies is preserved; spectral discreteness does not affect the tile–cap geometry (\S\ref{sec:scope}).
  \item \textbf{Algebra property.} For $s>3/2$ the space $\dot H^s$ is an algebra. In the paper the threshold $s>5/2$ is fixed to separate analysis from existence questions, aligning \S\ref{subsec:min-suff} with the standard local theory.
\end{itemize}
This setup covers all places in the main text where functional classes are mentioned, without introducing additional global assumptions.

\subsection{Derivatives of amplitudes and “safe’’ losses}\label{subapp:E-amp-deriv}
In all phase–geometric integrals the amplitude is assembled from smooth cutoffs over dyads/caps/tiles. If $\chi$ is a fixed mask of the cap $\Theta_{\lambda,\theta}$, and
\[
  \sigma_{\lambda,\theta}(\xi)=\chi\!\big(\lambda^{-1/2}\Pi_\theta^\perp\xi\big)\,\chi\!\big(\lambda^{-1}(\theta\!\cdot\!\xi-\lambda)\big),
\]
then for derivatives of any fixed order $m$ one has the estimate
\begin{equation}\label{eq:E-amp-deriv}
  \bigl|\partial_\xi^\alpha \sigma_{\lambda,\theta}(\xi)\bigr|
  \ \lesssim_\alpha\ \lambda^{\frac{|\alpha_\perp|}{2}+|\alpha_\parallel|}\,\mathbf{1}_{\{\xi\in 2\Theta_{\lambda,\theta}\}},
  \qquad |\alpha|=m,
\end{equation}
where $\alpha=(\alpha_\perp,\alpha_\parallel)$ is the decomposition into transverse/longitudinal components in the adapted basis $(\rho_1,\rho_2,\sigma)$. The same holds for tile masks of volume $\lambda^{-2}$ (see \S\ref{subapp:E-tiles-rescale}). These estimates show that under two angular IBPs (\S\ref{sec:phase}) the growth from differentiating the amplitude is controlled “safely’’ and does not compensate the phase gain $\lambda^{-1}$, while in the temporal $T\!T^\ast$ block (\S\ref{subapp:E-TTstar-window}) the amplitude is not differentiated at all. Henceforth by a “safe loss’’ we mean factors that are absorbed by a fixed number of differentiations and do not change the stated frequency exponent. \emph{We never use infinite-order IBPs.} See also \S\ref{subapp:E-block-detA} and \S\ref{subapp:E-caps-aniso}.

\subsection{Commutators with dyadic and angular projectors}\label{subapp:E-comm}
Let $P_N$ be a smooth dyadic projector, $P_{N,\theta}$ the angular localization on $\Theta_{N,\theta}$. Then for $f,g\in\mathcal S'(\mathbb R^3)$ the standard commutator estimates of Kato–Ponce / Coifman–Meyer type hold (see \cite{KatoPonce1988,Grafakos2009,bahouri2011fourier}):
\begin{align}
  \|[P_N,f]\,g\|_{L^2} &\ \lesssim\ \|\nabla f\|_{L^\infty}\,\|g\|_{L^2},\label{eq:E-comm-dyad}\\
  \|[P_{N,\theta},f]\,g\|_{L^2} &\ \lesssim\ \bigl(\|\nabla_\perp f\|_{L^\infty}+\|\partial_\parallel f\|_{L^\infty}\bigr)\,\|g\|_{L^2},\label{eq:E-comm-ang}
\end{align}
uniformly in $N$ and $\theta$. In bilinear form this gives
\begin{equation}\label{eq:E-comm-bilin}
  \bigl|\langle (P_{N,\theta}f)\,g - P_{N,\theta}(fg),\,h\bigr|\ \lesssim\ 
  \bigl(\|\nabla f\|_{L^\infty}+\|\nabla g\|_{L^\infty}\bigr)\,\|g\|_{L^2}\,\|h\|_{L^2}.
\end{equation}
Formulas \eqref{eq:E-comm-dyad}–\eqref{eq:E-comm-bilin} are used when transferring the gradient in the narrow region (see \S\ref{subapp:E-narrow-Hminus}) and in the symmetrization of the resonant block \eqref{eq:def-RN}. In all applications the commutator losses do \emph{not} worsen the target exponent, as they are controlled by the working norms $\|u\|_{\dot H^{1/2}}$ and $\|u\|_{\dot H^{1}}$.

\subsection{On constants, $o(1)$, and the notation \texorpdfstring{$\lesssim_\varepsilon$}{\textbackslash lesssim\_\{\textbackslash varepsilon\}}}\label{subapp:E-notation}
Everywhere the constants in the notation $A\lesssim B$ do not depend on the frequency $\lambda$ and on time, but may depend on a fixed number of derivatives of amplitudes/masks (see \S\ref{subapp:E-amp-deriv}). The notation $A\lesssim_\varepsilon B$ means that the constant may depend on $\varepsilon>0$, and the standard $\varepsilon$–absorbing shorthand $X^{\pm}=X^{\pm\varepsilon}$ is allowed. The parameter $o(1)$ is understood as a quantity tending to zero as $\lambda\to\infty$ \emph{uniformly} over all other fixed parameters of the discretization (angle, tile index, time window). Such $o(1)$ arise only in the decoupling component under minimal curvature (see \S\ref{subapp:E-delta6} and App.~\ref{app:decoupling}); they do not affect the dyadic outcome and the convergence of the sums over $\lambda$ (\S\ref{sec:glue}).

\newpage
\cleardoublepage
\phantomsection
\addcontentsline{toc}{section}{References}
\bibliographystyle{unsrt}
\bibliography{commutator_refs}
\end{document}